\documentclass{amsart}

\usepackage{amsmath,amssymb,amsthm,amsrefs,comment,multicol,tikz}

\usetikzlibrary{patterns}

\newtheorem{thm}{Theorem}[section]
\newtheorem{prop}[thm]{Proposition}
\newtheorem{lem}[thm]{Lemma}
\newtheorem{cor}[thm]{Corollary}

\theoremstyle{definition}
\newtheorem{defn}[thm]{Definition}

\theoremstyle{remark}
\newtheorem*{rem}{Remark}

\newcommand{\C}{\mathbb{C}}

\newcommand{\E}{\mathbb{E}}

\newcommand{\flow}{\operatorname{Flow}}
\newcommand{\im}{\operatorname{Im}}
\newcommand{\N}{\mathbb{N}}
\newcommand{\PP}{\mathbb{P}}
\newcommand{\Q}{\mathbb{Q}}
\newcommand{\R}{\mathbb{R}}

\newcommand{\re}{\operatorname{Re}}

\newcommand{\Span}{\operatorname{Span}}

\numberwithin{equation}{section}

\begin{document}

\title[Invariant Gibbs Measure]{Invariant Gibbs Measure for 3D NLW \\ in Infinite Volume}
\author{Samantha Xu}
\begin{abstract}
Consider the radial nonlinear wave equation $-\partial_t^2 u + \Delta u = u^3$, $u :\R_t \times \R_x^3 \to \R$, $u(t,x) = u(t,|x|)$. In this paper, we construct a Gibbs measure for this system and prove its invariance under the flow of the NLW. In particular, we are in the infinite volume setting.

For the finite volume analogue, specifically on the unit ball with zero boundary values, an invariant Gibbs measure was constructed by Burq, Tvetkov, and de Suzzoni in \cite{BT2,Suz} as a Borel measure on super-critical Sobolev spaces.

In this paper, we advocate that the finite volume Gibbs measure be considered on a space of weighted H\"older continuous functions. The measure is supported on this space and the NLW is locally well-posed there, a counter-point to the Sobolev super-criticality noted by Burq and Tzvetkov. Furthermore, the flow of the NLW leaves this measure invariant.

We use a multi-time Feynman--Kac formula to construct the infinite volume limit measure by computing the asymptotics of the fundamental solution of an appropriate parabolic PDE. We use finite speed of propagation and results from descriptive set theory to establish invariance of the infinite volume measure. 

To the best of our knowledge, this paper provides the first construction and proof of invariance of a Gibbs measure in infinite volume outside of the 1D case.
\end{abstract}
\maketitle

\tableofcontents


\section{Introduction}

Consider the 3D radial, defocusing, cubic, nonlinear wave equation (NLW)
	\begin{equation}\label{NLW}
	\left\{\begin{array}{l}
	-\partial_t^2 u + \Delta u = u^3,\\
	u:\R_t \times \R_x^3 \to \R, \\ 
	u(t,x) = u(t,|x|) = u(t,r)
	\end{array}\right.
	\end{equation}
with Hamiltonian 
	$$
	H(u) = \int_{\R^3} \tfrac{1}{4}|u|^4 + \tfrac{1}{2}|\nabla u|^2 + \tfrac{1}{2}|u_t|^2.
	$$
In this paper, we construct the Gibbs measure $m_\infty = m_{\infty,1} \otimes m_{\infty,2}$ for this system, which we informally write as
	\begin{align*}
	dm_{\infty}(u,u_t) &= \frac{1}{Z} \exp\left(-H(u) \right) ``d(u,u_t)"
		\\
	dm_{\infty,1}(u) &= \frac{1}{Z_1} \exp\left(-\int_{\R^3} \tfrac{1}{4}|u|^4 + \tfrac{1}{2}|\nabla u|^2 \right) ``du"
		\\
	dm_{\infty,2}(u_t) &= \frac{1}{Z_2} \exp\left(-\int_{\R^3} \tfrac{1}{2}|u_t|^2\right) ``du_t"
	\end{align*}
and show that the flow of \eqref{NLW} is defined for all time on the support of $m_\infty$, and leaves $m_\infty$ invariant. Namely, the flow is a measure preserving transformation for all time (cf., Theorem~\ref{mainer}). 

To prove this theorem, we introduce a change of variables to reformulate Theorem~\ref{mainer} in a one dimensional setting. In this setting, we break the reformulation into three parts (Theorems~\ref{invariancefinite}, \ref{infvolmeas}, and \ref{invarianceinf}), and prove each part separately. This reformulation is detailed in Section~\ref{reform}.

To better understand the interpretation of such a measure, let us first examine a simpler system. Let $L > 0$ and let $B(0,L) \subseteq \R^3$ denote the ball in $\R^3$ of radius $L$ and with center at $0$. Let $\Delta_L = \Delta_{L,Dir}$ denote the Laplacian on $B(0,L)$ with Dirichlet boundary conditions. Consider the linear wave equation
	\begin{equation}\label{freewave}
	\left\{\begin{array}{l}
	-\partial_t^2 u + \Delta_L u = 0,\\
	u:\R_t \times B(0,L) \to \R, \\ 
	u(t,x) = u(t,|x|) = u(t,r), \\
	u|_{\R_t \times \partial B(0,L)} \equiv 0
	\end{array}\right.
	\end{equation}
as well as its corresponding ``free measure"
	\begin{equation}\label{freegibbs}
	d\tilde{m}_{L}(u,u_t) 
		=
	\frac{1}{Z} \exp\left(-\int_{B(0,L)} \frac{1}{2}|\nabla u|^2 \right) ``du" \otimes \frac{1}{Z'} \exp\left(-\int_{B(0,L)} \frac{1}{2}|u_t|^2\right) ``du_t".
	\end{equation}
Since we can diagonalize $\Delta_L$ in this setting, each of the two factors above is suggestive of a Gaussian measure on some infinite dimensional Hilbert space. 

Indeed, consider the following orthonormal\footnote{The normalized Lebesgue measure on $B(0,L)$ is given by $\frac{1}{4\pi} r^2 dr d\sigma_{S^2}$.} basis for $L^2_{rad}(B(0,L))$ consisting of eigenfunctions of $\Delta_L$:
	\begin{equation}\label{3deigfun}
	f_{n,L}(r) := \frac{ \sqrt{2/L}\sin(n\pi r /L)}{r}, \hbox{\hskip 18pt } n =1,2,\ldots
	\end{equation}
with eigenvalues $(n\pi /L)^2$, $n =1,2,\ldots,$ respectively. For each $s \in \R$, consider the radial, homogeneous Sobolev space
		\begin{align*}
	&\dot{H}^s_{rad,0}(B(0,L) \to \C) :=
		\\
	&\hbox{\hskip 8pt}  
	\left\{ 
		g = \sum_{n=1}^\infty c_n f_{n,L}  	
			\;\middle|\;
		c_n \in \C	, \|g\|_{\dot{H}^s_{rad,0}}^2 := \sum_{n=1}^\infty |n\pi/L|^{2s} |c_n|^2 < \infty
	\right\},
	\end{align*}
and $\dot{H}^s_{rad,0}(B(0,L) \to \R)$, the sub-space with real coefficients. We equip these spaces with the usual Borel $\sigma$-algebra.

Letting $s < \frac{1}{2}$, we interpret the expression in \eqref{freegibbs} as the image measure on 
	\[
	\dot{H}^s_{rad,0} (B(0,L) \to \R) \times \dot{H}^{s-1}_{rad,0}(B(0,L) \to \R)
	\]
under the map
	\begin{equation}\label{randomize3d}
	\omega \mapsto \left( \sum_{n=1}^\infty \frac{a_n(\omega)}{n \pi / L} f_{n,L}(r), \sum_{n=1}^\infty b_n(\omega) f_{n,L}(r) \right) \hbox{\hskip 18pt }a_n(\omega), b_n(\omega) \sim \mathcal{N}_\R(0,1) \hbox{ i.i.d.}
	\end{equation}
where $\omega$ is an element of some suitable probability space $(\Omega, \mathcal{F}, P)$ and $a_n,b_n$ are i.i.d. standard Gaussian random variables. Indeed,	the restriction $s < \frac{1}{2}$ is due to the fact that
	\[
	\E^\omega\left[ \Big\|\sum_{n=1}^\infty \frac{a_n(\omega)}{n \pi / L} f_{n,L}(r)\Big\|_{\dot{H}^s_{rad,0}} + \Big\|\sum_{n=1}^\infty b_n(\omega) f_{n,L}(r)\Big\|_{\dot{H}^{s-1}_{rad,0}} \right] < \infty
	\]
if and only if $s < \frac{1}{2}$ and so the free measure is actually supported on such Sobolev spaces. It can be shown (cf., \cite[Prop. 2.10]{Suz}) that this measure is left invariant under the flow of the linear wave equation in \eqref{freewave}.

Let us now consider a system that is the nonlinear version of \eqref{freewave}, as well as the finite volume version of \eqref{NLW}:
	\begin{equation}\label{secondorderfinite}
	\left\{\begin{array}{l}
	-\partial_t^2 u + \Delta_L u = u^3,\\
	u:\R_t \times B(0,L) \to \R, \\ 
	u(t,x) = u(t,|x|) = u(t,r), \\
	u|_{\R_t \times \partial B(0,L)} \equiv 0
	\end{array}\right.
	\end{equation}
and its Gibbs measure
	\begin{equation}\label{secondorderfinitegibbs}
	dm_L(u,u_t) 
		:= 
	dm_{L,1}(u) \otimes dm_{L,2}(u_t)
	\end{equation}
where
	\begin{align*}
	dm_{L,1}(u)
		&=
	\frac{1}{Z_1} \exp\left(-\int_{B(0,L)} \frac{1}{4}|u|^4 + \frac{1}{2}|\nabla u|^2 \right) ``du"
		\\
	dm_{L,2}(u_t)
		&=
	\frac{1}{Z_2} \exp\left(-\int_{B(0,L)} \frac{1}{2}|u_t|^2\right) ``du_t"
	\end{align*}
Observe that, formally speaking, 
	\begin{equation}\label{3dmeasradnik}
	dm_L(u,u_t) = \frac{1}{Z} \exp\left(-\int_{B(0,L)} \frac{1}{4}|u|^4 \right) d\tilde{m}_{L}(u,u_t),
	\end{equation}
with $Z$ being some normalizing constant. This system was studied\footnote{Burq and Tzvetkov actually considered the Gibbs measure for NLW $-\partial_t^2 u + \partial_L u = |u|^{p} u$ when $p < 3$, though invariance of the Gibbs measure was only shown in the case $p=2$ by de Suzzoni.} by Burq and Tzvetkov in \cite{BT2} and by de Suzzoni in \cite{Suz}. They considered the case $L =1$, but the analysis for general $L >0$ is identical if we simply re-scale. In \cite{BT2}, it was noted that 
	\begin{equation}\label{nonnull}
	\left\| \sum_{n=1}^\infty \frac{a_n(\omega)}{n\pi / L} f_{n,L}(r)  \right\|^4_{L^4(B(0,L))}  < \infty, \hbox{\hskip 18pt} \omega\hbox{ a.s.}
	\end{equation}
In particular, the Radon--Nikodym derivative
	\begin{equation}\label{radnik3d}
	(u,u_t) \mapsto \exp\left(-\frac{1}{4} \int_{B(0,L)} |u|^4\right)
	\end{equation}
is positive $\tilde{m}_L$ almost surely, and so the expression in \eqref{3dmeasradnik} is indeed well-defined. In particular, $m_{L}$ is mutually absolutely continuous with respect to the free measure $\tilde{m}_L$. It was shown by Burq and Tzvetkov that \eqref{secondorderfinite} is globally well-posed on a set of $m_L$ measure $1$. Building upon the results in \cite{BT2}, it was shown by de Suzzoni that the flow of \eqref{secondorderfinite} leaves the Gibbs measure $m_L$ invariant. We discuss a modification of these results in Theorem~\ref{invariancefinite}.
	
To construct the Gibbs measure for our system \eqref{NLW}, we shall consider a suitable infinite volume limit of the Gibbs measures in \eqref{secondorderfinitegibbs}. To this end, we shall employ techniques from the theory of stochastic processes. In particular, we seek to work in a space where evaluation at a point is a well-defined linear functional.
	
We first discuss the construction and interpretation of
	\begin{align*}
	dm_{\infty, 1}(u) := \frac{1}{Z_1} \exp\left(-\int_{\R^3} \frac{1}{4}|u|^4 + \frac{1}{2}|\nabla u|^2 \right) ``du"
	\end{align*}
by further analyzing the support of $m_{L,1}$. Consider the random series 
	\begin{align}
		\label{bblaw}
	f^\omega(r) 
		&= 
	\sum_{n=1}^\infty \frac{a_n(\omega)}{n\pi / L} \sqrt{\frac{2}{L}}\sin(n \pi r /L), 
		\hbox{\hskip 18pt} a_n(\omega) \sim \mathcal{N}_\R(0,1) \hbox{ i.i.d.}
		\\
		\nonumber
		&= 
	r \sum_{n=1}^\infty \frac{a_n(\omega)}{n \pi /L} f_{n,L}(r), 
	\end{align}
By Mercer's theorem (cf., \cite{Kac} or \cite{Loe}), the series almost surely converges uniformly and has law of the standard Brownian bridge from  $(0,0)$ to $(L,0)$ . Namely, $f^\omega(r)$ is a Gaussian process in $r$ of mean zero with
	\begin{equation}\label{brownianbridge}
	\E[ f^\omega(r)f^\omega(r')] = r\left(1 - \frac{r'}{L}\right), \hbox{\hskip 18pt } 0 \leq r \leq r' \leq L.
	\end{equation}
In particular, $f^\omega(r)$ is almost surely $s$-H\"older continuous for every $s \in [0, \frac{1}{2})$. Since $m_{L,1}$ is mutually absolutely continuous with respect to the law of $r^{-1} f^\omega(r)$, it follows that $m_{L,1}$ is supported on the space
	\begin{align*}
	r^{-1}C^s_{rad,0}(B(0,L) \to \R) 
		&:= 
	\{ f : f|_{\partial B(0,L)} \equiv 0, \lim_{r \to 0} rf(r) = 0, \|f\|_{s,L} < \infty\}
	\end{align*}
for every $s \in [0,\frac{1}{2})$, where 
	\begin{equation}\label{weightedholdernorm}
	\|f\|_{s,L} := \sup_{\substack{r_1, r_2 \in [0,L]\\r_1 \neq r_2}} \frac{r_1f(r_1)-r_2f(r_2)}{|r_1-r_2|^s}.
	\end{equation}
Note that the condition $\lim_{r \to 0} rf(r) =0$ comes from the fact that $f^\omega(0) = 0$ almost surely. It follows from the discussion in Section~\ref{reform} and in Appendix~\ref{LS} that $r^{-1}C^s_{rad,0}(B(0,L) \to \R)$ is indeed a measurable subset of $\dot{H}^s_{rad,0}(B(0,L) \to \R)$ and that the restriction of the Borel $\sigma$-algebra on $\dot{H}^s_{rad,0}$ to $r^{-1}C^s_{rad,0}$ coincides with the Borel $\sigma$-algebra generated by \eqref{weightedholdernorm}. We shall construct $m_{\infty,1}$ on the space of locally (weighted-)H\"older continuous functions as a suitable infinite volume limit of $m_{L,1}$. 
\begin{defn}
For $s \in [0,1]$, consider the space
	\[
	r^{-1} C^s_{rad,loc}(\R^3 \to \R) := \{ f : \lim_{r \to 0} rf(r) = 0, \|f\|_{s,L} < \infty \hbox{ for all } L > 0\},
	\]
equipped with the topology and Borel $\sigma$-algebra induced by
	\[
	d(f_1, f_2) = \sum_{n=1}^\infty 2^{-n} \frac{\|f_1 - f_2\|_{s,n}}{1 + \|f_1 - f_2\|_{s,n}}.
	\]	
For $R > 0$, let 
	\[
	r^{-1}C^s_{rad}(B(0,R) \to \R) := \{ f : \lim_{r \to 0} rf(r) =0, \|f\|_{s,R} < \infty\}
	\]
be equipped with Borel $\sigma$-algebra induced by $\|\cdot\|_{s,R}$. For $0 < R < L \leq \infty$, we denote the restriction map $\rho^L_R : r^{-1} C^s_{rad}(B(0,L)) \to r^{-1}C^s_{rad}(B(0,R))$ by
	\[
	\rho^L_R f := f|_{B(0,R)}.
	\]

We define the Borel measure $m_{L,1}|_{B(0,R)}$ on $r^{-1} C^s_{rad}(B(0,R) \to \R)$ via
	\[
	m_{L,1}|_{B(0,R)}(A) := m_{L,1}((\rho^L_R)^{-1}(A))
	\]
for Borel measurable subsets $A \subseteq r^{-1} C^s_{rad}(B(0,R) \to \R)$. We say that a Borel measure $m_{\infty,1}$ on $r^{-1} C^s_{rad,loc}(\R^3 \to \R)$ is \textit{an infinite volume limit of $\{m_{L,1}\}_{L > 0}$} if, for every Borel measurable $A \subseteq r^{-1} C^s_{rad}(B(0,R) \to \R)$, 
	\begin{equation}\label{compatible}
	\lim_{L \to \infty} m_{L,1}|_{B(0,R)}(A) = m_{\infty,1}|_{B(0,R)}(A),
	\end{equation}
where the latter expression is $m_{\infty,1}|_{B(0,R)}(A) := m_{\infty,1}((\rho^\infty_R)^{-1}(A))$.
\end{defn}

We now discuss the construction and interpretation of 
	\[
	dm_{\infty,2} = \frac{1}{Z_2} \exp\left(-\int_{\R^3} \frac{1}{2} |u_t|^2 \right) ``du_t",
	\]
also known as (radial) white noise on $\R^3$. This is a classical object and may be constructed via abstract methods such as B\^ochner--Minlos theory (cf., \cite{Sim}). However, we shall take an alternate, direct approach that is more suited towards proving invariance. To this end, we analyze the support of $m_{L,2}$. 

Now, the random series $\sum_{n=1}^\infty b_n(\omega) \sqrt{\frac{2}{L}} \sin(n \pi r / L)$, $b_n(\omega) \sim \mathcal{N}_\R(0,1)$ i.i.d., almost surely does not converge as a function, and so we cannot expect $m_{L,2}$ to be supported on a space of functions. However, the image measure $((-\Delta_L)^{-1/2})^* m_{L,2}$ on $\dot{H}^s_{rad,0}(B(0,L) \to \R)$, $s < \frac{1}{2}$, has the law
	\[
	[(-\Delta_L)^{-1/2})^* m_{L,2}](A) = P\Big(\sum_{n=1}^\infty \frac{b_n(\omega)}{n\pi/L} f_{n,L}(r) \in A \Big),
	\] 
for Borel $A \subset \dot{H}^s_{rad,0}(B(0,L) \to \R)$. Again, this is the law of Brownian bridge divided by $r$. Thus, $((-\Delta_L)^{-1/2})^* m_{L,2}$ is supported on $r^{-1} C^s_{rad,0}(B(0,L) \to \R)$ for $0 \leq s < 1/2$. As the infinite volume limit of Brownian bridge is Brownian motion (we will make this notion precise in Theorem~\ref{construct} and Section~\ref{brownian}), we make the following definition. 

\begin{defn}
Let $\{B(r)\}_{r \geq 0}$ denote a version of Brownian motion, i.e., $B(r)$ is a Gaussian process of mean zero and $\E[X(r)X(r')] = \min(r,r')$. Fixing $s \in [0,\frac{1}{2})$, let $\mathcal{F}$ denote the restriction of the Borel $\sigma$-algebra on $r^{-1} C^s_{rad,loc}(\R^3 \to \R)$ to the set\footnote{This set consists of tempered distributions and, by the Law of the Iterated Logarithm, is in the support of $r^{-1}B(r)$.} 	
	\[
	A = \{f \in r^{-1} C^s_{rad,loc}(\R^3 \to \R) \mid \lim_{r\to\infty} f(r) = 0 \}.
	\]
We define $m_{\infty,2}$ to be a measure on $(-\Delta)^{1/2}A$, equipped with $\sigma$-algebra $(-\Delta)^{1/2}\mathcal{F}$, such that 
	\begin{equation}\label{whitenoise}
	[((-\Delta)^{-1/2})^* m_{\infty,2}](A) = P(r^{-1} B(r) \in A)
	\end{equation}
for all $A \in \mathcal{F}$.
\end{defn}

We now state our main result.

\begin{thm}\label{mainer}
Fix $s \in [0,\frac{1}{2})$. There exists a unique Borel measure $m_{\infty,1}$ on $r^{-1} C^s_{rad,loc}(\R^3 \to \R)$ which is the infinite volume limit of $\{m_{L,1}\}_{L > 0}$. 

With $m_{\infty,2}$ defined in \eqref{whitenoise}, let
	\[
	m_{\infty} := m_{\infty,1} \otimes m_{\infty,2}.
	\]
For $\frac{1}{3} < s < \frac{1}{2}$, the flow of \eqref{NLW} is globally defined on a set of $m_\infty$-measure $1$, and the restriction of the flow to this set is a measure preserving transformation for all time.
\end{thm}
\begin{rem}
As mentioned earlier, we shall prove this theorem by first reformulating the problem in similar setting. In particular, we complexify \eqref{NLW} as well as reduce to a wave equation on the half-line. We construct an infinite volume limit measure on a space of functions on the half-line, and establish the invariance of this measure under the flow on the reduced, complexified NLW. We detail this reduction in Section~\ref{reform}.
\end{rem}
	
	\subsection*{Acknowledgements}
	
	I would like to thank my PhD advisor, Rowan Killip, for suggesting this line of investigation and for many helpful discussions. I would also like to thank Terence Tao and Erik Walsberg for helpful comments. This work was supported in part by NSF grants DMS 1265868 (P.I. Rowan Killip) and DMS 1001531 (P.I. Rowan Killip).
	
	\subsection{General Background}
A foundational result is Liouville's theorem, which states that the flow of a Hamiltonian ODE preserves Lebesgue measure. Invariant measures for Hamiltonian PDEs were first considered by Lebowitz, Rose, and Speer in \cite{LRS}, and refined by Bourgain in \cite{Bou}. In these papers, they considered a focusing, non-linear Schr\"odinger (NLS) equation on the circle and constructed an $L^2$-truncated Gibbs measure. By applying a frequency truncation (to obtain a finite dimensional system), invoking Liouville's theorem for Hamiltonian ODEs, and using uniform probabilistic estimates to remove the truncations, Bourgain proved global existence of solutions on a set of full measure and the invariance of the Gibbs measure under the NLS. Prior to Bourgain's result, only local well-posedness results were available in that setting. 

One benefit of randomization is that one may work in systems with super-critical scaling. Data with ill-behaved solutions generally lie in null sets of these measures, and the invariance of the Gibbs measure can be used as a conservation law to upgrade local in time existence to global existence. Indeed, the critical scaling for \eqref{secondorderfinite} is $s_c = \frac{1}{2}$, and so $m_L$ is supported on super-critical Sobolev spaces.

Furthermore, we may construct the Gibbs measure for $-\partial_t^2 u + \Delta_L u = |u|^p u$ for all $p < 4$. In \cite{BT2}, Burq and Tzvetkov show almost surely global existence on a set of full Gibbs measure in the case $p < 3$. The case $3 \leq p < 4$ was proven by Bourgain and Bulut in \cite{BB3}.

Another example is \cite{NPS}, where Nahmod, Pavlovi\'c, and Staffilani considered the 2D and 3D Navier--Stokes equations with randomized initial data in super-critical spaces. This is not in a Hamiltonian setting, and hence one cannot expect invariant measures via Liouville's theorem on the Fourier truncations. Nevertheless, they randomized about a fixed initial datum, and applied large deviation estimates, to obtain almost sure global existence of weak solutions, with uniqueness in the 2D case.

Another recent work is \cite{NaSt}, where Nahmod and Staffilani proved almost sure local well-posedness for a 3D quintic NLS on $\mathbb{T}^3$ with data below $H^1(\mathbb{T}^3)$, also in a supercritical regime. Other works on finite volume spatial domains include \cite{Bou2D}, where Bourgain used Wick ordering to construct an invariant Gibbs measure for a 2D NLS on the torus; \cite{Ohwhite}, where Oh proved invariance of mean 0 white noise for the 1D KdV equation on the circle; \cite{CoOh}, where Colliander and Oh showed almost sure global existence of solutions of 1D cubic NLS with initial datum in $H^s(\mathbb{T})$, $-\frac{1}{12} < s < 0$. Further works in the finite volume setting include \cite{BGross, BB1, BB2, NORS, NRSS, Ohcoup, Rich, SoS, ThTz, Tzve, Tzono}, etc, and references therein.

When the underlying spatial domain is of infinite volume, then the construction of such an invariant Gibbs measure is significantly more delicate. In contrast with the finite volume case, Gibbs measures in this setting are singular with respect to an analogous infinite volume ``free measure''. 

The only other case in which an infinite volume invariant measure is constructed is in \cite{McV}, where Mckean and Vaninsky constructed a Gibbs measure for a 1D NLW on the half-line. Using techniques from stochastic analysis, they reduced the construction of such a measure to computing the asymptotics of the fundamental solution of a parabolic PDE with time-independent coefficients. Indeed, the measure was realized as a stationary diffusion on the half-line. They also constructed invariant measures for analogous finite volume NLW and used finite speed of propagation to upgrade to invariance of the infinite volume measure. 

In contrast to \cite{McV}, our measure does not correspond to a stationary diffusion. The fact that we are not in a strictly one dimensional setting means that our parabolic PDE has time dependent coefficients with singularities as we approach the space-time origin (cf., \eqref{parabpde}). 

The fact that our finite volume systems are posed with zero boundary values also makes the measure theory more delicate: generic paths in the support of the finite volume measures are ignored by the infinite volume measure. This is in contrast to \cite{McV}, where the finite volume systems were considered with periodic boundary values.

Other work in infinite volume settings include: in \cite{Bou2}, where Bourgain analyzed a 1D periodic NLS with uniform estimates on arbitrarily large intervals; in \cite{Thom}, where Thomann randomized coefficients of eigenfunctions of a Schr\"odinger operator with a confining potential, and showed almost sure global existence of solutions of power-type NLS on $\R^d$; in \cite{LM}, where L\"uhrmann and Mendelson fixed an initial datum and randomized with respect to its Littlewood--Paley pieces. Other works in the infinite volume setting include \cite{Rid} and \cite{Suz2}. Note that these works do not consider Gibbs measures. 
	
To the best of our knowledge, this paper provides the first construction and proof of invariance of a Gibbs measure in infinite volume outside of the 1D case.

\subsection{Reformulation of Theorem~\ref{mainer}}\label{reform}

As mentioned earlier, we shall make some reductions to our NLW as well as revisit the Burq--Tzvetkov--de Suzzoni result. In this reduced setting, the analogue of Theorem~\ref{mainer} is broken up into three parts---Theorems~\ref{invariancefinite}, \ref{infvolmeas}, and  \ref{invarianceinf}---which we shall prove in Sections~\ref{finitevol}, \ref{bigconstruct}, and \ref{invariance}, respectively. 

We first consider the finite volume setting. Recall, the three dimensional radial Laplacian can be written as $\Delta = \partial_r^2 + \frac{2}{r} \partial_r$. So, $u(t,r)$ is a (classical) solution of \eqref{secondorderfinite} if and only if $v(t,r) := ru(t,r)$ is a (classical) solution of 
	\begin{equation}\label{secondorderfinite1d}
	\left\{\begin{array}{l}
	-\partial_{t}^2 v + \partial_{r}^2 v = \frac{v^3}{r^2}, 
		\\
	v : \R_t \times [0,L] \to \R,
		\\
	v(t,0)=v(t,L) = 0.
	\end{array}\right.
	\end{equation}
Here, $\partial_r^2$ denotes Laplacian with zero boundary values, which is easily understood through Fourier expansion (see below).

Let us also complexify \eqref{secondorderfinite1d} and reduce it to a first order equation. Letting $|\partial_r| := \sqrt{-\partial_r^2}$, then $\phi$ is a solution of \eqref{secondorderfinite1d} if and only if 
	\begin{equation}
	w := v + i |\partial_r|^{-1} \partial_t v = ru + i|\partial_r|^{-1} \partial_t (ru)
	\end{equation}
is a solution (see definition below) of
	\begin{equation}\label{firstorderfinite1d}
		\left\{
		\begin{array}{l}
	-i \partial_t w + |\partial_r| w = - |\partial_r|^{-1} \left(\frac{(\re w)^3}{r^2} \right), 
		\\
	w: \R_t \times [0,L] \to \C,
		\\
	w(t,0) = w(t,L) = 0.
		\end{array}
		\right.
	\end{equation}
We also have the following change of variables on the initial data:
	\begin{equation}\label{initdatachange}
	(u,u_t)|_{t=0} = (f_1, f_2) \implies w|_{t=0} = rf_1 + i|\partial_r|^{-1} (rf_2).
	\end{equation}
We precisely define what we mean by a solution.

\begin{defn}[Strong Solution for \eqref{firstorderfinite1d}]\label{strongfinite}
Let $B = \dot{H}^s_0([0,L])$ or $B = C^s_0([0,L])$. Let $T \in [0,\infty)$. We say that $w(t,r) : [-T,T] \times [0,L] \to \C$ is a \textit{strong solution of \eqref{firstorderfinite1d} on $[-T,T]$ with initial datum $g \in B$} if
	\begin{enumerate}
	\item $w \in C^0_{t} B([-T,T]\times [0,L] \to \C)$, which is to say: for fixed $t \in [-T,T]$ we have $w(t,r) \in B$ and that the $B$-norm of $w$ varies continuously in time.
	\item $w(0,r) = g(r)$, and
	\item $w(t,r)$ obeys the corresponding Duhamel formula
		\[
		w(t,r) = \left[e^{-it|\partial_r|}g\right](r) -i\int_0^t \left[\frac{e^{-i(t-\tau)|\partial_r|}}{|\partial_r|} \re\left(\frac{w(\tau,\cdot)^3}{(\cdot)^2}\right)\right](r) d\tau,
		\]
	for each $t \in [-T,T]$.
	\end{enumerate}
Furthermore, if $w$ is the unique strong solution on $[-T,T]$ with initial datum $g \in B$, then we write 
	\[
	\flow_L(t,g)(r) := w(t,r), \hbox{\hskip 18pt }|t| \leq T.
	\]
\end{defn}

We next turn to the relevant function spaces and their measurable structures. Consider the following orthonormal basis of $L^2_r([0,L])$ consisting of eigenfunction of $\partial_r^2$:
	\[
	e_{n,L}(r) := \sqrt{2/L} \sin(n\pi r /L), \hbox{\hskip 18pt }n=1,2,\ldots,
	\]
with eigenvalues $(n\pi / L)^2$, $n = 1,2,\ldots$, respectively. For each $s \in \R$, consider the homogeneous Sobolev space
	\[
	\dot{H}^s_0([0,L] \to \C) := \left\{ g = \sum_{n=1}^\infty c_n e_{n,L} \mid c_n \in \R, \|g\|_{H^s_0}^2 := \sum_{n=1}^\infty (n\pi/L)^{2s} |c_n|^2 < \infty  \right\}.
	\]
and $\dot{H}^s_0([0,L] \to \R)$, the sub-space with real coefficients. We equip these spaces with the usual Borel $\sigma$-algebra. 

Recalling \eqref{3deigfun}, observe that $e_{n,L} = rf_{n,L}$. In general, $f(r) \mapsto r f(r)$ is (up to a constant multiple) an isometry from $L^2_{rad}(B(0,L)) \to L^2_r([0,L])$, because
	\[
	\int_{S^2}\int_0^L |f(r)|^2 \; r^2 dr \frac{dS}{4\pi} = \int_0^L |rf(r)|^2 \; dr.
	\]
Furthermore, multiplication by $r$ is an isometry $\dot{H}^s_{rad,0}(B(0,L)) \to \dot{H}^s_0([0,L])$.

The analogue of the randomization in \eqref{randomize3d} under the change of variables $(f_1,f_2) \mapsto r f_1 + i|\partial_r|^{-1}(rf_2)$ (cf., \eqref{initdatachange}) is
	\begin{equation}\label{randomize}
	\omega \mapsto \sum_{n=1}^\infty \frac{a_n(\omega) + ib_n(\omega)}{n \pi /L} e_{n,L}(r), \hbox{\hskip 18pt} a_n(\omega), b_n(\omega) \sim \mathcal{N}_\R(0,1), \hbox{ i.i.d.}
	\end{equation}
where, again, $\omega$ is an element of some suitable probability space $(\Omega, \mathcal{F}, P)$. Fixing $s < \frac{1}{2}$, let us denote $\mu_{L}$ to be the image measure on $\dot{H}^s_0([0,L] \to \C)$ under the map in \eqref{randomize}. To separate the randomizations, let $\mu_{L,1}$ and $\mu_{L,2}$ denote the image measure on $\dot{H}^s_0([0,L]\to \R)$ under the maps
	\begin{equation}\label{seprand}
	\omega \mapsto \sum_{n=1}^\infty \frac{a_n(\omega)}{n\pi /L} e_{n,L}(r) \hbox{\hskip 18pt and \hskip 18pt} \omega \mapsto \sum_{n=1}^\infty \frac{b_n(\omega)}{n\pi /L} e_{n,L}(r),
	\end{equation}
respectively. The connection between $\mu_{L}, \mu_{L,1},$ and $\mu_{L,2}$ is as follows: for Borel sets $A_1, A_2 \subseteq \dot{H}^s_0([0,L] \to \R)$, 
	\begin{equation}
	\mu_{L}\big( \{ g \mid \re(g) \in A_1, \im(g) \in A_2 \} \big) = \mu_{L,1}(A_1) \mu_{L,2}(A_2).
	\end{equation}
As with the original case in \eqref{freewave} and \eqref{freegibbs}, $\mu_L$ is invariant under the flow of the linear, first order wave equation 
	\[
		\left\{
		\begin{array}{l}
	-i \partial_t w + |\partial_r| w = 0, 
		\\
	w: \R_t \times [0,L] \to \C,
		\\
	w(t,0) = w(t,L) = 0.
		\end{array}
		\right.
	\]

Let us construct the analogue of \eqref{3dmeasradnik} in this setting. Recall that the Radon--Nikodym derivative in \eqref{radnik3d} only depends on the first component and that, for Borel measurable functions $\Psi: \dot{H}^s_0([0,L]\to \R) \to \R$, we have 
	\[
	\int_{\dot{H}^s_0([0,L])} \Psi(f) \; d\mu_{L,1}(f) = \int_{\dot{H}^s_{rad,0}(B(0,L))}  \Psi(rf) \; d\tilde{m}_{L,1}(f).
	\]
In other words, letting 
	\[
	\Psi(f) = \exp\left(-\frac{1}{4}\int_0^L |f(r)|^4r^{-2} \; dr\right)
	\]
then $\Psi(rf) = \exp\left(-\frac{1}{4}\int_0^L |f(r)|^4r^{2} \; dr\right) = \exp\left(-\frac{1}{4}\int_{B(0,L)} |f(r)|^4\right)$, which recovers \eqref{radnik3d}. Thus, under the change of variables \eqref{initdatachange}, the corresponding Gibbs measure for \eqref{firstorderfinite1d} is
	\begin{equation}\label{1dgibbs}
	d\nu_L(g) := \frac{1}{Z_L}\exp\left(-\frac{1}{4}\int_0^L (\re g(r))^4 r^{-2} \; dr\right) d\mu_L(g),
	\end{equation}
with $Z_L$ being a normalization constant. As with the case of $\tilde{m}_L$ and $m_L$, the measure $\mu_L$ and $\nu_L$ are mutually absolutely continuous. Let us separate $\nu_L$ into its ``real'' and ``imaginary'' components. Namely, let
	\begin{equation}\label{1dgibbsreal}
	d\nu_{L,1}(f) := \frac{1}{Z_L} \exp\left(-\frac{1}{4}\int_0^L |f(r)|^4r^{-2} \; dr\right) d\mu_{L,1} (f) 
	\end{equation}
and let 
	\begin{equation}\label{1dgibbsimag}
	\nu_{L,2} = \mu_{L,2}.
	\end{equation} 
Then the connection between $\nu_L,\nu_{L,1},\nu_{L,2}$ is as follows: for Borel measurable $A_1,A_2 \subseteq \dot{H}^s_0([0,L] \to \C)$,
	\begin{equation}\label{finitevolmeas2}
	\nu_{L}(\{g \mid \re(g) \in A_1 , \im(g) \in A_2\}) = \nu_{L,1}(A_1)\nu_{L,2}(A_2).
	\end{equation}
To construct the infinite volume measure, we will consider the infinite volume limits of $\nu_{L,1}$ and $\nu_{L,2}$ separately, and piece them back together in a way similar to \eqref{finitevolmeas2}. See Section~\ref{bigconstruct} for further details.

Using techniques similar to those by Bourgain, the following result was shown by Burq, Tzvetkov, and de Suzzoni. 
\begin{thm}[Burq--Tvetkov--de Suzzoni\footnote{Strictly speaking, the results in \cite{BT2,Suz} were stated in the 3-dimensional, complexified setting. Theorem~\ref{BTS} is the restatement of their result after multiplication by $r$. }, \cite{BT2,Suz}]\label{BTS}
Let $s < \frac{1}{2}$, let $L > 0$, and let $\nu_L$ be as in \eqref{1dgibbs}. There exists a measurable set $\Pi_L \subseteq \dot{H}^s_{0}([0,L] \to \C)$ with the following properties
	\begin{enumerate}
	\item $\nu_L(\Pi_L) = 1$.
	\item Each $g \in \Pi_L$ admits a unique, global, strong solution in $\dot{H}^s_0$. Namely, $\flow_L(t,g)$ is defined for all $t \in \R$, and, for each $T \in [0,\infty)$, we have 
		\[
		\flow_L(t,g) \in C_t^0\dot{H}^s_{0}([-T,T] \times [0,L] \to \C).
		\]
	\item For each measurable set $A \subseteq \Pi_L$ and for each $t \in \R$, the set $\flow_L(t,A) = \{ \flow_L(t,g) \mid g \in A\}$ is a measurable subset of $\dot{H}^s_{0}([0,L]\to\C)$ and 
		\[
		\nu_L(\flow_L(t,A)) = \nu_L(A).
		\]
	\end{enumerate}
\end{thm}

In Section~\ref{finitevol}, we modify the Burq--Tzvetkov--de Suzzoni result. As mentioned above in \eqref{bblaw}, the randomization in \eqref{seprand} obeys the law of the standard Brownian bridge from $(0,0)$ to $(L,0)$. Thus, for $0 \leq s < \frac{1}{2}$, $\mu_{L,1}$ and $\mu_{L,2}$ are supported on 
	\[
	C^{s}_0([0,L]\to \R):= \{f : [0,L] \to \R \mid f \hbox{ is }s\hbox{-H\"older continuous and } 0=f(0)=f(L)\},	
	\]
equipped with the usual H\"older norm. It follows that $\mu_L$ is supported on $C^{s}_0([0,L]\to \C).$ By mutual absolute continuity, $\nu_L$ is also supported on $C^s_0([0,L] \to \C)$. 

We will show that \eqref{firstorderfinite1d} is locally well-posed in $C^s_0([0,L]\to \C)$, thus making it an excellent space for the analysis of this random initial data problem. This also provides a counter-point to the Sobolev super-criticality noted by Burq and Tzvetkov.

Also, we slightly extend the measure theory in the following sense: we complete the Borel $\sigma$-algebra on $C^s_0([0,L]\to\C)$ with respect to the measure $\nu_L$, and we call a set \textit{$\nu_L$-measurable} if it is an element of this larger $\sigma$-algebra. By abuse of notation, we also denote the extension of the measure to this larger $\sigma$-algebra by $\nu_L$. As we shall see in Section~\ref{invariance}, this setting is convenient because for every Borel set $A \subseteq C^s_0([0,L]\to\C)$ and for every $0 < R < L$, the set 
	\[
	\tilde{A} = \left\{f:[0,L] \to C \mid \exists g\in A \hbox{ such that } g|_{[0,R]} \equiv f|_{[0,R]} \right\}
	\]
is not necessarily Borel, but is still $\nu_L$-measurable. Indeed, we shall see that $\tilde{A}$ is analytic (i.e., a continuous image of a Borel set). Sets of this form arise naturally when we seek to apply finite speed of propagation arguments.

The proof of the following modified result, as well as the discussion of the proof method, is the content of Section~\ref{finitevol}.

\begin{thm}\label{invariancefinite}
Let $\frac{1}{3} < s < \frac{1}{2}$, let $L > 0$. Then \eqref{firstorderfinite1d} is locally well-posed in $C^s_0([0,L]\to \C)$. Let $\nu_L$ be as in \eqref{1dgibbs}. Then, there exists a Borel measurable set $\Omega_L \subseteq C^{s}_0([0,L]\to \C)$ such that
	\begin{enumerate}
	\item $\nu_L(\Omega_L) = 1$.
	\item Each $g \in \Omega_L$ admits a unique global, strong solution in $C^s_0$. Namely, $\flow_L(t,g)$ is defined for all $t \in \R$, and, for each $T \in (0,\infty)$, we have 
		\[
		\flow_L(t,g) \in C_t^0C_r^s([-T,T] \times [0,L] \to \C).
		\] 
	\item For each $\nu_L$-measurable set $A \subseteq \Omega_L$ and for each $t \in \R$, the set 
	\[
	\flow_L(t,A) := \{ \flow_L(t,g) \mid g \in A\}
	\] 
	is also $\nu_L$-measurable subset of $C^{s}_0([0,L]\to \C)$ and 
		\[
		\nu_L(\flow_L(t,A)) = \nu_L(A).
		\]
	Moreover, if $A$ is Borel, then so is $\flow_L(t,A)$.
	\end{enumerate}
\end{thm}
\begin{rem}
Observe that the scaling $w \mapsto w_\lambda(t,r) := w(\lambda t, \lambda r)$ preserves solutions of \eqref{firstorderfinite1d}. Thus, scaling-invariant space is $C^{0}_r$. It follows that $C^s_r$, with $\frac{1}{3} < s < \frac{1}{2}$, are sub-critical spaces with respect to this scaling.
\end{rem}

We turn to the infinite volume setting. Given that the finite volume measures $\nu_L$ are supported on $C^s_0([0,L] \to \C)$, we shall construct the infinite volume limit measure on the space
	\[
	C^s_{loc}([0,\infty) \to \C) := \{ f : [0,\infty) \to \C \mid \|f\|_{C^s([0,L])} < \infty \hbox{ for all } L > 0\},
	\]
where we equip this space with the metric
	\begin{equation}\label{metricinf}
	d(g_1,g_2) = \sum_{n=1}^\infty 2^{-n} \frac{\|g_1 - g_2\|_{C^s([0,n])}}{1+\|g_1 - g_2\|_{C^s([0,n])}}
	\end{equation}
as well as the induced metric topology and Borel structure. Also, we expect a compatibility criterion similar to that of \eqref{compatible}. Given the measure in this 1D setting, we may recover the measure in the original, 3D setting by taking the image measure via the map 	
	\[
	g \mapsto \left(r^{-1} \re(g) , r^{-1} |\partial_r|\im(g) \right).
	\]
The proof of the following result, as well as the discussion of the proof method, is the content of Section~\ref{bigconstruct}.

\begin{thm}\label{infvolmeas}
Fix $0 \leq s < \frac{1}{2}$. There exists a unique Borel probablity measure $\nu_\infty$ on $C^s_{loc}([0,\infty) \to \C)$ such that for each $R > 0$ and for each Borel measurable $A \subseteq C^s([0,R] \to \C)$,
		\begin{equation}\label{compatible1d}
		\lim_{L \to \infty} \nu_{L}|_{[0,R]}(A) = \nu_{\infty}|_{[0,R]}(A).
		\end{equation}
As before, $\nu_{L}|_{[0,R]}$ and $\nu_{\infty}|_{[0,R]}$ denote the (Borel, probability) image measures on $C^0([0,R] \to \C)$ given by the image, under the restriction map $g \mapsto g|_{[0,R]}$, of $\nu_L$ and $\nu_\infty$, respectively.

Moreover, for each $L > 0$, the measures $\nu_L|_{[0,R]}$ and $\nu_\infty|_{[0,R]}$ are mutually absolutely continuous. Let $\mathcal{F}_R$ denote the completion of the Borel $\sigma$-algebra on $C^s([0,R])$ with respect to any of these measures. Then \eqref{compatible1d} holds for every $A \in \mathcal{F}_R$. 
\end{thm}

Finally, we turn to the PDE. Similar to before, the change of variables
	\[
	u \mapsto w(t,r):=ru(t,r)+ir \left(|\partial_r|^{-1} \partial_t u\right)(t,r)
	\] 
is a bijective correspondence between solutions $u$ of \eqref{NLW} and solutions $w$ of 
	\begin{equation}\label{1DNLW}
	\left\{\begin{array}{l}
	-i\partial_t w + |\partial_r| w = -|\partial_r|^{-1} \left( \frac{(\re w)^3}{r^2} \right) \\
	w(t,r): \R_t \times [0,\infty) \to \C\\
	w(t,0) = 0
	\end{array}\right.
	\end{equation}
In this setting, $|\partial_r| := \sqrt{-\partial_r^2}$, where the Laplacian is defined with zero boundary values on $[0,\infty)$.

\begin{defn}[Strong Solution for \eqref{1DNLW}]\label{stronginf}
Let $T \in [0,\infty)$ and let $g \in C^s_{loc}([0,\infty) \to \C)$. We say that $w(t,r) : [-T,T] \times [0,\infty) \to \C$ is a \textit{strong solution of \eqref{firstorderfinite1d} on $[-T,T]$ with initial datum $g$} if
	\begin{enumerate}
	\item $w(0,r) = g(r)$,
	\item $w(t,r)$ obeys the corresponding Duhamel formula
		\[
		w(t,r) = \left[e^{-it|\partial_r|}g\right](r) -i\int_0^t \left[\frac{e^{-i(t-\tau)|\partial_r|}}{|\partial_r|} \re\left(\frac{w(\tau,\cdot)^3}{(\cdot)^2}\right)\right](r) d\tau,
		\]
	for each $t \in [-T,T]$, and
	\item For each $R > 0$, we have $w \in C^0_{t} C^s_r([-T,T]\times [0,R] \to \C)$.
	\end{enumerate}
Furthermore, if $w$ is the unique strong solution on $[-T,T]$ with initial datum $g$, then we write 
	\begin{equation}\label{flowinf}
	\flow_\infty(t,g)(r) := w(t,r), \hbox{\hskip 18pt }|t| \leq T.
	\end{equation}
\end{defn}

As with Theorem~\ref{invariancefinite}, we complete the Borel $\sigma$-algebra on $C^s_{loc}([0,\infty)\to\C)$ with respect to the measure $\nu_\infty$, and we call a set \textit{$\nu_\infty$-measurable} if it is an element of this larger $\sigma$-algebra. By abuse of notation, we also denote the extension of the measure to this larger $\sigma$-algebra by $\nu_\infty$.

The proof of the following result, as well as the discussion of the proof method, is the content of Section \ref{invariance}.

\begin{thm}\label{invarianceinf}
Fix $\frac{1}{3} < s < \frac{1}{2}$. There exists a Borel measurable set $\Omega_\infty \subseteq C^s_{loc}([0,\infty) \to \C)$ such that
	\begin{enumerate}
	\item $\nu_\infty(\Omega_\infty) = 1$;
	\item Each $g \in \Omega_\infty$ admits a unique global, strong solution: $\flow_\infty(t,g)$ is defined for all $t \in \R$ and, for each $T>0$ and $R>0$, we have 
		\[
		\flow_\infty(t,g)|_{[0,R]} \in C_t^0C^s_r([-T,T] \times [0,R] \to \R).
		\]
	\item For each $\nu_\infty$-measurable subset $A \subseteq \Omega_\infty$ and for each $t \in \R$, the set 
	\[
	\flow_\infty(t,A) := \{ \flow_\infty(t,g) \mid g \in A\}
	\]
is also a $\nu_\infty$-measurable subset of $C^s_{loc}([0,\infty) \to \C)$ and
		\[
		\nu_\infty(\flow_\infty(t,A)) = \nu_\infty(A).
		\] 
	Moreover, if $A$ is Borel, then so is $\flow_\infty(t,A)$.
	\end{enumerate}
\end{thm}

\section{Finite volume invariant measures}\label{finitevol}
	
In this section, we prove Theorem~\ref{invariancefinite}. The main new ingredients are a local well-posedness theory on $C^s_0([0,L] \to \C)$, as well as some measure theoretic considerations from Appendix~\ref{LS}. 


\subsection{Local Well-Posedness in $C^s$}

We first establish the local well-posedness of \eqref{secondorderfinite1d} in $C^s_0([0,L]\to\R)$, as the free propagator in this setting can be written down explicitly. Afterwards, we show that the complexified wave equation \eqref{firstorderfinite1d} is locally well-posed in $C^s_0([0,L] \to \C)$.
 
To obtain explicit formulas for the linear evolution on $[0,L]$ with Dirichlet boundary values, we apply the usual odd reflections, and use d'Alembert's formula. For $L \geq 2$, for $0 \leq t \leq 1$, and for $0 < r < L$, we have	
	\begin{equation}\label{costnabla}
	\left[\cos(t|\partial_r|) f \right](r) 
		= 
	\left
		\{\begin{array}{ll}
	\frac{1}{2}(f(r+t) - f(t-r)) & r-t < 0, \\ \\
	\frac{1}{2}(f(r+t) + f(r-t)) & 0 \leq r - t \leq r+t \leq L, \\ \\
	\frac{1}{2} \left(f(r-t)-f(2L-r - t) \right) & L < r+t
		\end{array}
	\right.
	\end{equation}
and
	\begin{equation}\label{sintnabla}
	\left[\frac{\sin(t|\partial_r|)}{|\partial_r|} g \right](r) 
		= 
	\left
		\{\begin{array}{ll}
	\displaystyle \frac{1}{2} \int_{t-r}^{t+r} g(\rho) \; d\rho  & r-t < 0, \\ \\
	\displaystyle \frac{1}{2} \int_{r-t}^{r+t} g(\rho) \; d\rho  & 0 \leq r-t \leq r+t < L, \\ \\
	\displaystyle \frac{1}{2} \int_{r-t}^{2L-r-t} g(\rho) \; d\rho  & L \leq r+t
		\end{array}
	\right.
	\end{equation}
with similar formulas when $-1 \leq t \leq 0$.

First, we establish some estimates on the linear propagator itself.

\begin{lem}\label{freeevgood}
Let $f \in C_0^{s}([0,L] \to \R)$. For $0 < T < \infty$, we have
	\[
	[\cos(t|\partial_r|)f](r), [\sin(t|\partial_r|)f](r)\in C_t^0 C_r^{s}([-T,T] \times [0,L] \to \R).
	\]
Furthermore, there exists $C > 0$ such that
	\begin{align}
	&\left\|[\cos(t|\partial_r|)f](r)\right\|_{C^0_tC^s_r([-T,T] \times [0,L])} \leq C \|f\|_{C^s_0} 
		\\
		\label{sinevol}
	&\left\|[\sin(t|\partial_r|)f](r)\right\|_{C^0_tC^s_r([-T,T] \times [0,L])} \leq C \|f\|_{C^s_0}.
	\end{align}
Also, for every $t \in [-T,T]$, we have
	\[
	0=\left[\sin(t|\partial_r|)f\right](0) = \left[\sin(t|\partial_r|)f\right](L) = \left[\cos(t|\partial_r|)f\right](0) = \left[\cos(t|\partial_r|)f\right](L)
	\]
\end{lem}
\begin{proof}	
The assertions for $\cos(t|\partial_r|)$ follow immediately from \eqref{costnabla}. 

Observe that, by \eqref{costnabla} and \eqref{sintnabla}, we have
	\[
	\partial_r \left[\frac{\sin(t|\partial_r|)}{|\partial_r|} f \right](r) = [\cos(t|\partial_r|)f](r),
	\]
and so 
	\[
	\left\|\partial_r \left[\frac{\sin(t|\partial_r|)}{|\partial_r|} f \right](r) \right\|_{C^0_tC^s_r([-T,T] \times [0,L])} \lesssim \|f\|_{C^s_0} 
	\]
Now, the operator $|\partial_r|(\partial_r)^{-1}$ is the finite volume Hilbert transform and, by Privalov's theorem (cf., \cite{Zyg}), is a bounded linear map from $C^s([0,L])$ to itself. Thus, \eqref{sinevol} follows.

Finally, recall that the Fourier series of H\"older continuous functions converge uniformly to the original function. For fixed $t$, the function $\sin(t|\partial_r|)f$ is $s$-H\"older continuous, and the Fourier series of $\sin(t|\partial_x|) f$ is still a sine series. Thus,
	\[
	0=\left[\sin(t|\partial_x|) f\right](0) = \left[\sin(t|\partial_x|) f\right](L).\qedhere
	\]
\end{proof}

We use Lemma~\ref{freeevgood} to establish a local well-posedness for the second order equation \eqref{secondorderfinite1d}. Similarly to Definition~\ref{strongfinite}, we say that $v(t,r) : [-T,T] \times [0,L]\to \R$ is a strong solution of \eqref{secondorderfinite1d} on $[-T,T]$ with initial data $(f_1,f_2)$ if
	\begin{enumerate}
	\item $v(t,r) \in C^0_tC^s_r([-T,T]\times[0,L]\to\R)$,
	\item $(v,v_t)|_{t=0} = (f_1,f_2)$,
	\item $v(t,r)$ obeys the Duhamel formula
		\[
	v(t,r) = \left[\cos(t|\partial_r|)f_1\right](r) + \left[\frac{\sin(t|\partial_r|)}{|\partial_r|}f_2\right](r) - [K(v)](t,r),
		\]
where
		\[
		[K(v)](t,r) := \int_0^t \left( \frac{\sin(t-\tau)|\partial_r|}{|\partial_r|}\frac{[v(\tau,\cdot)]^3}{(\cdot)^{2}} \right)(r) d\tau.
		\]
	\end{enumerate}

\begin{prop}\label{lwpfiniteforsecondorder}
Fix $\frac{1}{3} < s < \frac{1}{2}$ and fix $L > 2$. Let $f_1 \in C^s_0([0,L] \to \R)$ and let $f_2$ be a distribution supported on $[0,L]$ such that $|\partial_r|^{-1}f_2 \in C^s_0([0,L] \to \R)$. There exists $T \in (0,1)$, whose value depends on $L$, $\|f_1\|_{C^s}$, $\||\partial_r|^{-1}f_2\|_{C^s}$, and a unique strong solution $v(t,r)$ of \eqref{secondorderfinite1d} on $[-T,T]$ with initial data $(f_1,f_2)$.

Moreover,
		\begin{equation}\label{oldflowbound}
		\|v\|_{C^0_tC^s_r([-T,T]\times[0,L] \to \R)} \lesssim_s \|f_1\|_{C^s_0} + \left\||\partial_r|^{-1} f_2 \right\|_{C^s_0}
		\end{equation}
and for every $\sigma \in [0,1)$,
	\begin{equation}\label{almostlip1}
	v(t,r) - \left[\cos(t|\partial_r|)f_1\right](r) - \left[\frac{\sin(t|\partial_r|)}{|\partial_r|}f_2\right](r) \in C^0_tC^\sigma_r([-T,T]\times[0,L] \to \R).
	\end{equation}
\end{prop}
\begin{proof}
We use the abbreviation $C_t^0C_r^s$ in place of $C_t^0 C_r^s([-T,T]\times[0,L])$, where $T$ is a constant to be specified later. We want to show that the mapping
	\begin{equation}\label{contract}
	v(t,r) \mapsto \left[\cos(t|\partial_r|)f_1\right](r) + \left[\frac{\sin(t|\partial_r|)}{|\partial_r|}f_2\right](r) - [K(v)](t,r),
	\end{equation}
admits a unique fixed point in $C_t^0C_r^s$ with zero boundary values by showing that it is a contraction for sufficiently small $T$. 

For $(t,r) \in [-T,T]$, we denote $D(t,r)$ to be the domain of dependence from $(t,r)$. For example, if $r - t < 0$, then
	\begin{align*}
	D(t,r) 
		&= 
	\{(\tau, \rho) \mid t - r \leq \tau \leq t \hbox{ and } r'-t+\tau\leq\rho\leq t+r'-\tau \} \; \cup
		\\
		&\hbox{\hskip 18pt}
	\{(\tau,\rho) \mid 0 \leq \tau \leq t - r \hbox{ and } t-r'-\tau \leq \rho \leq  t+r'-\tau \}.
	\end{align*}

We use two key estimates. Recalling $u(t,0)= 0$, the first estimate is
	\begin{equation}\label{holderbound}
	|v(t,r)| \leq r^s \|v\|_{C_t^0C_r^s},
	\end{equation}
which shall give us integrability in the Duhamel terms. Recalling $3s > 1$, the second estimate is
		\begin{equation}\label{lip}
		b^{3s}-a^{3s} \lesssim_L b-a, \hbox{\hskip 36pt} 0 \leq a \leq b \leq L
		\end{equation}
since the function $f(x) =x^{3s}$ is Lipschitz on bounded domains.

For the remainder of this proof, let us assume $0 \leq r' < r \leq L$. We treat only the case $t \geq 0$; the negative time case is similar.

\begin{itemize}
\item[\textbf{Case 1},] $t \leq \frac{r-r'}{2}$. In this case, $D(t,r) \cap D(t,r') = \emptyset$. So, we simply use
	\begin{equation}\label{nointersect}
	|[K(v)](t,r) - [K(v)](t,r')|  \leq |[K(v)](t,r)| + |[K(v)](t,r')|
	\end{equation}
and estimate each term separately. We seek a bound of the form
	\[
	|[K(v)](t,r) - [K(v)](t,r')| \lesssim_L \|v\|_{C^0_t C^s_r}^3 t,
	\]
as we may then use the estimate $t \lesssim t^{1-\sigma} (r-r')^\sigma$ for every $\sigma \in [0,1)$. In particular, specializing to $\sigma = s$ and taking $t$ sufficiently small will lead us to the desired fixed point of \eqref{contract}. We only estimate $[K(v)](t,r)$; the case $[K(v)](t,r')$ is similar.
	\begin{itemize}
	\item[\textbf{Subcase 1.1},] $r - t < 0$. In particular, $0 < r < t$. Then \eqref{sintnabla}, \eqref{holderbound}, and \eqref{lip} gives
	\begin{align*}
	|[K(v)](t,r)| 
		&\lesssim
	\left[\int_0^{t-r} \int_{(t-\tau)-r}^{r+(t-\tau)} + \int_{t-r}^t \int_{r-(t-\tau)}^{r+(t-\tau)} \right]
		\frac{|v(\tau,\rho)|^3}{\rho^{2}} \; d\rho \; d\tau
		\\
		\\
		&\lesssim
	\|v\|_{C_t^0C_r^s}^3 \left[\int_0^{t-r} \int_{(t-\tau)-r}^{r+(t-\tau)} + \int_{t-r}^t \int_{r-(t-\tau)}^{r+(t-\tau)} \right]
		\rho^{3s-2} \; d\rho \; d\tau
		\\
		&\lesssim_s
	\|v\|_{C_t^0C_r^s}^3 \Big[ (t+r)^{3s} - (t-r)^{3s} \underbrace{- 2(r)^{3s}}_{\leq 0} \Big] 
		\\
		&\lesssim_{s,L}
	\|v\|_{C_t^0C_r^s}^3 r
		\\
		&\lesssim_{s,L}
	\|v\|_{C_t^0C_r^s}^3 t
	\end{align*}
	\item[\textbf{Subcase 1.2},] $0 \leq r-t \leq r + t \leq L$. Then \eqref{sintnabla}, \eqref{holderbound}, and \eqref{lip} gives
		\begin{align*}
	|[K(v)](t,r)| 
		&\lesssim
	\|v\|_{C_t^0C_r^s}^3\int_0^t \int_{r-(t-\tau)}^{r+(t-\tau)} \rho^{3s-2} \; d\rho \; d\tau
		\\
		&\lesssim_s
	\|v\|_{C_t^0C_r^s}^3\big[ (r+t)^{3s} - r^{3s} + \underbrace{(r-t)^{3s} - r^{3s}}_{\leq 0} \big]
		\\
		&\lesssim_{s,L}
	\|v\|_{C_t^0C_r^s}^3 t 
	\end{align*}
	\item[\textbf{Subcase 1.3},] $L < r + t$. In particular, $L-r < t$ and $r - t < L -t < 2L - r - t$. Then \eqref{sintnabla}, \eqref{holderbound}, and \eqref{lip} gives 
			\begin{align*}
	|[K(v)](t,r)| 
		&\lesssim
	\|v\|_{C_t^0C_r^s}^3\left[ \int_0^{t-(L-r)} \int_{r-(t-\tau)}^{2L - r - (t-\tau)} + \int_{t-(L-r)}^t \int_{r-(t-\tau)}^{r+(t-\tau)}\right] 
		\rho^{3s-2} \; d\rho \; d\tau
		\\
		&\lesssim_s
	\|v\|_{C_t^0C_r^s}^3 \big( 2(L^{3s} - r^{3s}) + \underbrace{(r-t)^{3s} - (2L - r - t)^{3s}}_{\leq 0}  \big)
		\\
		&\lesssim_{s,L}
	\|v\|_{C_t^0C_r^s}^3 (L-r) 
		\\
		&\lesssim_{s,L}
	\|v\|_{C_t^0C_r^s}^3t.
	\end{align*}
	\end{itemize}
\item[\textbf{Case 2},] $\frac{r-r'}{2} < t $. In this case, $D(t,r) \cap D(t,r') \neq \emptyset$. Using \eqref{holderbound}, we have
	\begin{equation}
	\big\vert [K(v)](t,r') - [K(v)](t,r)\big\vert \leq \|v\|_{C^0_t C^s_r}^3 \iint_{D(t,r) \triangle D(t,r')} \rho^{3s-2} \; d\rho \; d\tau,
	\end{equation}
	where $A \triangle B := (A \setminus B) \cup (B \setminus A)$ denotes the symmetric difference set. Here, we seek an estimate of the form
		\[
		\iint_{D(t,r) \triangle D(t,r')} \rho^{3s-2} \; d\rho \; d\tau 
			\lesssim_{s,L}
		r-r'
		\]
	as we may then use the estimate $r-r' \lesssim t^{1-\sigma}(r-r')^\sigma$ for every $\sigma \in (0,1)$.
	
	\begin{itemize}
	\item[\textbf{Subcase 2.1},] $0\leq r'-t$ and $r +t \leq L$. The domain of dependence is sketched in Figure~\ref{midmid}. In this case, $D(t,r') \triangle D(t,r) \subset R_1 \cup R_2$, where
		\begin{align*}
		R_1 
			&= 
		\{(\tau,\rho) \mid  0 \leq \tau \leq t \hbox{ and } r'+t-\tau\leq\rho\leq r+t-\tau \},
		\\
		R_2
			&=
		\{ (\tau,\rho) \mid  0 \leq \tau \leq t \hbox{ and } r'-t+\tau\leq\rho\leq r-t+\tau \}.
		\end{align*}
	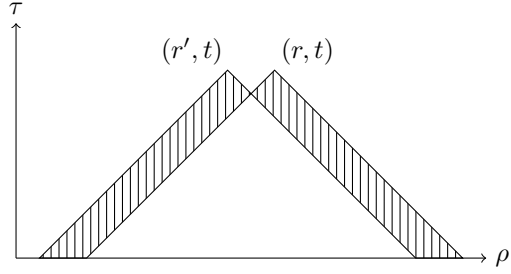
\begin{figure}
\begin{tikzpicture}[scale=1.25]
\draw[->] (0,0) -- (0,2.5);
\draw[->] (0,0) -- (5,0);
\node [left] at (2.3,2.2) {$(r',t)$};
\node [right] at (2.72,2.2) {$(r,t)$};
\node [above] at (0,2.5) {$\tau$};
\node [right] at (5,0) {$\rho$};
\path[draw,pattern=vertical lines] (0.25,0) -- (2.25,2) -- (2.5,1.75) -- (0.75,0) -- cycle;
\path[draw,pattern=vertical lines] (2.5,1.75) -- (2.75,2) -- (4.75,0) -- (4.25,0) -- cycle;
\end{tikzpicture}
\caption{Subcase 2.1: shaded region is $D(t,r) \triangle D(t,r')$.} \label{midmid}
	\end{figure}

Using \eqref{lip}, we have
		\begin{align*}
		\iint_{R_1} \rho^{3s-2} \; d\rho \; d\tau 
			&=
		\int_0^t\int_{r'+t-\tau}^{r+t-\tau} \rho^{3s-2} \; d\rho \; d\tau
			\\
			&\lesssim
		-(r)^{3s}+(r+t)^{3s}+(r)^{3s}-(r'+t)^{3s}
			\\
			&\lesssim_L
		r-r'.
		\end{align*}
	and
		\begin{align*}
		\iint_{R_2} \rho^{3s-2} \; d\rho \; d\tau
			&=
		\int_0^t\int_{r'-t+\tau}^{r-t+\tau} \rho^{3s-2} \; d\rho \; d\tau
			\\
			&\lesssim_s
		(r)^{3s} - (r')^{3s} + \underbrace{(r'-t)^{3s} - (r-t)^{3s}}_{\leq 0}
			\\
			&\lesssim_{s,L}
		r-r'.
		\end{align*}
	\item[\textbf{Subcase 2.2},] $r' -t < 0 \leq \frac{r' +r}{2}-t$ and $r +t \leq L$. The domain of dependence is sketched in Figure~\ref{lomidnocross}. In this case, $D(t,r') \triangle D(t,r) \subset R_1 \cup R_2 \cup R_3$, where
		\begin{align*}
		R_1 
			&= 
		\{(\tau,\rho) \mid  0 \leq \tau \leq t\hbox{ and } r'+t-\tau\leq\rho\leq r+t-\tau \},
		\\
		R_2
			&=
		\{ (\tau,\rho) \mid  t-r' \leq \tau \leq t \hbox{ and } r'-t+\tau\leq\rho\leq r-t+\tau \},
		\\
		R_3
			&=
		\{ (\tau,\rho) \mid  0 \leq \tau \leq t -r'\hbox{ and } t-r'-\tau\leq\rho\leq t+r'-\tau \}.
		\end{align*}
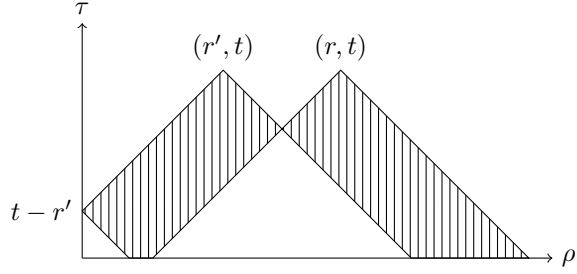
\begin{figure}
\begin{tikzpicture}[scale=1.25]
\draw[->] (0,0) -- (0,2.5);
\draw[->] (0,0) -- (5,0);
\node [left] at (0,0.5) {$t-r'$};
\node [above] at (1.5,2) {$(r',t)$};
\node [above] at (2.75,2) {$(r,t)$};
\node [above] at (0,2.5) {$\tau$};
\node [right] at (5,0) {$\rho$};
\path[draw,pattern=vertical lines] (0.5,0) -- (0,0.5) -- (1.5,2) -- (2.125,1.375) -- (0.75,0) -- cycle;
\path[draw,pattern=vertical lines] (2.125,1.375) -- (2.75,2) -- (4.75,0) -- (3.5,0) -- cycle;
\end{tikzpicture}
\caption{Subcase 2.2: shaded region is $D(t,r)\triangle D(t,r')$.} \label{lomidnocross}
\end{figure}

	Similarly to Subcase 2.1, we have $\iint_{R_1 \cup R_2} \rho^{3s-2} \; d\rho \; d\tau \lesssim_L r-r'$. Using \eqref{lip} and that $t \leq \frac{r'+r}{2}$, we have
		\begin{align*}
		\iint_{R^3} \rho^{3s-2}\;d\rho\;d\tau
			&\lesssim
		\int_0^{t-r'} \int_{t-r'-\tau}^{t+r'-\tau} \rho^{3s-2} \; d\rho \; d\tau
			\\
			&\lesssim_s
		-(2r')^{3s}+(t+r')^{3s}\underbrace{-(t-r')^{3s}}_{\leq 0}
			\\
			&\lesssim_{s,L}
		t - r' 
			\\
			&\lesssim_{s,L}
			r - r'.
		\end{align*}
	\item[\textbf{Subcase 2.3},] $\frac{r' +r}{2} -t < 0 \leq r-t$ and $r +t \leq L$. The domain of dependence is sketched in Figure~\ref{lomidcross}. In this case, $D(t,r') \triangle D(t,r) \subset \bigcup_{j=1}^4 R_j$, where
		\begin{align*}
		R_1 
			&= 
		\{(\tau,\rho) \mid  0 \leq \tau \leq t\hbox{ and } r'+t-\tau\leq\rho\leq r+t-\tau \},
		\\
		R_2
			&=
		\{ (\tau,\rho) \mid  t-r' \leq \tau \leq t \hbox{ and } r'-t+\tau\leq\rho\leq r-t+\tau \},
		\\
		R_3
			&=
		\{ (\tau,\rho) \mid  t-\tfrac{r'+r}{2} \leq \tau \leq t -r'\hbox{ and } 0 \leq\rho\leq r-r' \},
		\\
		R_4
			&=
		\{ (\tau,\rho) \mid  0 \leq \tau \leq t-\tfrac{r'+r}{2}\hbox{ and } r-t+\tau\leq\rho\leq t-r'-\tau \}.
		\end{align*}
	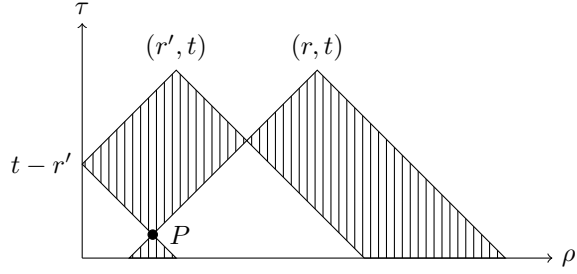
\begin{figure}
\begin{tikzpicture}[scale=1.25]
\draw[->] (0,0) -- (0,2.5);
\draw[->] (0,0) -- (5,0);
\draw[fill] (0.75,0.25) circle [radius=0.05];
\node [right] at (0.75,0.25) {$\;P$};
\node [above] at (1,2) {$(r',t)$};
\node [above] at (2.5,2) {$(r,t)$};
\node [left] at (0,1) {$t-r'$};
\node [above] at (0,2.5) {$\tau$};
\node [right] at (5,0) {$\rho$};
\path[draw,pattern=vertical lines] (0,1) -- (1,2) -- (1.75,1.25) -- (0.75,0.25) -- cycle;
\path[draw,pattern=vertical lines] (0.5,0) -- (0.75,0.25) -- (1,0) -- cycle;
\path[draw,pattern=vertical lines] (1.75,1.25) -- (2.5,2) -- (4.5,0) -- (3,0) -- cycle;
\end{tikzpicture}
\caption{Subcase 2.3: shaded region is $D(t,r)\triangle D(t,r')$, and $P = \big(t-\frac{r+r'}{2},\tfrac{r-r'}{2}\big)$.} \label{lomidcross}
\end{figure}

	Similarly to Subcase 2.1, we have $\iint_{R_1 \cup R_2} \rho^{3s-2} \; d\rho \; d\tau \lesssim_L r - r'$. Using \eqref{lip}, we have
		\begin{align*}
		\iint_{R^3} \rho^{3s-2}\;d\rho\;d\tau
			&\lesssim
		\int_{t-\frac{r'+r}{2}}^{t-r'} \int_0^{r-r'} \rho^{3s-2} \; d\rho d\tau
			\\
			&\lesssim_{s,L}
		r - r'.
		\end{align*}
	Also, we have
		\begin{align*}
		\iint_{R^4} \rho^{3s-2} \; d\rho \;d\tau
			&\lesssim
		\int_{0}^{t-\frac{r'+r}{2}} \int^{t-r'-\tau}_{r-t+\tau} \rho^{3s-2} \; d\rho \; d\tau
			\\
			&\lesssim_s
		-\left(\frac{r-r'}{2}\right)^{3s} +(t-r')^{3s}-\left(\frac{r-r'}{2}\right)^{3s}+(r-t)^{3s}.
		\end{align*}
	Using \eqref{lip} and the fact that $t \leq r$ means $t - \frac{r}{2} \leq \frac{r}{2}$, we have
		\[
		-\left(\frac{r-r'}{2}\right)^{3s} +(t-r')^{3s} \lesssim_L t- \frac{r}{2} - \frac{r'}{2} \lesssim_L r - r'.
		\]
	Using \eqref{lip} and the fact that $-t <-\frac{r +r'}{2}$, we have
		\[
		-\left(\frac{r-r'}{2}\right)^{3s}+(r-t)^{3s} \lesssim_L r - t -\frac{r-r'}{2} \lesssim_L r - r'.
		\]
	\item[\textbf{Subcase 2.4},] $r -t <0$ and $r +t \leq L$. The domain of dependence is sketched in Figure~\ref{lolo}. In this case, $D(t,r') \triangle D(t,r) \subseteq \bigcup_{j=1}^4 R_j$, where
		\begin{align*}
		R_1 
			&= 
		\{(\tau,\rho) \mid  0 \leq \tau \leq t\hbox{ and } r'+t-\tau\leq\rho\leq r+t-\tau \},
		\\
		R_2
			&=
		\{ (\tau,\rho) \mid  t-r' \leq \tau \leq t \hbox{ and } r'-t+\tau\leq\rho\leq r-t+\tau \},
		\\
		R_3
			&=
		\{ (\tau,\rho) \mid  t-r \leq \tau \leq t -r'\hbox{ and } 0 \leq\rho\leq r-r' \},
		\\
		R_4
			&=
		\{ (\tau,\rho) \mid  0 \leq \tau \leq t-r\hbox{ and } t-r'-\tau\leq\rho\leq t-r-\tau \}.
		\end{align*}
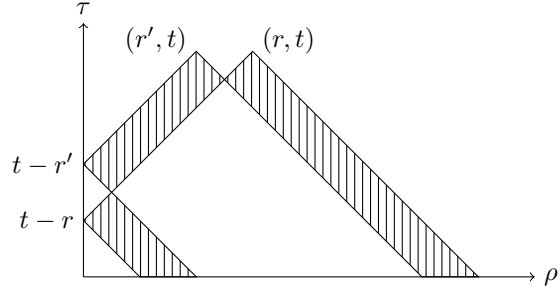
\begin{figure}
\begin{tikzpicture}[scale=1.5]
\draw[->] (0,0) -- (0,2.25);
\draw[->] (0,0) -- (4,0);
\node [left] at (0,1) {$t-r'$};
\node [left] at (0,0.5) {$t-r$};
\node [left] at (1,2.1) {$(r',t)$};
\node [right] at (1.5,2.1) {$(r,t)$};
\node [above] at (0,2.25) {$\tau$};
\node [right] at (4,0) {$\rho$};
\path[draw,pattern=vertical lines] (0,1) -- (1,2) -- (1.25,1.75) -- (0.25,0.75) -- cycle;
\path[draw,pattern=vertical lines] (1.25,1.75) -- (1.5,2) -- (3.5,0) -- (3,0) -- cycle;
\path[draw,pattern=vertical lines] (0.5,0) -- (1,0) -- (0.25,0.75) -- (0,0.5) -- cycle;
\end{tikzpicture}
\caption{Subcase 2.4: shaded region is $D(t,r) \triangle D(t,r')$.}\label{lolo}
\end{figure}

	Similarly to Subcase 2.1, we have $\iint_{R_1 \cup R_2} \rho^{3s-2} \; d\rho \; d\tau \lesssim_L r - r'$. Similarly to Subcase 2.3, we also have $\iint_{R^3} \rho^{3s - 2} \; d\rho\; d\tau \lesssim_L r - r'$. Using \eqref{lip}, we have
		\begin{align*}
		\iint_{R^4} \rho^{3s-2} \; d\rho \; d\tau 
			&\lesssim
		\int_0^{t-r} \int_{t-r'-\tau}^{t-r-\tau} \rho^{3s-2} \; d\rho \; d\tau
			\\
			&\lesssim_s
		-(r-r')^{3s}+(t-r')^{3s}-(t-r)^{3s} 
			\\
			&\lesssim_{s,L}
		r-r'.
		\end{align*}
	\item[\textbf{Subcase 2.5},] $0 \leq r' - t$ and $\frac{r+r'}{2} + t \leq L < r + t$. This case follows from reflecting the domains in Subcase 2.2 across the line $\rho = L/2$ and noting that $\rho^{3s-2}$ is a decreasing function for $\rho \geq 0$.
	\item[\textbf{Subcase 2.6},] $0 \leq r' - t$ and $r'+t \leq L < \frac{r+r'}{2} + t$, which follows from reflecting the domains in Subcase 2.3 across the line $\rho = L/2$.
	\item[\textbf{Subcase 2.7},] $r' + t > L$, which follows from reflecting the domains in Subcase 2.4 across the line $\rho = L/2$.
	\end{itemize}
\end{itemize}

Combining all the results from these cases, we have
	\[
	|K(v)(t,r) - K(v)(t,r')| \lesssim_L \|v\|_{C_t^0C_r^s}^3 (r-r')^\sigma t^{1-\sigma}
	\]
for every $\sigma \in (0,1)$, which is to say,
	\begin{equation}\label{holderselfmap}
	\big\|[K(v)](t,r) \big\|_{C^0_tC^\sigma_r([-T,T] \times [0,L])} \lesssim_L T^{1-\sigma} \|v\|_{C_t^0C_r^s}^3 .
	\end{equation}
Using the fact that $|a^3 - b^3| \lesssim |a-b|(|a|^2 + |b|^2) $, a similar computation shows that
	\begin{equation}\label{holdercontract}
	\big\|[K(v)](t,r) - [K(\tilde{v})](t,r) \big\|_{C^0_tC^\sigma_r} \lesssim_L T^{1-\sigma} \|v - \tilde{v}\|_{C_t^0C_r^s} \left( \|v\|_{C_t^0C_r^s}^2 + \|\tilde{v}\|_{C_t^0C_r^s}^2 \right).
	\end{equation}

Let us specialize to $\sigma = s$. Then, for $T$ sufficiently small, the map in \eqref{contract} is a self-mapping of the closed ball
	\[
	\left\{ v \in C^0_tC^s_r \mid \|v\|_{C^0_tC^s_r} \leq (C+1) \Big( \|f_1\|_{C^s_0} + \left\||\partial_r|^{-1} f_2\right\|_{C^s_0} \Big) \right\}
	\]
as well as a contraction. Here, $C$ is the constant from Lemma~\ref{freeevgood}. By contraction mapping, \eqref{contract} admits a unique fixed point $v$, which is also the desired strong solution.
\end{proof}	


\begin{prop}\label{lwpfinite}
Fix $\frac{1}{3} < s < \frac{1}{2}$ and fix $L > 2$. For each $\Lambda > 0$, there exists $T \in (0,1)$, whose value depends on $s,L,$ and $\Lambda$, with the following properties:
	\begin{enumerate}
	\item Let $g\in C^{s}_0([0,L] \to \C)$ such that $\|g\|_{C^s([0,L])} \leq \Lambda$. Then $\flow_L(t,g)$ is defined for all $t \in [-T,T]$. Furthermore,
			\begin{equation}\label{flowbound}
	\|\flow_L(t,g)\|_{C^0_tC^s_r([-T,T]\times[0,L])} \lesssim_{s,L} \|g\|_{C^s([0,L])}^3 + \|g\|_{C^s([0,L])},
		\end{equation}
	and for each $\sigma \in [0,1)$, 
		\begin{equation}\label{lipgood}
		\flow_L(t,g) - e^{-it|\partial_r|}g \in C_t^0 C^\sigma_r([-T,T] \times [0,L] \to \C).
		\end{equation}
	\item Let $g$ be as above. If $\tilde{g} \in C^s_0([0,L])$ is such that $\|g-\tilde{g}\|_{C^s([0,L])}$ is sufficiently small, depending on $\Lambda, s, L$, then $\flow_L(t,\tilde{g})$ also exists for all $t \in [-T,T]$ and
			\begin{equation}\label{flowcontract}
	\|\flow_L(t,g) - \flow_L(t,\tilde{g})\|_{C^0_tC^s_r([-T,T]\times[0,L])} \lesssim_{s,L} \|g - \tilde{g}\|_{C^s} (\|g\|_{C^s}^2 + \|\tilde{g}\|_{C^s}^2).
			\end{equation}
	\end{enumerate}
\end{prop}
\begin{proof}
Let $g \in C^s_0([0,L]\to\C)$ such that $\|g\|_{C^s([0,L])} \leq \Lambda$. Observe that the pair $(\re(g), |\partial_r| \im(g))$ obeys the hypotheses of Proposition \ref{lwpfiniteforsecondorder}. Hence there is a time interval $[-T,T]$, depending on $\Lambda,s,L$, and a unique strong solution $v(t,r) :[-T,T]\times[0,L]\to\R$ of \eqref{secondorderfinite1d} with initial data $(\re(g), |\partial_r| \im(g))$. By \eqref{oldflowbound},
	\begin{equation}\label{goodbound}
	\|v\|_{C^0_tC^s_r([-T,T]\times[0,L])} \lesssim_{s,L} \|g\|_{C^s([0,L])}.
	\end{equation}
	
Observe that
	\[
	|\partial_r|^{-1}\partial_t v = -\sin(t|\partial_r|)\re(g) + \cos(t|\partial_r|)\im(g) - \int_0^t \frac{\cos((t-\tau)|\partial_r|)}{|\partial_r|} \frac{(v(\tau,\cdot))^3}{(\cdot)^2} \; d\tau.
	\]
We claim that 
	\[
	\left(|\partial_r|^{-1}\partial_t v\right)(t,r) \in C^0_tC^s_r([-T,T]\times[0,L]\to\C)
	\] 
and 
	\[
	0 = \left(|\partial_r|^{-1}\partial_t v\right)(t,0) = \left(|\partial_r|^{-1}\partial_t v\right)(t,L).
	\]
Assuming this claim, then 	
	\begin{equation}\label{solution}
	w := v + i|\partial_r|^{-1}\partial_t v
	\end{equation}
would be the unique\footnote{Uniqueness follows from $(v,v_t) = (\re(w), |\partial_r| \im(w))$ and the uniqueness aspect of Proposition \ref{lwpfiniteforsecondorder}.} strong solution of \eqref{firstorderfinite1d}.

We write $\|F\|_{C^0_t C^s_r} := \|F\|_{C^0_tC^s_r([-T,T]\times[0,L]\to\C)}$. By Lemma~\ref{freeevgood}, we have
	\begin{equation}\label{linflow}
	\left\|-\sin(t|\partial_r|)\re(g) + \cos(t|\partial_r|)\im(g)\right\|_{C^0_tC^s_r} \lesssim \|g\|_{C^s([0,L] \to \C)}.
	\end{equation}
Letting
	\[
	[\tilde{K}(v)](t,r) := \int_0^t \left[\frac{\cos((t-\tau)|\partial_r|)}{|\partial_r|} \frac{(v(\tau,\cdot))^3}{(\cdot)^2}\right](r) \; d\tau,
	\]
we shall prove $[\tilde{K}(v)](t,r) \in C^0_t W^{1,p}_r([-T,T] \times [0,L] \to \C)$ for large, but finite, $p$ and apply Sobolev embedding.

First, we realize $|\partial_r|^{-1}$ as a convolution operator: for $f \in L^p([0,L] \to \C)$ with $p \in (1,\infty]$, we extend it to $[-L,L]$ via $f(-r) = -f(r)$ for $r \in [0,L]$, and then extend to $\R$ via $2L$ periodicity. Having this extension, then
	\begin{equation}\label{convo}
	\left[|\partial_r|^{-1} f\right](r) = (f * h)(r) = \int_{-L}^L f(r-\rho)h(\rho) \; d\rho,
	\end{equation}
where
	\[
	h(r) = \sum_{n=1}^\infty (n \pi/ L)^{-1} \cos(n \pi r / L) = \re \left( \log (1-e^{i\pi r/L}) \right).
	\]
Note that $h(r) \in L^q([-L,L] \to \C)$ for every $q \in [1,\infty)$. 

Next, we check boundary conditions. By \eqref{holderbound}, we have
	\begin{equation}\label{origindecay}
	\|(v(t,r))^3r^{-2}\|_{C^0_t L^p_r ([-T,T] \times [0,L])} \lesssim_p \|v\|^3_{C^0_tC^s_r}.
	\end{equation}
for every $p \in [1, \frac{1}{2-3s})$. Given \eqref{convo}, \eqref{origindecay}, and the fact that both $\cos(n\pi (-\rho)/L)$ and $\cos(n\pi (L-\rho)/L) = (-1)^n\cos(n\pi \rho/L)$ are even in $\rho$, we have
	\[
	\left[|\partial_r|^{-1} \frac{(v(t,\cdot))^3}{(\cdot)^2} \right](0) = 0 = \left[|\partial_r|^{-1} \frac{(v(t,\cdot))^3}{(\cdot)^2} \right](L)
	\]
for all $t$. Given \eqref{costnabla}, it follows that $0=[\tilde{K}(v)](t,0) = [\tilde{K}(v)](t,L)$, as well.

 In view of \eqref{costnabla}, we also have
	\[
	\left\| \left[\cos((t-\tau)|\partial_r|) \frac{(v(\tau,\cdot))^3}{(\cdot)^2}\right](r) \right\|_{C^0_\tau L^p_r([0,t]\times[0,L])} \lesssim \|v\|_{C^0_t C^s_r}^3.
	\]
for every $t$. By H\"older's inequality,
	\begin{equation}\label{fubiniflow}
	\int_0^t \int_{-L}^L \left\vert h(r-\rho) \right\vert \left\vert \left[\cos((t-\tau)|\partial_r|) \frac{(v(\tau,\cdot))^3}{(\cdot)^2}\right](\rho) \right\vert \; d\rho \; d\tau \lesssim_L \|v\|_{C_t^0C^s_r}^3,
	\end{equation}
where we extend the integrands from $[0,L]$ to $\R$ in the manner above. Hence,
	\begin{equation}\label{tildeKbound}
	\left\| [\tilde{K}(v)](t,r) \right\|_{C^0_t C^0_r([-T,T] \times [0,L])} \lesssim \|v\|_{C_t^0C^s_r}^3.
	\end{equation}
Furthermore, by Fubini's Theorem, we may also write
	\begin{equation}\label{tildeKdef}
	[\tilde{K}(v)](t,r) = \left[|\partial_r|^{-1} F(t,\cdot)\right](r)
	\end{equation}
where
	\[
	F(t,r) := \int_0^t \left[\cos((t-\tau)|\partial_r|) \frac{(v(\tau,\cdot))^3}{(\cdot)^2}\right](r) \; d\tau.
	\]
		
If $0 \leq r -t < r + t \leq L$, then \eqref{costnabla} and \eqref{holderbound} gives
	\begin{align*}
	|F(t,r)|
		&\lesssim
	\int_0^t \left\lvert \frac{(v(\tau,r+(t-\tau))^3}{(r+(t-\tau))^2} \right\rvert + \left\lvert \frac{(v(\tau,r-(t-\tau))^3}{(r-(t-\tau))^2} \right\rvert d\tau
		\\
		&\lesssim
	\|v\|_{C^0_tC^s_r}^3 \int_0^t (r+(t-\tau))^{3s-2} + (r-(t-\tau))^{3s-2} \; d\tau
		\\
		&\lesssim_{s,L}
	\|v\|_{C^0_tC^s_r}^3.
	\end{align*}
When $r - t < 0$ or when $r + t > L$, we may similarly show $|F(t,r)| \lesssim_{s,L} \|v\|^3_{C^0_tC^s_r}$, and thus
	\begin{equation}\label{w1pbound}
	\|F(t,r)\|_{C^0_t C^0_r} \lesssim_{s,L} \|v\|^3_{C^0_tC^s_r}.
	\end{equation}
Finally, the operator $\partial_r|\partial_r|^{-1}$ (i.e., finite volume Hilbert transform) is a bounded linear operator from $L^p_r([0,L])$ to itself for every $p \in [1,\infty)$. Thus, \eqref{tildeKbound}, \eqref{tildeKdef}, and \eqref{w1pbound} gives
	\begin{equation}\label{tildeKw1p}
	\left\| [\tilde{K}(v)](t,r) \right\|_{C^0_t W^{1,p}_r([-T,T] \times [0,L])} \lesssim_{p,L} \|v\|_{C_t^0C^s_r}^3
	\end{equation}
for every $p \in [1,\infty)$. By Sobolev embedding, we have 
	\begin{equation}\label{almostlip2}
	\tilde{K}(v) \in C^0_t C^\sigma_r([-T,T]\times[0,L] \to \C)
	\end{equation}
for every $\sigma \in [0,1)$ and
	\begin{equation}\label{holderflow}
	\left\| [\tilde{K}(v)](t,r) \right\|_{C^0_t C^\sigma_r([-T,T] \times [0,L])} \lesssim_{\sigma,L} \|v\|_{C_t^0C^s_r}^3.  
	\end{equation}
Specializing to $\sigma = s$ and combining \eqref{goodbound}, \eqref{linflow}, \eqref{holderflow}, this proves
	\[
	\|w\|_{C^0_t C^s_r}\lesssim \|g\|_{C^s}^3 + \|g\|_{C^s}.
	\]
In particular, we have established \eqref{flowbound} and that $w$ is the unique strong solution of \eqref{firstorderfinite1d} with initial data $g$. \eqref{lipgood} follows from \eqref{almostlip1} and \eqref{almostlip2}. 

The second part of the proposition follows from the fact that \[|a^3 - b^3| \lesssim |a-b|(|a|^2+|b|^2)\] and arguing as above (and replacing $T$ by $T/2$, if necessary).
\end{proof}

An immediate corollary is the following continuity result. 

\begin{cor}\label{stability}
Let $L > 2$ and let $0 \leq R \leq L-2$. Let $g_k \in C^s_0([0,L]\to\C)$, $1 \leq k \leq \infty$, such that each $g_k$ admits a unique strong solution of \eqref{firstorderfinite1d} for $|t| \leq 1$. If
	\[
	\lim_{k \to \infty}\|g_k -g_\infty\|_{C^s([0,R+1])} = 0,
	\]
then 
	\begin{equation}\label{localconv}
	\lim_{k \to \infty} \|\flow_L(t,g_k) - \flow_L(t,g_\infty)\|_{C^0_tC^s_r([-1,1]\times[0,R])} = 0.
	\end{equation}
Furthermore, if 
	\[
	\lim_{k \to \infty}\|g_k -g_\infty\|_{C^s([0,L])} = 0,
	\]
then
	\begin{equation}\label{allconv}
	\lim_{k \to \infty} \|\flow_L(t,g_k) - \flow_L(t,g_\infty)\|_{C^0_tC^s_r([-1,1]\times[0,L])} = 0.
	\end{equation}
\end{cor}
\begin{proof}
For each $\lambda \in [0, L-1]$, we define the ``linear cut-off'' operator $\Psi_R : C^s_0([0,L]) \to C^s_0([0,L])$ via
	\[
	[\Psi_\lambda g](r)
		=
	\left\{\begin{array}{ll}
	g(r) & 0 \leq r \leq \lambda, \\
	g(\lambda)(\lambda+1-x) & \lambda \leq r \leq \lambda+1, \\
	0 & \lambda+1 \leq r \leq L.
	\end{array}\right.
	\]
Indeed,
	\begin{equation}\label{linearbounds}
	\|\Psi_\lambda g\|_{C^s_0([0,L])} \lesssim_{s,L} \|g\|_{C^s_0([0,L])}
	\end{equation}
where the implicit constant is independent of the choice of $\lambda$. Also, we clearly have
	\begin{equation}\label{extensionconv}
	\lim_{k \to \infty} \| \Psi_{R+1}g_k - \Psi_{R+1} g_\infty\|_{C^s([0,L])} = 0.
	\end{equation}

Letting 
	\[
	\Lambda := \|\flow_L(t,g_\infty)\|_{C^0_tC^s_r([-1,1] \times [0,L])},
	\]
then \eqref{linearbounds} gives
	\[
	\|\Psi_{R+1}g_\infty\|_{C^s([0,L])} \lesssim_{s,L} \|g_\infty\|_{C^s([0,L])} \leq \Lambda.
	\]
By Proposition~\ref{lwpfinite}, there exists a time $T$, depending upon $s,L,$ and $\Lambda$, such that $\flow_L(t,\Psi_{R+1}g_\infty) \in C^0_tC^r_s( [-T,T]\times[0,L] \to \C)$. Proposition~\ref{lwpfinite} also gives 
	\[
	\lim_{k \to \infty} \|\flow_L(t,\Psi_{R+1}g_k) - \flow_L(t,\Psi_{R+1}g_\infty)\|_{C^0_tC^r_s( [-T,T]\times[0,L])} = 0.
	\]
By finite speed of propagation,
	\[
	\lim_{k \to \infty} \|\flow_L(t,g_k) - \flow_L(t,g_\infty)\|_{C^0_tC^r_s( [-T,T]\times[0,R+1-T])} = 0.
	\]
We now seek to iterate the argument.

As above, we have
	\[
	\|\Psi_{R+1-T} \flow_L(T,g_\infty)\|_{C^s([0,L])} \lesssim_{s,L}  	\|\flow_L(T,g_\infty)\|_{C^s([0,L])} \leq \Lambda
	\]
and 
	\[
	\lim_{k \to \infty} \|\Psi_{R+1-T}\flow_L(T,g_k) - \Psi_{R+1-T}\flow_L(T,g_\infty)\|_{C^r_s([0,L])} = 0.
	\]
Thus, we may argue as above, using Proposition~\ref{lwpfinite}, to conclude that
	\[
	\lim_{k \to \infty} \|\flow_L(t,g_k) - \flow_L(t,g_\infty)\|_{C^0_tC^r_s( [0,2T]\times[0,R+1-2T])} = 0.
	\]
We may iterate this argument approximately $2 \lfloor 1/T \rfloor$ times, using the value $\Lambda$ as a persistent bound, to establish \eqref{localconv}. To prove \eqref{allconv}, we argue as above, again using $\Lambda$ as a persistent bound, but without using the operator $\Psi_R$.
\end{proof}

\subsection{Proof of Theorem \ref{invariancefinite}}\label{prooffinite}

By Theorem~\ref{holdersobolev}, the set $C^{s}_0([0,L] \to \C)$ is a Borel subset of $\dot{H}^s_0([0,L] \to \C)$ and the Borel $\sigma$-algebras on $C^s_0([0,L] \to \C)$ generated by the $s$-H\"older norm and by the Sobolev norm must agree. Recall from the discussion in Section~\ref{reform} that $\nu_L(C^{s}_0([0,L])) = 1$.

Let $\Pi_L$ be as in Theorem \ref{BTS}, then $\Pi_L \cap C^{0,s}_0([0,L] \to \C)$ has full $\nu_L$ measure. By Theorem \ref{BTS}, the set
	\[
	\Omega_L := \bigcap_{t \in \Q} \flow_L\left(t,\Pi_L \cap C^{s}_0\right)
	\]
is a well-defined Borel subset of $C^s_0([0,L]\to \C)$ with $\nu_L$ measure $1$. This gives the first assertion of Theorem~\ref{invariancefinite}. 

For the second assertion, let $T \in (0,\infty)$ and let $g \in \Omega_L$. By Theorem \ref{BTS}, $\flow_L(t,g)$ is defined globally in time and we have $\flow_L(t,g) \in C_t^0 \dot{H}^s_0([-T,T] \times [0,L] \to \C)$. The fact that $\flow_L(t,g) \in C_t^0C_r^s([-T,T] \times [0,L] \to \C)$ follows from Proposition \ref{lwpfinite} and the definition of $\Omega_L$. 

The fact that $\flow_L$ preserves Borel measurability follows from Theorem~\ref{BTS} and Theorem~\ref{holdersobolev}. Moreover, by Theorem~\ref{BTS}, $\flow_L$ preserves the measure of all Borel sets.  

Finally, let $A \subseteq C^s_0([0,L] \to \C)$ be $\nu_L$-measurable with $\nu_L(A) = 0$. Recalling Proposition~\ref{regular}, for every $n > 0$, there exists an open set $U_n \supset A$ such that $\nu_L(U_n) < \frac{1}{n}$. By the previous paragraph, $\flow_L(t,U_n)$ is Borel measurable with $\nu_L(\flow_L(t,U_n)) < \frac{1}{n}$. It follows that $\bigcap_{n=1}^\infty \flow_L(t,U_n)$ is a Borel set of measure $0$ which contains $\nu_L(\flow_L(t,A))$. Thus $\nu_L(\flow_L(t,A))$ is $\nu_L$-measurable with measure $0$. As every $\nu_L$-measurable set is the union of a Borel set and a $\nu_L$-null set, Theorem~\ref{invariancefinite} follows.

	
\section{Construction of the infinite volume measure}\label{bigconstruct}

In this section, we prove Theorem~\ref{infvolmeas}. Ultimately, this result will follow from an analysis of the long-time asymptotics of the fundamental solution of a particular parabolic PDE (cf., Definition~\ref{fundsolndefn} and \eqref{parabpde} below). We reduce the measure theoretic problem to the parabolic PDE computation in the following manner: first observing the equivalence of the Borel and the cylinder $\sigma$-algebras (see definition below) on $C^s_0([0,L] \to \C)$ and on $C^s_{loc}([0,\infty) \to \C)$, and then seeking to utilize the Kolmogorov consistency and continuity theorem.

\begin{defn}\label{cylsigalg}
Let $\Lambda = \C$ or $\R$. Let $I \subseteq \R$ be an interval and let $X$ be a subset of $\Lambda^I:=\{ f : I \to \Lambda\}$. The \textit{cylinder set $\sigma$-algebra on $X$} is the $\sigma$-algebra generated by sets of the form
	\[
	\{f \in X \mid f(r) \in B\}
	\]
where $r \in I$ and $B \subseteq \Lambda$ is Borel. We say that a cylinder set probability measure $\mu$ on $X$ is supported on a cylinder measurable subset $A \subseteq X$ if $\mu(A) = 1$.
\end{defn}

\begin{prop}\label{holderpolish}
Let $\Lambda = \C$ or $\R$. The Borel and cylinder $\sigma$-algebras on $C^s_0([0,L]\to \Lambda)$ coincide. Also, endow $C^s_{loc}([0,\infty) \to \Lambda)$ with the metric
	\begin{equation}\label{metric}
	d(f,g) = \sum_{n=1}^\infty 2^{-n} \frac{\|f-g\|_{C^s([0,n])}}{1 + \|f-g\|_{C^s([0,n])}}
	\end{equation}
and the induced metric topology. Then $C^s_{loc}([0,\infty)\to \Lambda)$ is a Polish space (i.e., separable, completely metrizable). Furthermore, the Borel $\sigma$-algebra generated by \eqref{metric} and the cylinder $\sigma$-algebra on $C^s_{loc}([0,\infty) \to \Lambda)$ coincide.
\end{prop}
\begin{rem}
The full strength of the fact that $C^s_{loc}([0,\infty)\to \C)$ is Polish will not be used until Section~\ref{invariance}. We record the result for convenience in the following proof.
\end{rem}
\begin{proof}
Note that convergence in the $C^s_0([0,L]\to \C)$ norm implies uniform convergence, and hence point-wise convergence. It follows that evaluation at a point is continuous with respect to this norm, which then implies that the cylinder $\sigma$-algebra is contained in the Borel $\sigma$-algebra on $C^s_0([0,L]\to\C)$. 

For the reverse containment, let us fix $f_0 \in C^s_0([0,L]\to\C)$ and $\lambda > 0$. Then
	\begin{align*}
	&\overline{B_L(f_0,\lambda)}
		:=
	\{f \in C^s_{0}([0,L]\to \C) \mid \|f - f_0\|_{C^s([0,L])} \leq \lambda \}
		\\
		&\hbox{\hskip 8pt} =
	\bigcap_{ \substack{r,r' \in \Q \cap [0,L],\\ r' \leq r}} \{f \mid \big|[f(r)-f_0(r)] - [f(r') - f_0(r')]\big| \leq \lambda(r-r')^s\}
		\\
		&\hbox{\hskip 8pt} =
	\bigcap_{ \substack{r,r' \in \Q \cap [0,L],\\ r' \leq r}} \left\{ f \mid f(r') \in \C,  f(r) \in \overline{B_\C\Big(f(r')+f_0(r') - f_0(r), \lambda(r-r')^s`\Big)} \right\},
	\end{align*}
which is a countable intersection of cylinder sets. Therefore,
	\[
	B_L(f_0,\lambda) := \{ f \mid \|f - f_0\|_{C^s([0,L])} < \lambda \} = \bigcup_{k=1}^\infty \overline{B_L(f_0,\lambda(1-2^{-k}))},
	\]
from which it follows that every Borel subset of $C^s_0([0,L]\to\C)$ is cylinder measurable. 

Fix $n \in \N$ and recall that $C^s([0,n]\to \C)$ is Polish. Let $\{f_{n,m} \mid m \in \N \}$ be a countable dense subset of $C^s([0,n] \to \C)$. We define
	\[
	\tilde{f}_{n,m}(r) = 
		\left\{\begin{array}{lc}
		f_{n,m}(r) & \hbox{ if } r \in [0,n]\\
		f_{n,m}(n) & \hbox{ if } r \in (n,\infty)
		\end{array}\right.
	\]
Then $\{\tilde{f}_{n,m} \mid n,m \in \N\}$ is a countable dense subset of $C^s_{loc}([0,\infty) \to \C)$. 

Also, a Cauchy sequence in $C^s_{loc}([0,\infty) \to \C)$ must also be Cauchy with respect to each semi-norm $\| \cdot \|_{C^s([0,n])}$. Completeness of $C^s_{loc}([0,\infty) \to \C)$ then follows from the completeness of $C^s([0,n] \to \C)$.

Note that convergence in the metric $d$ implies local uniform convergence, and hence point-wise convergence. Similarly to above, the cylinder $\sigma$-algebra is hence contained in the Borel $\sigma$-algebra induced by $d$. 

Arguing as above shows that, for $n \geq 1$, for $f_0 \in C^s_{loc}([0,\infty)\to\C)$ and for $\lambda > 0$,
	\begin{equation}\label{subbase}
	B_n(f_0,\lambda) := \{ f \in C^s_{loc}([0,\infty)\to \C) \mid \|f - f_0\|_{C^s([0,n])} < \lambda \}
	\end{equation}
is also cylinder measurable. As open sets of the form \eqref{subbase} constitute a sub-basis for the topology on $C^s_{loc}([0,\infty) \to \C)$ and this space is separable, it follows that every open ball (and hence every Borel set) is cylinder measurable. Finally, observe that all of the arguments also hold if we replace $\C$ by $\R$.
\end{proof}

To express the Kolmogorov theorem, we first recall a definition.

\begin{defn}\label{consistent}
Let $\Lambda = \C$ or $\R$. Let $I$ be an infinite index set, and for each finite sub-index $A \subseteq I$, let $P_A$ be some Borel probability measure on $\Lambda^{|A|}$. We say that the collection $\{P_A\}_{A\subseteq I, |A| < \infty}$ is a \textit{consistent family of finite dimensional distributions indexed on $I$} if, for every finite $A \subseteq I$ and every $r \in I \setminus A$, we have
	\[
	P_{A}(B) = P_{A \cup \{r\}} (B \times \Lambda)
	\]
for every Borel set $B \subseteq \Lambda^{|A|}$. 
\end{defn}

The proof of the Kolmogorov theorem be found in \cite{ReY} and in \cite{SVar}.

\begin{thm}[Kolmogorov Continuity and Consistency]\label{kolmogorov}
Let $\Lambda = \C$ or $\R$. Let $\{P_{r_1, \ldots, r_n}\}$ be some consistent family of finite dimensional distributions indexed on some interval $I \subseteq \R$. Then there exists a unique (cylinder) probability measure $P$ on $\Lambda^{I} = \{ f : I \to \Lambda\}$ such that for Borel sets $B_1, \ldots, B_n \subseteq \Lambda$ and for $r_1, \ldots, r_n \in I$,
	\[
	P(f(r_j) \in B_j, j=1,\ldots, n) = P_{r_1,\ldots,r_n} (B_1 \times \cdots \times B_n).
	\]
Let $\beta,\gamma > 0$. If, for each compact sub-interval $K \subseteq I$, there exists a $C_K < \infty$ such that for all $r,s \in K$,
	\[
	\E^{P} \left[|f(r) - f(s)|^\beta \right] \leq C_K |r - s|^{1 + \gamma},
	\]
then $P$ is supported on $C^0_{loc}(I \to \Lambda)$. Furthermore, for every $0 \leq s < \gamma/\beta$, $P$ is also supported on $C^{s}_{loc}(I \to \Lambda)$.
\end{thm}

In order to apply Kolmogorov's theorem to construct the infinite volume Gibbs measure, denoted $\nu_\infty$, we shall reduce the problem to computing the asymptotics of the fundamental solution of a certain parabolic PDE.

\begin{defn}\label{fundsolndefn}
Let $C_1(r,x)$, $C_2(r,x)$, and $C_3(r,x)$ be functions from $[0,\infty) \times \R$ to $\R$ and consider the equation
	\[
	L\phi := -\partial_r \phi + C_1(r,x) \partial_x^2 \phi + C_2(r,x) \partial_x \phi + C_3(r,x) \phi = 0.
	\] 
Let $\phi(r,x;s,y)$ be a function on the following domain: $x,y \in \R$, $s \geq 0$, and $r > s$. We say that $\phi$ is the \textit{fundamental solution of $L\phi =0$ at $(s,y)$} if it obeys both of the following conditions:
    \begin{enumerate}
    \item $\phi$ is continuously differentiable once in $r$ and twice in $x$ and satisfies, as a function of $r$ and $x$, the equation $L\phi =0$ (in the classical sense).
    \item $\lim_{r \downarrow s} \phi(r,x;s,y) = \delta_{x-y}$ as linear functionals on $C_0(\R)$: for $f \in C_0(\R)$,
			\begin{equation}\label{deltafunc}
			\lim_{r\downarrow s} \int_{\R} \phi(r,x;s,y) f(x) \; dx = f(y).
			\end{equation}
    \end{enumerate}
If $\phi(r,x;s,y)$ is the fundamental solution of $L\phi = 0$ at every $(s,y) \in \R^{\geq 0} \times \R$, then we simply say that \textit{$\phi$ is the fundamental solution of $L\phi = 0$}.
\end{defn}

For example, the heat kernel
		\begin{equation}\label{theheatker}
		\phi_0(r,x;s,y) := \frac{1}{\sqrt{2\pi(r-s)}} \exp\left(-\frac{(x-y)^2}{2(r-s)} \right)
		\end{equation}
is the fundamental solution of the heat equation $-\partial_r\phi_0 = \frac{1}{2} \partial_x^2 \phi_0$.

Recalling the definitions of $\nu_L, \nu_{L,1}$, and $\nu_{L,2}$ in \eqref{1dgibbs}, \eqref{1dgibbsreal}, and \eqref{1dgibbsimag}. We now state the main result of this section.

\begin{thm}\label{construct}
	\begin{enumerate}
	\item There exists a (strictly positive) function $\phi(r,x;s,y)$ which is the fundamental solution of
	\begin{equation}\label{parabpde}
	-\partial_r \phi + \frac{1}{2} \partial_{x}^2 \phi - \frac{1}{4} \frac{x^4}{r^2} \phi = 0
	\end{equation}
at each $(s,y) \in (0,\infty) \times \R$ and at $(s,y) = (0,0)$.
	\item For each $L > 1$, the measure $\nu_{L,1}$ obeys the following law: let $0 < r_1 < \cdots < r_N < L$ and let $B_1, \ldots, B_N \subseteq \R$ be Borel sets, then, with $\phi$ as above,
	\begin{align*}
	&\PP_{\nu_{L,1}}(f(r_j) \in B_j, j=1,\ldots, N) 
	\\
	&\hbox{\hskip 6pt}= \int_{B_1} \cdots \int_{B_N} \frac{\phi(L,0;r_N,x_N)}{\phi(L,0;0,0)} \prod_{j=2}^{N}\phi(r_j,x_j;r_{j-1},x_{j-1}) \phi(r_1,x_1;0,0) \; dx_n \cdots dx_1.
	\end{align*}
	\item There is a positive, bounded, continuous function $F(s,y): (0,\infty) \times \R \to \R$ such that, for fixed $s$,
	\[
	\lim_{L \to \infty} \left\|\frac{\phi(L,0;s,y)}{\phi(L,0;0,0)} - F(s,y)\right\|_{C_y^0} = 0.
	\]
	\item There exists a unique cylinder probability measure $\nu_{\infty, 1}$ on $C^0_{loc}([0,\infty) \to \R)$ such that for $0 < r_1 < \cdots < r_N $ and for Borel sets $B_1, \ldots, B_N \subseteq \R$, 
		\begin{align*}
		&\PP_{\nu_{\infty,1}}(f(r_j) \in B_j, j=1,\ldots, N)
		\\
		&\hbox{\hskip 8pt}= \int_{B_1} \cdots \int_{B_N} F(r_N,x_N) \prod_{j=2}^{N}\phi(r_j,x_j;r_{j-1},x_{j-1}) \phi(r_1,x_1;0,0) \; dx_N \cdots dx_1.
		\end{align*}
	Furthermore, for every $s \in [0,\tfrac{1}{2})$, $\nu_{\infty,1}$ is supported on $C^{s}_{loc}([0,\infty) \to \R)$.
	\item Fix $s \in [0,\frac{1}{2})$. Let $1 < R < L$ and let $\nu_{\infty,1}|_{[0,R]}$ and $\nu_{L,1}|_{[0,R]}$ denote the image measure (or push-froward measure) of $\nu_{\infty,1}$ and $\nu_{L,1}$, resp., on $C^s([0,R] \to \R)$ under the restriction map $f \mapsto f|_{[0,R]}$. Then $\nu_{\infty,1}|_{[0,R]}$ and $\nu_{L,1}|_{[0,R]}$ are mutually absolutely continuous, with Radon--Nikodym derivative
		\[
		\frac{d \nu_{L,1}|_{[0,R]}}{d \nu_{\infty,1}|_{[0,R]}}(f) 
			= 
		\frac{\phi(L,0;R,f(R))}{F(R,f(R)) \phi(L,0;0,0)}.
		\]
	For every Borel subset $A \subseteq C^{0}([0,R] \to \R)$, we have
		\begin{equation}\label{convergence1}
		\lim_{L \to \infty} \nu_{L,1}|_{[0,R]}(A) = \nu_{\infty,1}|_{[0,R]}(A).
		\end{equation}
	\item Let $W$ denote the Wiener measure on $C^s_{loc}([0,\infty) \to \R)$. Let $\nu_\infty$ be the Borel probability measure on $C^s_{loc}([0,\infty) \to \C)$ given by
	\begin{equation}\label{infvolmeasure}
	\nu_{\infty}(\{ g \mid \re(g) \in A_1, \im(g) \in A_2 \}) := \nu_{\infty,1}(A_1) W(A_2)
	\end{equation}
for Borel measurable sets $A_1,A_2 \subseteq C^s_{loc}([0,\infty) \to \R)$.

	Let $1 < R < L$. Then $\nu_\infty|_{[0,R]}$ and $\nu_L|_{[0,R]}$ are mutually absolutely continuous measures on $C^s([0,R] \to \C)$. For every Borel subset $A \subseteq C^s([0,R] \to \C)$, we have
	\begin{equation}\label{convergenceall}
	\lim_{L \to \infty} \nu_L|_{[0,R]}(A) = \nu_{\infty}|_{[0,R]}(A).
	\end{equation}
Furthermore, $\nu_\infty$ is the unique probability measure which obeys \eqref{convergenceall}.
	
	Let $\mathcal{F}_R$ denote the completion of the Borel $\sigma$-algebra on $C^s([0,R])$ with respect to any of these measures. Then \eqref{convergenceall} holds for every $A \in \mathcal{F}_R$. 
	\end{enumerate}
\end{thm}

We will break up the proof of Theorem \ref{construct} into several subsections. 

\subsection{Proof of Theorem \ref{construct}, part 1.}\label{existfundsoln}
The main difficulty in constructing a fundamental solution of \eqref{parabpde} lies with the coefficient $-\frac{1}{4}\frac{x^4}{r^2}$. It is neither bounded in $x$ nor is it uniformly H\"older continuous. For $x \neq 0$, this coefficient also goes to $-\infty$ point-wise as $r$ goes to $0$.

To handle these issues, we apply suitable cut-offs. For real numbers $a,b \in \R$, let 
	\[
	a \vee b := \max(a,b)
	\hbox{\hskip 18pt and \hskip 18pt}
	a \wedge b := \min(a,b).
	\] 
Let us consider the cut-off equations
    \begin{equation}\label{cutoff}
    L_n \phi : = -\partial_r \phi + \frac{1}{2} \partial_x^2 u - \frac{1}{4} \frac{ x^4 \wedge n}{r^2 \vee \frac{1}{n}} \phi = 0, \hbox{\hskip 28pt} n = 1,2,\ldots
    \end{equation}
For each $n \in \N$, it is not hard to see that the coefficient $-\frac{1}{4}\frac{x^4 \wedge n}{r^2 \vee 1/n}$ is bounded and uniformly H\"older continuous in $x$ (indeed, Lipschitz). By the parametrix method (cf., \cite[pg. 14--20]{Fri} or \cite{LSU}) there exists a there exists a unique fundamental solution $\phi_n(r,x;s,y)$ of the cut-off PDE $L_n \phi = 0$ at all $(s,y) \in [0,\infty) \times \R$ and obeys the Duhamel formula
	\begin{align}\label{cutoffduhamel}
	&\phi_n(r,x;s,y)  
		\\
		\nonumber
	&\hbox{\hskip 18pt} = \phi_0(r,x;s,y) - \frac{1}{4} \int_s^r \int_\R \phi_0(r,x;\rho,w) \frac{w^4 \wedge n}{\rho^2 \vee \frac{1}{n}} \phi_n(\rho,w;s,y) \; dw \; d\rho.
	\end{align}
Furthermore, $\phi_n > 0$ for each $n \in \N$. Note that $\phi_0$ denotes the standard heat kernel (cf., \eqref{theheatker}), which is consistent with \eqref{cutoff}.

\begin{lem}
Let $\phi_n$ be as above. Then
		\begin{equation}\label{decrease}
		\phi_0 \geq \phi_1 \geq \phi_2 \geq \cdots \geq 0.
		\end{equation}
Furthermore, the function 
	\begin{equation}\label{limit}
	\phi(r,x;s,y) := \lim_{n\to\infty} \phi_n(r,x;s,y)
	\end{equation}
is well defined and obeys the estimate
	\begin{equation}\label{gaussbound}
	0 \leq \phi(r,x;s,y)\leq \phi_0(r,x;s,y)
	\end{equation} 
as well as the Duhamel formula
	\begin{equation}\label{duhamel}
	\phi(r,x;s,y) = \phi_0(r,x;s,y) - \frac{1}{4} \int_s^r \int_\R \phi_0(r,x;\rho,w) \frac{w^4}{\rho^2} \phi(\rho,w;s,y) \; dw \; d\rho
	\end{equation}
for $(s,y) = (0,0)$ or $(s,y) \in (0,\infty) \times \R$. Moreover, $\phi$ is the fundamental solution of \eqref{parabpde} at $(s,y) = (0,0)$ and all $(s,y) \in (0,\infty) \times \R$. Finally, we have $\phi > 0$.
\end{lem}
\begin{proof}
Observe that $r^2\vee \frac{1}{n} \geq r^2\vee\frac{1}{n+1}$ and $x^4\wedge n \leq x^4 \wedge (n+1)$. It follows that, for $n \in \N$,
    $$
    - \frac{x^4 \wedge (n+1)}{r^2 \vee\frac{1}{n+1}} \leq -\frac{x^4\wedge n}{r^2 \vee \frac{1}{n}}.
    $$
By the comparison principle (cf., \cite[p. 45-46]{Fri}), we have \eqref{decrease}. As decreasing sequences that are bounded below must tend to a limit, the function $\phi$ given by \eqref{limit} is well-defined and clearly obeys \eqref{gaussbound}.

To establish the Duhamel formula, we first recall a heat semi-group-like identity. Let $B(x,y) = \frac{\Gamma(x)\Gamma(y)}{\Gamma(x+y)}$ be the beta function, and let $\lambda > 0$. For $-\infty < \alpha, \beta< \frac{3}{2}$, we recall that
		\begin{align}
		\label{semigroup}
		&\int_s^r \int_{\R} (r-\rho)^{-\alpha} \exp\left(-\frac{\lambda|x-w|^2}{2(r-\rho)}\right)(\rho - s)^{-\beta}\exp\left(-\frac{\lambda|w-y|^2}{2(\rho-s)}\right) \; dw \; d\rho
			\\
		\nonumber
    &\hbox{\hskip 18pt}
		= \left(\frac{2\pi}{\lambda}\right)^{1/2} B(\tfrac{3}{2}-\alpha, \tfrac{3}{2} - \beta) \cdot (r-s)^{\frac{3}{2} - \alpha - \beta} \exp\left(-\frac{\lambda|x-y|^2}{2(r-s)}\right).
			\end{align}
Indeed, the proof of this identity can be found in \cite[pg. 15]{Fri}.

We first consider the case $(s,y) = (0,0)$. For each $n$, \eqref{decrease} gives
	\begin{align}
	\label{origindamping}
  \frac{w^4\wedge n}{\rho^2 \vee \frac{1}{n}} \phi_n(\rho,w;0,0)
		&\lesssim
	\frac{1}{\rho^{1/2}}\left[\frac{w^4}{\rho^2} \exp\left(-\frac{w^2}{2\rho}\right) \right]
		\\
		\nonumber
		&\lesssim
	\frac{1}{\rho^{1/2}} \exp\left(-\frac{w^2}{8\rho}\right).
	\end{align}
Thus, \eqref{semigroup} gives
	\begin{align}\label{originduhamel}
	&\bigg|\int_0^r \int_\R  \phi_0(r,x;\rho,w) \frac{w^4 \wedge n}{\rho^2 \vee \frac{1}{n}} \phi_n(\rho,w;0,0) \; dw d\rho\bigg|  
		\\ \nonumber
	&\hbox{\hskip 18pt}\lesssim
	\int_0^r \int_\R \frac{\exp\left(-\frac{(x-w)^2}{8(r-\rho)}\right)}{(r-\rho)^{1/2}} \frac{\exp\left(-\frac{w^2}{8\rho}\right)}{\rho^{1/2}} \; dw d\rho
		\\ \nonumber
	&\hbox{\hskip 18pt}\lesssim \sqrt{r}\exp\left(-\frac{x^2}{8r}\right),
	\end{align}
where the implicit constants are independent of $n$. Thanks to \eqref{origindamping}, dominated convergence implies that the first integral in \eqref{originduhamel} converges as $n$ goes to $+\infty$.  Using \eqref{cutoffduhamel} and \eqref{limit}, we have
	\begin{align*}
	\phi(r,x;0,0)
		&=
	\lim_{n \to \infty} \phi_n(r,x;0,0) \\ 
		&=
	\phi_0(r,x;0,0) - \frac{1}{4} \int_0^r \int_\R \phi_0(r,x;\rho,w) \frac{w^4}{\rho^2} \phi(\rho,w;0,0) \; dw \; d\rho,
	\end{align*}
which establishes \eqref{duhamel} in the case $(s,y) = (0,0)$. 

We consider the case $(s,y) \in (0,\infty) \times \R$. For each $n \in \N$ and $\rho > s$, \eqref{decrease} gives
	\begin{align}
	\label{postimedamping}
		\frac{w^4\wedge n}{\rho^2 \vee \frac{1}{n}} \phi_n(\rho,w;s,y)
    	&\leq
    \frac{w^4}{\rho^2} \frac{1}{(\rho - s)^{1/2}} \exp\left(-\frac{(w-y)^2}{2(\rho - s)}\right)
			\\
			\nonumber
    	&\lesssim
    \frac{1+y^4}{(\rho - s)^{1/2}}\left(\frac{1}{\rho^2}+\frac{(w-y)^4}{(\rho-s)^2}\right) \exp\left(-\frac{(w-y)^2}{2(\rho - s)}\right) 
			\\
			\nonumber
    	&\lesssim_{s,y} 
		\frac{1}{(\rho - s)^{1/2}} \exp\left(-\frac{(w-y)^2}{8(\rho - s)}\right).
    \end{align}
Thus, applying \eqref{semigroup} gives
	\begin{equation}\label{positivetimeduhamel}
	\bigg|\int_s^r \int_\R \phi_0(r,x;\rho,w) \frac{w^4 \wedge n}{\rho^2 \vee \frac{1}{n}} \phi_n(\rho,w;s,y) \; dw \; d\rho  \bigg| \lesssim_{y,s} e^{-\frac{(x-y)^2}{8(r-s)}}\sqrt{r-s} .
	\end{equation}
Again, the estimate \eqref{positivetimeduhamel} is uniform in $n$. Using \eqref{cutoffduhamel}, \eqref{decrease}, and \eqref{limit}, dominated convergence gives
	\begin{align*}
	\phi(r,x;s,y)
		&=
	\lim_{n \to \infty} \phi_n(r,x;s,y) \\ 
		&=
	\phi_0(r,x;s,y) - \frac{1}{4} \int_0^r \int_\R \phi_0(r,x;\rho,w) \frac{w^4}{\rho^2} \phi(\rho,w;s,y) \; dw \; d\rho,
	\end{align*}
which establishes \eqref{duhamel} in the case $(s,y) \in (0,\infty) \times \R$.

An immediate corollary of the proof of \eqref{duhamel} is the delta function property \eqref{deltafunc} of fundamental solutions. Indeed, \eqref{duhamel}, \eqref{originduhamel}, and \eqref{positivetimeduhamel} imply
	\[
	\lim_{r \downarrow s} \phi(r,x;s,y) = \lim_{r \downarrow s} \phi_0(r,x;s,y) = \delta_{x-y}.
	\]

Next, we use \eqref{duhamel} to prove that $\phi$ is continuously differentiable twice in $x$ and once in $r$. Clearly, $\phi_0(r,x;s,y) = \frac{1}{\sqrt{2\pi (r-s)}} e^{-\frac{(x-y)^2}{2(r-s)}}$ is infinitely differentiable in every variable. Letting 
	\[
	D(r,x;s,y) := \frac{1}{4} \int_{s}^r \int_\R \phi_0(t,x;\rho,w) \frac{w^4}{\rho^2} \phi(\rho,w;s,y) \; dw \; d\rho
	\]
we claim that $D$ is continuously differentiable twice in $x$ and once in $r$, and
    \begin{align}
	\label{1stder}    
	\partial_x D(r,x;s,y) 
    	&= 
    \frac{1}{4} \int_{s}^r \int_\R \partial_x \phi_0(r,x;\rho,w) \frac{w^4}{\rho^2} \phi(\rho,w;s,y) \; dw \; d\rho  
    	\\
	\label{2ndder}    
	\partial_x^2 D(r,x;s,y) 
    	&= 
    \frac{1}{4} \int_{s}^r \int_\R \partial_x^2 \phi_0(r,x;\rho,w) \frac{w^4}{\rho^2} \phi(\rho,w;s,y) \; dw \; d\rho
    	\\
	\label{timeder}   
    \partial_r D(r,x;s,y) 
    	&= 
    \frac{1}{4}\frac{x^4}{r^2} \phi(r,x;s,y) + \frac{1}{2}\partial_x^2 D(r,x;s,y).
    \end{align}
If we accept this claim for now, then the relation $\phi = \phi_0 - D$ gives
	\[
	\partial_r \phi
		=
	\partial_r \phi_0 - \partial_r D 
		=
	\frac{1}{2}\partial_x^2 \phi_0 - \frac{1}{4}\frac{x^4}{r^2} \phi - \frac{1}{2}\partial_x^2 D
		= 
	\frac{1}{2}\partial_x^2 \phi - \frac{1}{4}\frac{x^4}{r^2} \phi,
	\]
which shows that $\phi$ is indeed the fundamental solution of \eqref{parabpde}, and hence finishes the proof of Theorem 2.1, part 1.

We first establish \eqref{1stder}. The mean-value theorem gives
	\begin{align}
	\label{meanval1stder}
	&\frac{D(r,x+h;s,y) - D(r,x;s,y)}{h} 
		\\
		\nonumber
	&\hbox{\hskip 38pt}= \frac{1}{4} \int_s^r \int_\R \partial_x \phi_0(r,x+\theta_h h;\rho,w)\frac{w^4}{\rho^2}\phi(\rho,w;s,y) \; dw \; d\rho
	\end{align}
for some $\theta_h \in [0,1]$. From \eqref{origindamping} and \eqref{postimedamping}, we have
	\begin{equation}\label{absorbfactor}
	\frac{w^4}{\rho^2} \phi(\rho,w;s,y) \lesssim_{s,y} \frac{1}{\left(\rho - s\right)^{1/2}} \exp\left(-\frac{(w-y)^2}{8(\rho - s)}\right).
	\end{equation}
Also, we have
	\begin{equation}\label{heatker1stder}
	|\partial_x \phi_0(r,x+\theta_h h;\rho,w)|
		\lesssim 
	\frac{1}{r-\rho} \exp\left(-\frac{(x+\theta_h h-w)^2}{8(r-\rho)}\right).
	\end{equation}
Combining \eqref{meanval1stder}, \eqref{absorbfactor}, \eqref{heatker1stder} together, and applying \eqref{semigroup},
	\begin{align}
	\label{1stderestimate}
	&\left\lvert \frac{D(r,x+h;s,y) - D(r,x;s,y)}{h} \right\rvert 
		\\ \nonumber
		&\hbox{\hskip 18pt}\lesssim_{s,y}
	\int_s^r \int_\R \frac{\exp\left(-\frac{(x+\theta_h h-w)^2}{8(r-\rho)}\right)}{r-\rho} \frac{\exp\left(-\frac{(w-y)^2}{8(\rho - s)}\right)}{\left(\rho - s\right)^{1/2}}   \; dw d\rho
		\\ \nonumber
		&\hbox{\hskip 18pt}\lesssim_{s,y} 
	\exp\left(-\frac{(x+\theta_h h - y)^2}{8(r-s)}\right).
	\end{align}
Since the right hand side of \eqref{1stderestimate} converges as $h \to 0$, the generalized dominated convergence theorem implies that $D$ is differentiable once in $x$ and satisfies \eqref{1stder}. The same proof also shows that $D$ is continuously differentiable in $x$. Furthermore, \eqref{heatker1stder} and  \eqref{1stderestimate} imply that, for each $R > 0$, 
		\begin{equation}
		\label{firstderfullestimate}
		\left\lvert \partial_x \phi(r,x;s,y) \right\rvert
			\lesssim_{R,s,y}
		\frac{\exp\left(-\frac{(x-y)^2}{8(r-s)}\right)}{r-s},
		\end{equation}
for $s < r \leq s+R$.

We now establish \eqref{2ndder}. We again apply the mean-value theorem to obtain
	\begin{align}
	\label{2nddermeanval}
	&\frac{\partial_x D(r,x+h;s,y) - \partial_x D(r,x;s,y)}{h}
		\\
		\nonumber	
	&\hbox{\hskip 16pt}= \frac{1}{4} \int_s^r \int_\R \partial_x^2 \phi_0(r,x+\theta_h h;\rho,w)\frac{w^4}{\rho^2}\phi(\rho,w;s,y) \; dw \; d\rho
	\end{align}
for some $\theta_h \in [0,1]$. Unfortunately, the estimate
	\begin{align}
	\label{heatker2ndder}
	&|\partial_x^2 \phi_0(r,x+\theta_h h;\rho,w)|  
		\\
		\nonumber
	&\hbox{\hskip 18pt}\lesssim (r-\rho)^{-\left(\frac{3}{2}-\beta\right)} |x+\theta_h h -w|^{-2\beta} \exp\left(-\frac{(x+ \theta_h h-w)^2}{8(r-\rho)}\right),
  \end{align}
for $\beta\geq 0$, cannot be easily used with \eqref{semigroup}, as the resulting singularity in $\rho$ turns out to not be integrable. Instead, we use integration by parts to move the singularity in $\rho$ to other factors: first recall that 
	\begin{align}
	\label{first}
	\partial_x \phi_0(r,x+\theta_h h;\rho,w) 
		&= 
	- \partial_w \phi_0 (r,x+\theta_h h;\rho,w)
		\\
	\label{second}
	\partial_x^2 \phi_0(r,x+\theta_h h;\rho,w) 
		&= 
	\partial_w^2 \phi_0 (r,x+\theta_h h;\rho,w)
	\end{align}
Using \eqref{second}, integrating by parts once in $w$, and then using \eqref{first} gives
	\begin{align}
	\label{ibpstep}
	&\int_s^r \int_\R \partial_x^2 \phi_0(r,x+\theta_h h;\rho,w)\frac{w^4}{\rho^2}\phi(\rho,w;s,y) \; dw \; d\rho =
		\\ \nonumber
	&\hbox{\hskip 36pt} \int_s^r \int_\R \partial_x \phi_0(r,x+\theta_h h;\rho,w)\frac{4w^3}{\rho^2}\phi(\rho,w;s,y) \; dw \; d\rho
		\\ \nonumber
	&\hbox{\hskip 48pt} + \int_s^r \int_\R \partial_x \phi_0(r,x+\theta_h h;\rho,w)\frac{w^4}{\rho^2} \partial_w\phi(\rho,w;s,y) \; dw \; d\rho
	\end{align}
Using \eqref{gaussbound}, \eqref{firstderfullestimate}, and arguing as in \eqref{postimedamping}, we have, for $s < \rho$,
	\begin{equation}\label{intbypartsbound}
	\left| \frac{4w^3}{\rho^2}\phi(\rho,w;s,y) \right| + \left|\frac{w^4}{\rho^2} \partial_w\phi(\rho,w;s,y) \right| \lesssim_{s,y} \frac{\exp\left(-\frac{(w-y)^2}{16(\rho-s)}  \right)}{\rho-s} .
	\end{equation}
Combining \eqref{heatker1stder}, \eqref{2nddermeanval}, \eqref{ibpstep}, \eqref{intbypartsbound} together, and then applying \eqref{semigroup}, 
	\begin{align}
	\label{2ndderestimate}
	&\left\lvert \frac{\partial_x D(r,x+h;s,y) - \partial_x D(r,x;s,y)}{h}\right\rvert
		\\ \nonumber
	&\hbox{\hskip 18pt} \lesssim_{s,y}
		\int_s^r\int_\R \frac{\exp\left(-\frac{(x+\theta_h h-w)^2}{16(r-\rho)}\right)}{r-\rho} \frac{\exp\left(-\frac{(w-y)^2}{16(\rho-s)}  \right)}{\rho-s} \; dw d\rho
		\\ \nonumber
	&\hbox{\hskip 18pt} \lesssim_{s,y} 
	\frac{1}{(r - s)^{1/2}} \exp\left(-\frac{(x + \theta_h h - y)^2}{16(r - s)}\right).
	\end{align}
Applying a similar generalized dominated convergence as above establishes \eqref{2ndder} and shows that $D$ is continuously differentiable twice in $x$. Furthermore, \eqref{heatker2ndder} and \eqref{2ndderestimate} imply that, for each $R>0$,
	\begin{equation}
	\label{secondderfullestimate}	
		\left\lvert \partial_x^2 \phi(r,x;s,y) \right\rvert
			\lesssim_{R,s,y}
		\frac{\exp\left(-\frac{(x-y)^2}{16(r-s)}\right)}{(r-s)^{3/2}},
	\end{equation}
for $s < r \leq s + R$.

Now, we establish \eqref{timeder}. For $h > 0$, the mean value theorem gives,
	\begin{align*}
	&\frac{D(r+h,x;s,y) - D(r,x;s,y)}{h}
		\\
	&\hbox{\hskip 28pt}= 
	\frac{1}{4}\int_\R \phi_0(r+h,x;r+\theta_h h,w) \frac{w^4 \phi(r+\theta_h h,w;s,y)}{(r+\theta_h h )^2}  \; dw \\
	&\hbox{\hskip 28pt}+ 
	\frac{1}{4}\int_s^r \int_\R \partial_r u_0(r+\theta_h h, x;\rho,w) \frac{w^4}{\rho^2} \phi(\rho,w;s,y) \; dw \; d\rho
	\end{align*}
for some $\theta_h \in [0,1)$. As $h \downarrow 0$, the first term converges to $\frac{1}{4}\frac{x^4}{r^2}\phi(r,x;s,y)$ by the delta function property of $\phi_0$. For second term, observe that $\partial_r \phi_0(r,x;s,y) = \frac{1}{2}\partial_x^2 \phi_0(r,x;s,y)$, and therefore the second term converges to $\frac{1}{2}\partial_x^2 D(r,x;s,y)$ as $h \downarrow 0$. A similar argument can be made for $h < 0$. Also, \eqref{timeder}, \eqref{absorbfactor}, and \eqref{2ndderestimate} imply that, for each $R > 0$,
	\begin{align}
		\label{timederfullestimate}	
		\left\lvert \partial_r \phi(r,x;s,y) \right\rvert
			&\lesssim_{R,s,y}
		\frac{\exp\left(-\frac{(x-y)^2}{16(r-s)}\right)}{(r-s)^{3/2}},
		\end{align}
for $s < r \leq s+R$.

Finally, $\phi > 0$ follows immediately from the maximum principle (c.f., \cite[p.~39]{Fri}). This finishes the final claim.
\end{proof}

\begin{rem}
Observe that, in the proof of the Duhamel formula \eqref{duhamel}, the Gaussian bounds are ineffective when $s =0$ and $y \neq 0$; in particular, the expression $-\frac{1}{4} \frac{x^4}{r^2} \phi(r,x;0,y)$ admits suitable bounds only when $y = 0$.
\end{rem}

\subsection{Proof of Theorem \ref{construct}, part 2.}\label{MTFK}

We revisit $\mu_{L,1}$ and $\nu_{L,1}$. As noted above, $\mu_{L,1}$ is the measure corresponding to the standard Brownian bridge from $r = 0$ to $r = L$ and $d\nu_{L,1}(f) = \frac{1}{Z_L} \exp\left(-\frac{1}{4} \int_0^L |f(r)|^4 r^{-2} \; dr \right) d\mu_{L,1}(f)$.

Recall that $\phi_n(r,x;s,y)$ is the fundamental solution of 
	\[
	-\partial_r \phi + \frac{1}{2} \partial_x^2 \phi - \frac{1}{4} \frac{x^4 \wedge n}{r^2 \vee (1/n)} \phi = 0.
	\]
Let
	\[
	dP_n(f) := \exp\left(-\frac{1}{4} \int_0^L \frac{(f(r))^4 \wedge n}{ r^2 \vee (1/n)} \; dr \right) d\mu_{L,1}(f)
	\]
be a Borel measure on $C^0([0,L] \to \R)$, not necessarily a probability measure. By the multi-time Feynman--Kac formula with respect to Brownian bridges (cf., Theorem \ref{multitimebridge}), $P_n$ obeys the following law: for Borel sets $B_1, \ldots, B_N \subseteq \R$ and for $0 < r_1 < r_2 < \cdots < r_N < L$, we have
	\begin{align}\label{truncatedfeynmankac}
	&P_n (f(r_j) \in B_j, j = 1, \ldots, N) \\
	\nonumber	
	&\hbox{\hskip 8pt} = \int_{B_1} \cdots \int_{B_n}\frac{\phi_n(L,0;r_N,x_N)}{\phi_0(L,0;0,0)}\prod_{j=1}^{N} \phi_n(r_j,x_j;r_{j-1},x_{j-1}) \; dx_N\cdots dx_1,
	\end{align}	
with $(x_0,r_0) := (0,0)$. Using \eqref{decrease} and \eqref{limit} and applying dominated convergence,
	\begin{align*}
	&\lim_{n \to \infty} RHS\eqref{truncatedfeynmankac}\\
	&\hbox{\hskip 8pt} = \int_{B_1}\cdots \int_{B_n}\frac{\phi(L,0;r_N,x_N)}{\phi_0(L,0;0,0)}\prod_{j=1}^{N} \phi(r_j,x_j;r_{j-1},x_{j-1}) \; dx_N\cdots dx_1.
	\end{align*}
Another application of dominated convergence gives
	\[
	\lim_{n \to \infty} LHS\eqref{truncatedfeynmankac} = P( f(r_j) \in B_j, j = 1, \ldots, N)
	\]
where
	\[
	dP(f) := \exp\left(-\frac{1}{4} \int_0^L \frac{(f(r))^4}{r^2} \; dr \right) d\mu_{L,1}(f).
	\]
Let $Z_L := \frac{\phi(L,0;0,0)}{\phi_0(L,0;0,0)}$ be a normalization constant, which is non-zero because of Theorem~\ref{construct}, Part 1 and because $\phi_0(L,0;0,0) = (2\pi L)^{-\frac{1}{2}}$. Then we have 
	\[
	\frac{1}{Z_L} dP(f) = d\nu_{L,1}(f),
	\]
which obeys the desired multi-time law. 

\subsection{Proof of Theorem \ref{construct}, part 3.}\label{asymptote}

To compute $\lim_{L \to \infty} \frac{\phi(L,0;s,y)}{\phi(L,0;0,0)}$, we seek to compute the asymptotics of each factor separately. The main obstruction is that the coefficient $-\frac{1}{4} \frac{x^4}{r^2}$ gives a quartic restoring force that decays with time, and contributes significantly to the asymptotics: it turns out that $\phi$ is neither of polynomial decay in $r$ (cf., heat kernel) nor exponential decay (cf., Mehler kernel).

To handle these issues, we change variables to remove $r$-dependence from the significant terms: the function
	\begin{equation}\label{changeofvar}
	\Phi(r, x; s,y) := \tfrac{s}{3} \phi\left(\tfrac{r^3}{27}, \tfrac{xr}{3};\tfrac{s^3}{27}, \tfrac{ys}{3}\right)
	\end{equation}
is the fundamental solution of
	\begin{equation}\label{changeparabpde}
	-\partial_r \Phi + \frac{1}{2} \partial_x^2 \Phi - \frac{1}{4}x^4 \Phi + \frac{x}{r} \partial_x \Phi = 0
	\end{equation}
at each $(s,y) \in (0,\infty) \times \R$. 

At this point, we set up the separation of variables. Let 
	\[
	H := -\tfrac{1}{2}\partial_x^2 + \tfrac{1}{4}x^4,
	\] 
This is an essentially self--adjoint operator with a discrete spectrum (cf., \cite[Section~5.14]{Titch}). By Sturm--Liouville theory, $H$ has simple eigenvalues, which we list as
	\[
	\lambda_0 < \lambda_1 < \cdots < \lambda _k < \cdots
	\]
Furthermore, each $\psi_k$ (eigenfunction of $H$ corresponding to $\lambda_k$) is Schwartz and $\psi_0$ is sign-definite. Without loss of generality, $\psi_0$ is positive. 

Recall that the eigenvalues of the harmonic oscillator $H_0 = -\frac{1}{2}\partial_x^2 + \frac{1}{2}x^2$ are $k + \frac{1}{2}$, $k \geq 0$. By the min-max principle (cf., \cite[Ch. XIII]{ReS4}) and the fact that $\frac{1}{4} x^4 \geq \frac{1}{2} x^2 - \frac{1}{4}$, we have 
	\begin{equation}\label{eigvalgood}
	\lambda_k \geq k + \tfrac{1}{4}.
	\end{equation}
As usual, the function 
	\[
	e^{-(r-s)H}(x,y) := \sum_{k=0}^\infty e^{-(r-s)\lambda_k} \psi_k(x) \psi_k(y)
	\]
is the fundamental solution of
	\[
	-\partial_r f + \tfrac{1}{2} \partial_x^2 f - \tfrac{1}{4}x^4 f = 0.
	\]

It turns out that the first-order term $\frac{x}{r} \partial_x \Phi$ still gives a large contribution to the asymptotics. To handle this term, first observe that 
	\[
	x \partial_x = -\tfrac{1}{2} + \left( x \partial_x + \tfrac{1}{2} \right)
	\]
is the decomposition of $x \partial_x$ into its self-adjoint and anti-self-adjoint parts. Rewrite \eqref{changeparabpde} as
	\begin{align*}
	0
		&=
	- \partial_r \Phi + \frac{1}{2} \partial_x^2 \Phi - \frac{1}{4}x^4 \Phi + \frac{x}{r} \partial_x \Phi\\ 
		&= 
	- \partial_r \Phi + \left(\frac{1}{2}\partial_x^2 - \frac{1}{4}x^4 - \frac{1}{2r} \right) \Phi + \left(\frac{x}{r} \partial_x + \frac{1}{2r} \right) \Phi
	\end{align*}
Note that $(s/r)^\frac{1}{2} e^{-(r-s)H}(x,y)$ is the fundamental solution of
	\[
	- \partial_r f + \left(\frac{1}{2}\partial_x^2 - \frac{1}{4}x^4 - \frac{1}{2r} \right) f = 0.
	\]
With this in mind, the corresponding Duhamel formula is
	\begin{align}
	\label{changeduhamel}
	\Phi(r,x;s,y) 
		&= 
	\left(\frac{s}{r}\right)^\frac{1}{2} e^{-(r-s)H}(x,y) 
		\\ 
		\nonumber 
		&\hbox{\hskip 3pt}
		+ 
	\int_s^r\int_\R \left(\frac{\rho}{r}\right)^\frac{1}{2} e^{-(r-\rho)H}(x,w) \left[ \frac{w}{\rho}\partial_w + \frac{1}{2\rho}\right] \Phi(\rho,w;s,y) \; dw d\rho.
	\end{align}
This turns out to the correct setting to compute the asymptotics of $\Phi$. The main result of this sub-section is the following proposition. 

\begin{prop}\label{asymptoticsmassless}
For every $(s,y) \in (0,\infty) \times \R$,
	\begin{equation}\label{asympt}
	\lim_{r \to \infty} \left(\frac{r}{s}\right)^\frac{1}{2} e^{(r-s)\lambda_0} \Phi(r,0;s,y) = G(s,y) \psi_0(0),
	\end{equation}
where 
	\[
	G(s,y) = \psi_0(y) + \int_s^\infty \int_\R \left(\frac{\rho}{s}\right)^\frac{1}{2} e^{-(s-\rho)\lambda_0} \psi_0(w) \left[\frac{1}{2\rho}+ \frac{w}{\rho}\partial_w\right]\Phi(\rho,w;s,y) \; dw \; d\rho
	\]
and obeys $0 \leq G(s,y) \lesssim \psi_0(y) + s^{-\frac{1}{2}}$. There exists an $M$ such that for all $s \geq M$, $G(s,y)$ is strictly positive when $|y| < 1$. For fixed $s$, the convergence is uniform in $y$; in particular, $G$ is continuous in $y$. 
\end{prop}

Assuming Proposition \ref{asymptoticsmassless} is valid, let us finish the proof of Theorem~\ref{construct}, part~3. Let $M$ be as above, and, for $s \geq 0$, let 
	\[
	N := \max(s+1,\tfrac{M^3}{27}).
	\]
For $L > N$, inverting the change of variables in \eqref{changeofvar} gives
	\[
	\phi(L,0;N,y) 
		= 
	N^{-\frac{1}{3}} \Phi\left(3L^\frac{1}{3},0;3N^\frac{1}{3}, yN^{-\frac{1}{3}}\right).
	\]
Also, recall the identity
	\[
	\phi(L,0;s,y) = \int_\R \phi(L,0;N,w) \phi(N,w;s,y) \; dw.
	\]
Because $0 < \phi(N,w;s,y) \leq \frac{1}{\sqrt{2\pi (N-s)}} e^{-\frac{(w-y)^2}{2(N-s)}}$, the following estimate is independent of $s$ and $y$:
	\begin{equation}\label{unifL1}
	\int_{\R} \phi(N,w;s,y) \; dw = \left\|\phi(N,w;s,y) \right\|_{L^1_w} \leq 1,
	\end{equation}
and so Proposition~\ref{asymptoticsmassless} and H\"older's inequality gives
	\begin{align}
	\nonumber
	\lim_{L \to \infty} L^\frac{1}{6} e^{3\lambda_0 L^{1/3}} \phi(L,0;s,y)
		&=
	\lim_{L \to \infty} \int_\R  L^\frac{1}{6} e^{3\lambda_0 L^{1/3}} \phi(L,0;N,w) \phi(N,w;s,y) \; dw 
		\\	
		\label{IN}
		&\approx_N
	\int_\R G\left(3N^\frac{1}{3}, wN^{-\frac{1}{3}}\right) \phi(N,w;s,y) \; dw.
	\end{align}	
For fixed $s$, the convergence is uniform in $y$, by Proposition~\ref{asymptoticsmassless} and by \eqref{unifL1}. The limit is finite because, by \eqref{unifL1} and the fact that $\psi_0$ is Schwartz,
	\[
	RHS\eqref{IN}
		\lesssim
	\int_\R\left [\psi_0\big(wN^{-\frac{1}{3}}\big) + 1\right] \phi(N,w;s,y) \; dw 
		\lesssim
	1,
	\]
Furthermore, 
	\begin{align*}
	RHS\eqref{IN}
		&\geq
	\int_{|w| < N^{1/3}} G\left(3N^\frac{1}{3}, wN^{-\frac{1}{3}}\right) \phi(N,w;s,y) \; dw,
	\end{align*}
hence, by the positivity aspect of Proposition \ref{asymptoticsmassless} and the fact that $\phi > 0$ (cf., Theorem~\ref{construct}, Part 1), the limit is also strictly positive. In particular, 
	\[
	\lim_{L \to \infty} L^\frac{1}{6} e^{3\lambda_0L^{1/3}} \phi(L,0;0,0) = C >0.
	\] 
It follows that 
	\[
	F(s,y) 
		:= 
	\lim_{L \to \infty} \frac{\phi(L,0;s,y)}{\phi(L,0;0,0)} 
		=
	C^{-1} s^{\frac{1}{6}} e^{\lambda_0 s^{1/3}} \psi_0(0) G\big(3s^{\frac{1}{3}}, ys^{-\frac{1}{3}}\big)
	\] 
is a well-defined, strictly positive function that is bounded in $s$ and $y$. For fixed $s$, the convergence is also uniform in $y$ and so $F$ is continuous in $y$. This finishes the proof of Theorem \ref{construct}, part 3.

We now focus on the proof of Proposition \ref{asymptoticsmassless}. We will compute the asymptotics of the two terms in \eqref{changeduhamel} separately.

\begin{lem}\label{quarticdecay}
For every $x,y \in \R$ and $s > 0$, 
	\[
	\lim_{r \to \infty} \left(\frac{r}{s}\right)^\frac{1}{2}e^{(r-s)\lambda_0}\left[\left(\frac{s}{r}\right)^\frac{1}{2} e^{-(r-s)H}(x,y)  \right]= \psi_0(x) \psi_0(y).
	\]
For fixed $s$, the convergence is uniform in $x$ and $y$.
\end{lem}
\begin{proof}
The identity $(-\frac{1}{2}\partial_x^2 + \frac{1}{4}x^4) \psi_k = \lambda_k \psi_k$ gives $\int_\R \frac{1}{2}|\partial_x \psi_k|^2 + \frac{1}{4}|x^2\psi_k|^2 = \lambda_k$, and so $\|\partial_x \psi_k\|_2 \lesssim \sqrt{\lambda_k}$. Observe that 
	\[
	\|\psi_k^2\|_\infty \leq \|\partial_x(\psi_k^2)\|_1 \leq \|\psi_k\|_2 \|\partial_x \psi_k\|_2 = \|\partial_x \psi_k\|_2,
	\] 
and so
	\begin{equation}\label{supmassless}
	\|\psi_k\|_\infty \lesssim (\lambda_k)^\frac{1}{2}.
	\end{equation}
Thus, for all $x,y\in \R$ and for all $r \geq s+1$, the fact that $\lambda_k \geq k$ (cf., \eqref{eigvalgood}) gives
	\begin{align*}
	\left\lvert \sum_{k=1}^\infty e^{-(r-s)\lambda_k} \phi_k(x) \phi_k(y) \right\rvert 
		&\lesssim 
	e^{-(r-s)\lambda_1} \sum_{k=1}^\infty \lambda_k e^{-(r-s)(\lambda_k - \lambda_1)} 
		\lesssim 
	e^{-(r-s)\lambda_1}.
	\end{align*}
For fixed $s$, it follows that
	\[
	\lim_{r \to \infty} e^{(r-s)\lambda_0} \sum_{k=1}^\infty e^{-(r - s)\lambda_k} \psi_k(x) \psi_k(y) = 0,
	\]
uniformly in $x$ and $y$, which in turn gives the result.
\end{proof}

The asymptotics for the other term in \eqref{changeduhamel} is significantly more delicate and requires several sets of additional, a priori, estimates, which we call short term and long term estimates. 

\subsubsection{Short Term Estimates}
We use Gaussian bounds to obtain rational function bounds in $r,s$. The goal is to obtain bounds so that integrating various expressions in $r$ from $s$ to $s+1$ is finite. For example, \eqref{gaussbound} and changing variables give
	\[
	\Phi(r,x;s,y) \lesssim \frac{s}{\sqrt{r^3 - s^3}}\exp\left(-\frac{3(xr - ys)^2}{2(r^3 - s^3)}\right),
	\]
which gives
	\begin{equation}\label{gaussianmasslesstransform}
	\|\Phi(r,x;s,y)\|_{L_x^2} 
		\lesssim 
	\frac{s}{r^\frac{1}{2}(r^3 - s^3)^\frac{1}{4}}.
	\end{equation}
Another application the comparison principle gives the bound
	\begin{equation}\label{quarticgaussian}
	e^{-(r-s)H}(x,y) \lesssim (r-s)^{-\frac{1}{2}} \exp\left(-\frac{(x-y)^2}{2(r-s)}\right),
	\end{equation}
which implies
	\begin{equation}\label{quarticshortterm}
	\|e^{-(r-s)H}(x,y) \|_{L_x^2} \lesssim \frac{1}{(r-s)^\frac{1}{4}} \hbox{\hskip 16pt and \hskip 16pt}  \|e^{-(r-s)H}(x,y) \|_{L_y^2} \lesssim \frac{1}{(r-s)^\frac{1}{4}}.
	\end{equation}
	
To handle the terms with derivatives, we have the following result.	
	
\begin{lem}\label{quarticderivativeshortterm}
Let $0 < s < r \leq s+1$. Then, for all $x,y \in \R$,
	\[
	\|x \partial_x e^{-(r-s)H}(x,y)\|_{L_x^2} 	\lesssim \frac{1 + |y|^5}{(r-s)^\frac{3}{4}}
		\hbox{\hskip 8pt and \hskip 8pt}
	\|y \partial_y e^{-(r-s)H}(x,y)\|_{L_y^2} 	\lesssim \frac{1 + |x|^5}{(r-s)^\frac{3}{4}}
	\]
\end{lem}
\begin{proof}
Note that $e^{-(r-s)H}(x,y) = \sum_{k=0}^\infty e^{-(r-s)\lambda_k} \psi_k(x) \psi_k(y)$ is symmetric in $x$ and $y$, so it suffices to establish the result for $\|x \partial_x e^{-(r-s)H}(x,y) \|_{L_x^2}$. 

Using the Duhamel formula
	\[
	e^{-(r-s)H}(x,y) = \phi_0(r,x;s,y)-\frac{1}{4}\int_s^r \int_\R \phi_0(r,x;s,w) w^4 e^{-(\rho-s)H}(w,y) \; dw \; d\rho,
	\]
then, a similar computation to \eqref{1stder} gives
	\begin{align*}
	\partial_x e^{-(r-s)H}(x,y) 
		&= 
	\partial_x \phi_0(r,x; s,y)
		\\ 
		&\;\;\;\;\;\;\;\;\;\;\;\;
		-\frac{1}{4}\int_s^r \int_\R \partial_x \phi_0(r,x; s,w) w^4 e^{-(\rho-s) H}(w,y) \; dw \; d\rho\\
	&=: (A) + (B).
	\end{align*}
It is not hard to see that
	\[
	|(A)| 
		\lesssim 
	\frac{x-y}{(r-s)^\frac{3}{2}} e^{-\frac{(x-y)^2}{2(r-s)}} 
		\lesssim 
	\frac{1}{r-s}e^{-\frac{(x-y)^2}{4(r-s)}}.
	\]
By Lemma~\ref{semigroup}, \eqref{quarticgaussian}, and the hypothesis $s \leq \rho \leq r \leq s+1$ (in particular, $1 \leq \frac{1}{\sqrt{\rho - s}}$), we have 
	\begin{align*}	
	|(B)|
		&\lesssim
	\int_s^r \int_\R \frac{x - w}{(r-s)^\frac{3}{2}} e^{-\frac{(x - w)^2}{2(r-\rho)} }  \frac{w^4}{\sqrt{\rho-s}} e^{-\frac{(w - y)^2}{2(\rho - s)}} \; dw \; d\rho
		\\
		&\lesssim
	\int_s^r \int_\R \frac{x- y - w}{(r-s)^\frac{3}{2}} e^{-\frac{(x-y - w)^2}{2(r-\rho)}} \frac{(1+y^4)(1+w^4)}{\sqrt{\rho - s}} e^{-\frac{w^2}{2(\rho - s)}} \; dw \; d\rho
		\\	
		&\lesssim
		(1+y^4)\int_s^r \int_\R \frac{x- y - w}{(r-s)^\frac{3}{2}} e^{-\frac{(x-y - w)^2}{2(r-\rho)}} \frac{1}{\sqrt{\rho - s}} e^{-\frac{w^2}{2(\rho - s)}} \; dw \; d\rho \\
		&\hbox{\hskip 18pt} + (1+y^4)\int_s^r \int_\R \frac{x- y - w}{(r-s)^\frac{3}{2}} e^{-\frac{(x-y - w)^2}{2(r-\rho)}} \frac{1}{\sqrt{\rho - s}} \left[\frac{w^4}{(\rho - s)^2} e^{-\frac{w^2}{2(\rho - s)}}\right] \; dw \; d\rho \\
		&\lesssim
	(1+y^4) \int_s^r \int_\R \frac{1}{r-s} e^{-\frac{(x-y - w)^2}{4(r-\rho)}} \frac{1}{\sqrt{\rho - s}} e^{-\frac{w^2}{4(\rho - s)}} \; dw \; d\rho
		\\
		&\lesssim
	(1+y^4) e^{-\frac{(x-y)^2}{4(r-s)}}.
	\end{align*}
Combining both estimates gives 
	\[
	|\partial_x e^{-(r-s)H}(x,y)| \lesssim \frac{1+y^4}{r-s} e^{-\frac{(x-y)^2}{4(r-s)}},
	\]
and so
	\begin{align*}
	\left\|x\partial_x e^{-(r-s)H}(x,y)\right\|_{L_x^2}^2
		&\lesssim
	\frac{1+y^8}{(r-s)^2}\int_\R x^2 \exp\left(-\frac{(x-y)^2}{2(r-s)}\right) \; dx 
		\\
		&\lesssim
	\frac{1+y^{10}}{(r-s)^2}\int_\R (1+x^2) \exp\left(-\frac{x^2}{2(r-s)}\right) \; dx \\ 
		&\lesssim
	\frac{1+y^{10}}{(r-s)^\frac{3}{2}}. \qedhere
	\end{align*}
\end{proof}

\subsubsection{Long Term Estimates} Here, we seek exponential decay estimates in $r-s$ whenever $r \geq s+1$.

\begin{lem}\label{decaymassless}
Let $(s,y) \in [1,\infty) \times \R$ be fixed. For all $r \geq s+1$, 
	\[
	\|\Phi(r,x; s,y) \|_{L_x^2} 
		\lesssim
	\left(\frac{s}{r}\right)^\frac{1}{2} e^{-\lambda_0 (r-s)}.
	\]
Furthermore, let $P_0^\perp$ denote orthogonal projection onto $(\Span(\psi_0))^\perp$. For all $r \geq s + 1$,
	\[
	\| [P_0^\perp \Phi](r,x;s,y)\|_{L_x^2} \lesssim \left(\frac{s}{r}\right)^\frac{1}{2} e^{-\lambda_1(r-s)}.
	\]
\end{lem}
\begin{proof}
Observe that 
	\begin{align*}
	\partial_x \Phi(r,x;s,y) &= \tfrac{rs}{3}\phi_x\left(\tfrac{r^3}{27},\tfrac{xr}{3};\tfrac{s^3}{27},\tfrac{ys}{3}\right)\\
	\partial_x^2 \Phi(r,x;s,y) &= \tfrac{r^2s}{9} \phi_{xx}\left(\tfrac{r^3}{27},\tfrac{xr}{3};\tfrac{s^3}{27},\tfrac{ys}{3}\right)\\
	\partial_r \Phi(r,x;s,y) &= \tfrac{r^2s}{9} \phi_r\left(\tfrac{r^3}{27},\tfrac{xr}{3};\tfrac{s^3}{27},\tfrac{ys}{3}\right) + \tfrac{xs}{3}\phi_x\left(\tfrac{r^3}{27},\tfrac{xr}{3};\tfrac{s^3}{27},\tfrac{ys}{3}\right)
	\end{align*}
By \eqref{gaussbound}, \eqref{firstderfullestimate}, \eqref{secondderfullestimate}, and \eqref{timederfullestimate}, each of the functions
	\[
	\partial_r \Phi, \;\;
	\partial_x^2 \Phi, \;\; 
	x\partial_x \Phi, \;\; \hbox{ and } \;\;
	x^4\Phi, 
	\]
obey Gaussian bounds in $x$ (with coefficients depending on $r,s,y$). In particular, each of the functions are in $L_x^2$.

Writing $\langle \cdot, \cdot \rangle$ for the $L^2_{x}$ inner product, then
    \[
    \partial_r \langle \Phi, \Phi \rangle
    	=
    2 \langle \partial_r\Phi, \Phi\rangle
    	=
    2 \langle \tfrac{1}{2}\partial_x^2 \Phi - \tfrac{1}{4}x^4 \Phi,\Phi \rangle + 2\left\langle \frac{x}{r}\partial_x\Phi, \Phi \right\rangle 
    \]
Observe that, for fixed $r$, 
	\begin{align}
	\label{subsol1}
	\left\langle(\tfrac{1}{2}\partial_x^2 - \tfrac{1}{4}x^4) \Phi , \Phi \right\rangle
		&=
	-\Big\langle \sum_{k=0}^\infty \langle \Phi, \psi_k\rangle \lambda_k \psi_k , \sum_{k=0}^\infty \langle\Phi,\psi_k\rangle \psi_k \Big\rangle
		\\ \nonumber
		&=
	-\sum_{k=0}^\infty \lambda_k \langle \Phi, \psi_k\rangle^2  
		\\ \nonumber
		&\leq
	-\sum_{k=0}^\infty \lambda_0 \langle \Phi, \psi_k\rangle^2
		=
	-\lambda_0 \langle\Phi,\Phi \rangle
	\end{align}
and that
	\begin{equation}\label{subsol2}
	2\left\langle \frac{x}{r}\partial_x\Phi, \Phi \right\rangle 
		=
	\left\langle \frac{x}{r}\partial_x\Phi, \Phi \right\rangle - \frac{1}{r}\left\langle \Phi, (1+ x\partial_x) \Phi \right\rangle
		=
	-\frac{1}{r} \langle \Phi , \Phi \rangle.
	\end{equation}
Combining \eqref{subsol1} and \eqref{subsol2}, we see that, as a function of $r$, $\langle \Phi, \Phi \rangle$ is a subsolution of the ODE $\partial_r f = -\left(2\lambda_0 + \frac{1}{r}\right)f$. At $r = s+1$, we have the initial condition $\|\Phi(s+1,x;s,y)\|^2_{L^2_x}$, which is uniformly bounded in $s$ and $y$ by \eqref{gaussianmasslesstransform}. Therefore,
	\[
	\langle \Phi, \Phi \rangle 
		\leq 
	\|\Phi(s+1,x;s,y)\|^2_{L^2_x} \left( \frac{s+1}{r} e^{-2\lambda_0 (r-(s+1))} \right)
		\lesssim 
	\frac{s}{r} e^{-2\lambda_0 (r-s)},
	\]
where we used the bound $s \geq 1$ to conclude $s + 1 \lesssim s$. The result for $P_0^\perp \Phi$ follows from the fact that we may write all the sums beginning at $k=1$, and then use $\lambda_1$ in place of $\lambda_0$.
\end{proof}

\begin{lem}
Let $s,y$ be fixed. For all $r \geq s+1$, 
	\[
	\|e^{-(r-s)H}(x,y) \|_{L_x^2} 
		\lesssim 
	e^{-\lambda_0 (r-s)}
		\hbox{\hskip 18pt and \hskip 18pt}
	\|x \partial_x e^{-(r-s)H}(x,y) \|_{L_x^2} 
		\lesssim 
	e^{-\lambda_0 (r-s)}
	\]
\end{lem}
\begin{proof}
Indeed, for $r \geq s+1$, \eqref{supmassless} gives
	\begin{align*}
	\| e^{-(r-s)H}(x,y) \|_{L_x^2} 
		&= 
	\sum_{k=0}^\infty e^{-(r-s)\lambda_k} |\psi_k(y)| \cdot \|\psi_k(x)\|_{L^2_x} \\
		&\lesssim
	e^{-\lambda_0(r-s)} \sum_{k=0}^\infty e^{-(r-s)(\lambda_k - \lambda_0)} (\lambda_k)^\frac{1}{2} \\
		&\lesssim
	e^{-\lambda_0(r-s)}.
	\end{align*}

The identity $-\tfrac{1}{2}\partial_x^2 \psi_k + \tfrac{1}{2}x^4 \psi_k = \lambda_k \psi_k$ implies the inequality 
	\[
	-\tfrac{1}{2}x^2\psi_k\partial_x^2 \psi_k \leq \lambda_k x^2 \psi^2_k
	\] 
Integrating both sides by parts gives
	\[
	\|x \partial_x \psi_k\|_2^2 
		\lesssim 
	\lambda_k \int x^2 \psi^2_k - \int (\partial_x \psi_k) (x\psi_k)
		\lesssim
	\lambda_k \|x\psi_k\|_2^2 + \|\partial_x \psi_k\|_2 \|x \psi_k\|_2.
	\]	
Furthermore $\|\partial_x \psi_k\|_2 \lesssim (\lambda_k)^\frac{1}{2}$ and $\|x \psi_k\|_2 \leq \|\psi_k\|_2 + \|x^2 \psi_k\|_2 \lesssim (\lambda_k)^\frac{1}{2}$. It follows that 
	\begin{equation}\label{xdereigfunc}
	\|x\partial_x \psi_k\|_2 \lesssim \lambda_k
	\end{equation}
A similar computation as above gives the second result. 
\end{proof}

At this point, we have all the necessary short term and long term estimates.

\begin{proof}[Proof of Proposition \ref{asymptoticsmassless}]
Recall the Duhamel formula,
	\[
	\Phi(r,0;s,y) 
		= 
	\left(\frac{s}{r}\right)^\frac{1}{2}e^{-(r-s)H}(0,y) 
	+ \int_s^r \int_\R \left(\frac{\rho}{r}\right)^\frac{1}{2} \sum_{k=0}^\infty J_k(\rho,w;s,y)  \; dw \; d\rho,
	\]
with
	\[
	J_k(\rho,w;s,y) := e^{-(r-\rho)\lambda_k} \psi_k(0) \psi_k(w) \left[\frac{1}{2\rho}+ \frac{w}{\rho}\partial_w\right]\Phi(\rho,w;s,y).
	\]
In view of Lemma \ref{quarticdecay}, we first seek to show that
	\[
	\lim_{r \to \infty} \left(\frac{r}{s}\right)^\frac{1}{2} e^{(r-s)\lambda_0} \int_s^r \int_\R \left(\frac{\rho}{r}\right)^\frac{1}{2}\sum_{k=1}^\infty J_k(\rho,w;s,y) \; dw \; d\rho = 0,
	\]
which is to say, that the higher eigenvalues do not contribute to the asymptotic. 

Note that $(\rho/r)^{1/2}\sum_{k=1}^\infty e^{-(r-\rho)\lambda_k} \psi_k(x)\psi_k(w)$ develops a singularity as $\rho$ goes to $r$, and that $\left[\frac{1}{2\rho}+ \frac{w}{\rho}\partial_w\right]\Phi(\rho,w;s,y)$ also develops a singularity as $\rho$ goes to $s$. We split the integral 
	\[
	\int_s^r \int_\R \left(\frac{\rho}{r}\right)^\frac{1}{2}\sum_{k=1}^\infty J_k(\rho,w;s,y)\; dw \; d\rho
	=
	I_1(r;s,y) + I_2(r;s,y) + I_3(r;s,y)
	\] 
into three parts, with
	\begin{align*}
	I_1(r;s,y) := \int_s^{s+1} \int_\R \left(\frac{\rho}{r}\right)^\frac{1}{2}\sum_{k=1}^\infty J_k(\rho,w;s,y)\; dw \; d\rho
		\\
	I_2(r;s,y) := \int_{s+1}^{r-1} \int_\R \left(\frac{\rho}{r}\right)^\frac{1}{2}\sum_{k=1}^\infty J_k(\rho,w;s,y)\; dw \; d\rho 
		\\
	I_3(r;s,y) := \int_{r-1}^r \int_\R \left(\frac{\rho}{r}\right)^\frac{1}{2}\sum_{k=1}^\infty J_k(\rho,w;s,y)\; dw \; d\rho
	\end{align*}
and consider the asymptotics of each part separately.

Before analyzing these integrals, we first record a useful estimate on $\int \sum_{k} J_k dw$. Since $\frac{1}{2} + w\partial_w$ is anti-self-adjoint, we obtain
	\[
	\int_\R \psi_k(w)\left(\tfrac{1}{2}+ w\partial_w\right)\Phi(\rho,w;s,y) \; dw
			= 
	\int_\R \left[-\left(\tfrac{1}{2}+ w\partial_w\right) \psi_k(w)\right]\Phi(\rho,w;s,y) \; dw
	\]
Applying Cauchy--Schwarz, \eqref{eigvalgood}, and \eqref{xdereigfunc},
	\[
	\left\lvert\int_\R \psi_k(w)\left(\tfrac{1}{2}+ w\partial_w\right)\Phi(\rho,w;s,y) \; dw\right\rvert
		\lesssim
	\lambda_k \|\Phi(\rho,w;s,y)\|_{L_w^2},
	\]
Using the previous estimate and the definition of $J_k(\rho,w)$, 
	\begin{align*}
	\left\lvert\int_\R \sum_{k=1}^\infty J_k(\rho,w;s,y) \; dw \right\rvert
		&\lesssim 
	\sum_{k=1}^\infty e^{-(r-\rho)\lambda_k} (\lambda_k)^\frac{3}{2} \frac{1}{\rho} \|\Phi(\rho,w;s,y)\|_{L_w^2}  \\
		&\lesssim
	\frac{1}{\rho}e^{-(r-\rho)\lambda_1} \|\Phi\|_{L_w^2} \sum_{k=1}^\infty e^{-(r-\rho)(\lambda_k - \lambda_1)} (\lambda_k)^\frac{3}{2}.
	\end{align*}
Thus, whenever $s \leq \rho \leq r-1$,
	\begin{equation}\label{dotproductmassless}
	\left\lvert\int_\R \sum_{k=1}^\infty J_k(\rho,w;s,y)\; dw \right\rvert
		\lesssim
	\frac{1}{\rho}e^{-(r-\rho)\lambda_1} \|\Phi(\rho,w;s,y)\|_{L_w^2}.
	\end{equation}
The asymptotics for $I_1$ and $I_2$ follow quickly from \eqref{dotproductmassless}.
	
Indeed, first applying \eqref{dotproductmassless} and then applying \eqref{gaussianmasslesstransform},
	\[
	|I_1(r;s,y)|
		\lesssim 
	\int_{s}^{s+1} \left(\frac{\rho}{r}\right)^\frac{1}{2} \frac{1}{\rho} e^{-(r-\rho)\lambda_1}  \frac{s}{\rho^\frac{1}{2} (\rho^3 - s^3)^\frac{1}{4}} \; d\rho
		\lesssim 
	\frac{1}{(rs)^\frac{1}{2}} e^{-(r-s)\lambda_1}.
	\]
For fixed $s$, it follows that 
	\[
	\lim_{r \to \infty} \left(\frac{r}{s}\right)^\frac{1}{2} e^{(r-s)\lambda_0} I_1(r;s,y) = 0,
	\]
with uniform convergence in $y$ (recall, $\lambda_1 > \lambda_0$). Applying \eqref{dotproductmassless} and Lemma \ref{decaymassless},
	\begin{align*}
	|I_2(r;s,y)| 
		&\lesssim 
	\int_{s+1}^{r-1}  \left(\frac{\rho}{r}\right)^\frac{1}{2}\frac{1}{\rho} e^{-(r-\rho)\lambda_1}  \left(\frac{s}{\rho}\right)^\frac{1}{2} e^{-(\rho - s)\lambda_0} \; d\rho 
		\\
		&\lesssim
	\left(\frac{s}{r}\right)^\frac{1}{2} e^{-\lambda_0(r-s)} \left[ \int_{s+1}^\frac{r}{2} e^{-(r-\rho)(\lambda_1 - \lambda_0)} \frac{d\rho}{\rho}  + \int_\frac{r}{2}^{r-1} e^{-(r-\rho)(\lambda_1 - \lambda_0)} \frac{d\rho}{\rho} \right]
		\\
		&\lesssim 
	\left(\frac{s}{r}\right)^\frac{1}{2} e^{-(r-s)\lambda_0} \left[ \frac{e^{-\frac{r}{2}(\lambda_1 - \lambda_0)}}{s+1} + \frac{2}{r}\right].
	\end{align*}
Again, for fixed $s >0$, 
	\[
	\lim_{r\to\infty} \left(\frac{r}{s}\right)^\frac{1}{2} e^{(r-s)\lambda_0} I_2(r;s,y) = 0
	\]
with uniform convergence in $y$. 

Using the fact that $\frac{1}{2} + w\partial_w$ is anti-self-adjoint and applying Cauchy--Schwarz,
	\begin{align*}
	&|I_3(r;s,y)|
		\\
		&\leq
	\int_{r-1}^r   \left(\frac{\rho}{r}\right)^\frac{1}{2}\frac{1}{\rho} \bigg\|\left(\tfrac{1}{2}+ w\partial_w\right) \sum_{k=1}^\infty e^{-(r-\rho)\lambda_k}\psi_k(0)\psi_k(w) \bigg\|_{L_w^2} \|\Phi(\rho,w;s,y)\|_{L_w^2} \; d\rho
	\end{align*}
For $r-1 \leq \rho< r$, applying \eqref{quarticshortterm} and Lemma \ref{quarticderivativeshortterm} gives
	\begin{align}
	\label{quarticfar}
		&\bigg\| \left[\tfrac{1}{2\rho}+w\partial_w\right] \sum_{k=1}^\infty e^{-(r-\rho)\lambda_k}\psi_k(0)\psi_k(w) \bigg\|_{L_w^2}
		\\
		\nonumber
		&\hbox{\hskip 8pt} \lesssim
	\big\|\left[\tfrac{1}{2}+ w\partial_w\right] e^{-(r-\rho)H}(0,w) \big\|_{L_w^2} 
		+ \big\|\left[\tfrac{1}{2}+ w\partial_w\right] e^{-(r-\rho)\lambda_0}\psi_0(0)\psi_0(w)\big\|_{L_w^2}
		\\
		\nonumber
		&\hbox{\hskip 8pt}
		\lesssim
	(r-\rho)^{-\frac{3}{4}}
	\end{align}
Combining \eqref{quarticfar} and Lemma~\ref{decaymassless},
	\begin{align*}
	|I_3(r;s,y)|
		&\lesssim
	\int_{r-1}^r \left(\frac{\rho}{r}\right)^\frac{1}{2} \frac{1}{(r-\rho)^\frac{3}{4}} \frac{1}{\rho} \left(\frac{s}{\rho}\right)^\frac{1}{2} e^{-(\rho-s)\lambda_0} \; d\rho \\
		&\lesssim
	\frac{1}{r} \left(\frac{s}{r}\right)^\frac{1}{2} e^{-(r-s)\lambda_0}.
	\end{align*}
It follows that, for fixed $s$, 
	\[
	\lim_{r \to \infty} \left(\frac{r}{s}\right)^\frac{1}{2} e^{(r-s)\lambda_0} I_3(r;s,y)
		= 
	0
	\]
with uniform convergence in $y$. 

At this point, we have established \eqref{asympt}, with uniform convergence in $y$. Our next goal is to establish the positivity results. In particular, we shall show that the integral term in
	\[
	G(s,y) := \psi_0(y) + \int_s^\infty \left(\frac{\rho}{s}\right)^\frac{1}{2} e^{-(s-\rho)\lambda_0} \psi_0(w) \left[\frac{1}{2\rho}+ \frac{w}{\rho}\partial_w\right] \Phi(\rho,w;s,y) \; dw \; d\tau
	\]
converges, uniformly in $y$, to $0$ as $s$ goes to $\infty$ (recall that $\psi_0(y) > 0$).

Applying Cauchy--Schwarz and \eqref{gaussianmasslesstransform},
	\begin{align}
	\label{nearasymptotic}
	&\left\lvert\int_s^{s+1} \int_\R \left(\frac{\rho}{s}\right)^\frac{1}{2} e^{-(s-\rho)\lambda_0} \psi_0(w) \left[ \frac{1}{2\rho} + \frac{w}{\rho}\partial_w \right] \Phi(\rho,w;s,y) \; dw \; d\rho\right\rvert
		\\
		\nonumber
		&\hbox{\hskip 18pt}\lesssim 
	\frac{\|(\tfrac{1}{2} + w\partial_w)\psi_0(w)\|_2}{s}  \int_s^{s+1}\|\Phi(\rho,w;s,y)\|_{L_w^2} \; d\rho
		\\
		\nonumber
		&\hbox{\hskip 18pt}\lesssim
	\frac{1}{s} \int_s^{s+1} \frac{s}{\rho^\frac{1}{2}(\rho^3-s^3)^\frac{1}{4}} \; d\rho 
		\\
		\nonumber
		&\hbox{\hskip 18pt}\lesssim
	\frac{1}{s}.
	\end{align}
For the integral over $(s+1,\infty)$, first recall that $\big[\big(\tfrac{1}{2} +w\partial_w\big)\psi_0\big](w)$ is perpendicular to $\psi_0(w)$. Integrating by parts, applying Cauchy--Schwarz, and applying Lemma~\ref{decaymassless},
	\begin{align}
	\label{farasymptotic}
	&\left\lvert\int_{s+1}^{\infty} \int_\R \left(\frac{\rho}{s}\right)^\frac{1}{2} e^{-(s-\rho)\lambda_0} \psi_0(w) \left[ \frac{1}{2\rho} + \frac{w}{\rho}\partial_w \right] \Phi(\rho,w;s,y) \; dw \; d\rho\right\rvert
		\\ \nonumber
	&\hbox{\hskip 18pt} = 
		\left\lvert\int_{s+1}^{\infty} \int_\R \left(\frac{1}{s\rho}\right)^\frac{1}{2} e^{-(s-\rho)\lambda_0} \left[ \tfrac{1}{2} + w\partial_w \right]\psi_0(w)  \Phi(\rho,w;s,y) \; dw d\rho\right\rvert
		\\ \nonumber
	&\hbox{\hskip 18pt} \lesssim
		\int_{s+1}^{\infty} \int_\R \left(\frac{1}{s\rho}\right)^\frac{1}{2} e^{-(s-\rho)\lambda_0} \|[P_0^\perp\Phi](\rho,w;s,y)\|_{L_w^2}  \; d\rho
		\\ \nonumber
	&\hbox{\hskip 18pt} \lesssim
		\int_{s+1}^{\infty} \int_\R \left(\frac{1}{s\rho}\right)^\frac{1}{2} e^{\lambda_0(\rho - s)} \left(\frac{s}{\rho}\right)^\frac{1}{2} e^{-\lambda_1(\rho - s)}  \; d\rho
		\\ \nonumber
	&\hbox{\hskip 18pt} \lesssim
		\frac{1}{s} \int_{s+1}^\infty e^{-(\lambda_1 - \lambda_0)(\rho - s)} d\rho
		\\ \nonumber
	&\hbox{\hskip 18pt} \lesssim
		\frac{1}{s}.
	\end{align}

Since $\psi_0$ is strictly positive and continuous, we have $\inf_{|y| < 1} \psi_0(y) > 0$. In view of \eqref{nearasymptotic} and \eqref{farasymptotic}, there exists some $M$ such that
	$
	G(s,y) > 0 
	$ 
for all $(s,y) \in [M,\infty) \times (-1,1)$. The same estimates also show that $|G(s,y)| \lesssim \psi_0(y) + s^{-\frac{1}{2}}$, which finishes the proof of Proposition \ref{asymptoticsmassless}.
\end{proof}

\subsection{Proof of Theorem \ref{construct}, part 4.}\label{holder}

Our main tool will be the Kolmogorov Continuity and Consistency theorem (cf., Theorem \ref{kolmogorov}), which allows us to upgrade a consistent family of finite dimensional distributions to a (cylinder) measure on path space. 

We construct our consistent family of measures. For $0 < r_1 < r_2 < \cdots < r_N$ and for Borel sets $B_0, B_1, \ldots, B_N \subseteq \R$, let
	\begin{align*}
	&P_{r_1,\ldots, r_N}(B_1 \times \cdots \times B_N) := 
	\\
	&\hbox{\hskip 8pt}\int_{B_1} \cdots \int_{B_N} F(r_N,x_N) \prod_{j=2}^{N}\phi(r_j,x_j;r_{j-1},x_{j-1}) \phi(r_1,x_1;0,0) \; dx_N \cdots dx_1.
	\end{align*}
and let
	\[
	P_{0,r_1,\ldots, r_N}(B_0 \times B_1 \times \cdots \times B_N) := \delta_0(B_0) P_{r_1,\ldots, r_N}(B_1 \times \cdots \times B_N)
	\]
where $\delta_0(B) = 1$ if $0 \in B$ and $\delta_0(B) = 0$ otherwise.

The consistency of this family follows from the semi-group property of $\phi$, given in \eqref{semigroupphi} below. First, recall that $\phi_n$ is the fundamental solution of the cut-off parabolic PDEs in \eqref{cutoff}. Fix $\rho > s$. For $r > \rho$, both $\int_\R \phi_n(r, x; \rho,w) \phi_n(\rho,w;s,y) \; dx$ and $\phi_n(r,x;s,y)$ solve the Cauchy problem $L_n u(r,x) = 0$ with initial data $u(\rho,w) = \phi_n(\rho,w;s,y)$. By uniqueness of bounded solutions of such parabolic PDEs (cf., \cite[Section 1.9]{Fri}), we must have
	\[
	\int_\R \phi_n(r, x; \rho,w) \phi_n(\rho,w;s,y) \; dx  = \phi_n(r,x;s,y).
	\]
Applying \eqref{limit}, \eqref{gaussbound}, and dominated convergence gives
	\begin{equation}\label{semigroupphi}
	\int_\R \phi(r, x; \rho,w) \phi(\rho,w;s,y) \; dx  = \phi(r,x;s,y).
	\end{equation}
Having this semi-group property, then, for all $r > r_N$,
	\begin{align*}
	\int_\R F(r,x) \phi(r,x;r_N,x_N) \; dx 
		&= 
	\lim_{L \to \infty} \int_\R \frac{\phi(L,0;r,x)}{\phi(L,0;0,0)} \phi(r,x;r_N,x_N) \; dx 
		\\
		&= 
	\lim_{L \to \infty} \frac{\phi(L,0;r_N,x_N)}{\phi(L,0;0,0)} 
		\\		
		&=
	F(r_N,x_N).
	\end{align*}
By Theorem \ref{kolmogorov}, there exists a unique cylinder measure, which we denote $\nu_{\infty,1}$, on $R^{[0,\infty)}$ with the desired finite dimensional distributions.

We now show that $\nu_{L,1}$ is supported on the space of continuous functions and, in particular, on locally $s$-H\"older continuous functions, with $s < \frac{1}{2}$. Fix $R > 1$ and let $p > 2$. For $0 < s < r < R$,  
	\begin{align*}
	&\PP_{\nu_{\infty,1}}(|f(r) - f(s)| > \lambda)\\
		&\hbox{\hskip 36pt}=
	\int_\R \int_{|x-y|> \lambda} F(r,x)\phi(r,x;s,y)\phi(s,y;0,0) \; dx \; dy
		\\
		&\hbox{\hskip 36pt}=
	\int_\R \int_{|x-y|> \lambda} \int_\R F(R,w) \phi(R,w;r,x)\phi(r,x;s,y)\phi(s,y;0,0) \; dw \; dx \; dy
	\end{align*}
Because $F(R,w) \lesssim_R 1$ and $\phi(R,w;r,x) \lesssim \frac{1}{\sqrt{R-r}} \exp\left(-\frac{(w-x)^2}{2(R-r)} \right)$,
therefore
	\[
	\int_\R F(R,w) \phi(R,w;r,x) \; dw \lesssim_R \int_\R \frac{1}{\sqrt{R-r}} \exp\left(-\frac{(w-x)^2}{2(R-r)} \right) \; dw \lesssim_R 1.
	\]
and, by \eqref{gaussbound},
	\[
	\int_{|x - y|> \lambda} \phi(r,x;s,y) \; dx 
		\lesssim 
	\int_\lambda^\infty \frac{1}{\sqrt{r-s}} \exp\left(-\frac{x^2}{2(r-s)} \right)\; dx
		\lesssim
	\exp\left(-\frac{\lambda^2}{2(r-s)}\right)
	\]
and, again by \eqref{gaussbound},
	\[
	\int_\R \phi(s,y;0,0) \; dy \leq 1.
	\]
Putting this all together,
	\[
	\PP_{\nu_{\infty,1}}(|f(r) - f(s)| > \lambda) \lesssim_R \exp\left(-\frac{\lambda^2}{2(r-s)}\right)
	\]
It follows that
	\begin{align*}
	\E^{\nu_{\infty,1}}\left[|f(r) - f(s)|^p\right]
		&=
	\int_0^\infty \lambda^{p-1} \PP_{\nu_{\infty,1}}(|f(r) - f(s)| > \lambda) \; d\lambda\\
		&\lesssim_R
	\int_0^\infty \lambda^{p-1} \exp\left(-\frac{\lambda^2}{2(r-s)}\right) \; d\lambda \lesssim_{p,R} (r-s)^\frac{p}{2}.
	\end{align*}
Thus, for all $s \in \big[0, \tfrac{p-2}{2p}\big)$, Theorem \ref{kolmogorov} shows that the measure $\nu_{L,1}$ gives measure one to $C^{s}_{loc}([0,\infty) \to \R)$. Finally, note that as $p \to \infty$, we have $\frac{p-2}{2p} \uparrow \frac{1}{2}$, and so we may choose any $s < \frac{1}{2}$ as our H\"older exponent. 

\subsection{Proof of Theorem \ref{construct}, part 5.}\label{convergegood}

By definition, the measures $\nu_{\infty,1}|_{[0,R]}$ and $\nu_{L,1}|_{[0,R]}$ are Borel measures on $C^{s}([0,R]\to \R)$ that obey the following laws: for $0 < r_1 < r_2 < \cdots < r_N := R$ and Borel sets $B_1, \ldots, B_N \subseteq \R$, 
	\begin{align*}
	&\PP_{\nu_{L,1}|_{[0,R]}}(f(r_j) \in B_j, j=1,\ldots, N) 
	\\
	&\hbox{\hskip 8pt}= \int_{B_1} \cdots \int_{B_N} \frac{\phi(L,0;R,x_N)}{\phi(L,0;0,0)} \prod_{j=2}^{N}\phi(r_{j},x_{j};r_{j-1},x_{j-1}) \phi(r_1,x_1;0,0) \; dx_N \cdots dx_1.
	\end{align*}
and
	\begin{align*}
	&\PP_{\nu_{\infty,1}|_{[0,R]}}(f(r_j) \in B_j, j=1,\ldots, N)
		\\
	&\hbox{\hskip 8pt}= \int_{B_1} \cdots \int_{B_N} F(R,x_N) \prod_{j=2}^{N}\phi(r_j,x_j;r_{j-1},x_{j-1}) \phi(r_1,x_1;0,0) \; dx_N \cdots dx_1.
		\end{align*}
Let $P$ be the Borel measure on $C^s([0,R])$ given by	
	\begin{equation}\label{radonnikodym}
	dP(f) := \frac{F(R,f(R))\phi(L,0;0,0)}{\phi(L,0;R,f(R))}d\nu_{L,1}|_{[0,R]}(f)
	\end{equation}
Recall that, by Theorem~\ref{construct}, part $1$, we have $\phi>0$. Furthermore, as $L >R$, division by $\phi(L,0;R,f(R))$ is well-defined.  It is not hard to see that
	\[
	P(f(r_j) \in B_j, j=1,\ldots, N) = \PP_{\nu_{\infty,1}|_{[0,R]}}(f(r_j) \in B_j, j=1,\ldots, N).
	\]
It follows that the finite dimensional distributions of $P$ and $\nu_{\infty,1}|_{[0,R]}$ are identical, and by the uniqueness aspect of Theorem \ref{kolmogorov}, we have
	\[
	\nu_{\infty,1}|_{[0,R]} = P.
	\]
As the Radon--Nikodym derivative in \eqref{radonnikodym} is strictly positive everywhere, a similar argument shows that $\nu_{L,1}|_{[0,R]}$ is absolutely continuous with respect to $\nu_{\infty,1}|_{[0,R]}$, with Radon--Nikodym derivative $\frac{\phi(L,0;R,f(R))}{F(R,f(R))\phi(L,0;0,0)}$.

Let $A \subseteq C^{s}([0,R] \to \R)$ be a Borel set. Note that, for each $f \in C^s([0,R] \to \R)$,
	\[
	\left\lvert\chi_A(f) \frac{\phi(L,0;R,f(R))}{F(R,f(R))\phi(L,0;0,0)} \right\rvert
		\leq
	\frac{\phi(L,0;R,f(R))}{F(R,f(R))\phi(L,0;0,0)}
	\]
and
	\[
	\lim_{L \to \infty} \frac{\phi(L,0;R,f(R))}{F(R,f(R))\phi(L,0;0,0)} =1
	\] 	
Also, for each $L > R$,
	\[
	\int_{C^0([0,R])} \frac{\phi(L,0;R,f(R))}{F(R,f(R))\phi(L,0;0,0)} d \nu_{\infty,1}|_{[0,R]} (f)
		= 
	\int_{C^0([0,R])} d\nu_{L,1}|_{[0,R]}(f)
		=
	1
	\]
and so the generalized dominated convergence theorem (cf., \cite[Exer. 2.20]{Fol}) gives
	\begin{align*}
	\lim_{L \to \infty} \nu_{L,1}|_{[0,R]}(A)
		&=
	\lim_{L \to \infty}\int_{C^0([0,R])} \chi_A(f) \frac{\phi(L,0;R,f(R))}{F(R,f(R)\phi(L,0;0,0)} \; d \nu_{L,1}|_{[0,R]} (f)	
		\\
		&=
	\int_{C^0([0,R])} \chi_A(f) \; d \nu_{\infty,1}|_{[0,R]}(f)
		\\
		&=
	\nu_{\infty,1}|_{[0,R]}(A),
	\end{align*}
as desired.

\subsection{Proof of Theorem \ref{construct}, part 6}\label{brownian}

Recall the heat kernel $\phi_0(r,x;s,y) = \frac{1}{\sqrt{2\pi(t-s)}} \exp\left(-\frac{(x-y)^2}{2(t-s)} \right)$ and recall that $W$, the standard Wiener measure, obeys the following law: for $0 < r_1 < r_2 < \cdots < r_N$ and for Borel sets $B_1, B_2, \ldots, B_N \subseteq \R$,
	\begin{align*}
		\PP_{W}(f(r_j) \in B_j, j=1,\ldots, N)= \hbox{\hskip 160pt}
		\\
	\hbox{\hskip 18pt} \int_{B_1} \cdots \int_{B_N} \prod_{j=2}^{N}\phi_0(r_j,x_j;r_{j-1},x_{j-1}) \phi_0(r_1,x_1;0,0) \; dx_n \cdots dx_1.
		\end{align*}
Also, recall that $\nu_{L,2} = \mu_{L,2}$ obeys the same finite dimensional distributions as the Brownian bridge from $0$ to $L$, i.e.,
	\begin{align*}
	&\PP_{\nu_{L,2}}(f(r_j) \in B_j, j=1,\ldots, N) =
		\\
	&\hbox{\hskip 8pt} \int_{B_1} \cdots \int_{B_N} \frac{\phi_0(L,0;r_N,x_N)}{\phi_0(L,0;0,0)} \prod_{j=2}^{N}\phi_0(r_j,x_j;r_{j-1},x_{j-1}) \phi_0(r_1,x_1;0,0) \; dx_n \cdots dx_1.
		\end{align*}

We denote by $W|_{[0,R]}$ the image measure of $W$ on $C^s([0,R] \to \R)$, with the Borel $\sigma$-algebra, under the restriction map $f \mapsto f|_{[0,R]}$ and similarly for $\nu_{L,2}|_{[0,R]}$. By an argument similar to Section \ref{holder}, the measures $W|_{[0,R]}$ and $\nu_{L,2}|_{[0,R]}$ are mutually absolutely continuous, with the Radon--Nikodym derivative
	\[
	\frac{d\nu_{L,2}|_{[0,R]}}{dW|_{[0,R]}}(f) = \frac{\phi_0(L,0;R,f(R))}{\phi_0(L,0;0,0)}.
	\]
Mutual absolute continuity of $\nu_\infty|_{[0,R]}$	and $\nu_L|_{[0,R]}$ follow from the tensor product structure of $\nu_\infty$ and the mutual absolute continuity of each of its components.
	
As $\lim_{L \to \infty} \frac{\phi_0(L,0;R,f(R))}{\phi_0(L,0;0,0)} = 1$ for every $f \in C^s([0,R]\to\infty)$, a similar argument as in Section \ref{convergegood} shows that for every Borel set $A \subseteq C^s([0,R] \to \R)$, we have
	\begin{equation}\label{convergence2}
	\lim_{L \to \infty} \nu_{L,2}|_{[0,R]}(A) = W|_{[0,R]}(A).
	\end{equation}
Furthermore, \eqref{convergenceall} follows from \eqref{convergence1}, \eqref{convergence2}, and the fact that $\nu_\infty$ is essentially a tensor product of $\nu_{\infty,1}$ and $W$.

Finally, as the measures $\nu_L|_{[0,R]}$ and $\nu_\infty|_{[0,R]}$ are mutually absolutely continuous on $C^s([0,R]\to \C)$, the completion of the Borel $\sigma$-algebra with respect to each of these measures must coincide. Let us denote this $\sigma$-algebra by $\mathcal{F}_R$. Also, each set $A \in \mathcal{F}_R$ is the union of a Borel set and a null set (cf., Proposition~\ref{regular}). Thus, \eqref{convergenceall} also holds for each $A \in \mathcal{F}_R$.

\section{Almost sure global existence and invariance}\label{invariance}

In this section, we prove Theorem~\ref{invarianceinf}. Let us recall the first order NLW in \eqref{1DNLW} and the definition of $\flow_\infty$ from \eqref{flowinf}. Furthermore, by Proposition~\ref{holderpolish}, $C^s_0([0,\infty)\to\C)$ is a Polish space, and hence we may utilize results from Appendix~\ref{LS}. As before, let $\rho^L_R : C^s_0([0,L]) \to C^s([0,R])$ and $\rho^\infty_R : C^s_{loc}([0,\infty)) \to C^s([0,R])$ denote the restriction maps $g \mapsto g|_{[0,R]}$.

The first two assertions of Theorem~\ref{invarianceinf} are quite immediate.

\begin{prop}
There exists a Borel subset $\Omega_\infty \subseteq C^s_{loc}([0,\infty) \to \C)$ such that 
	\begin{enumerate}
	\item $\nu_\infty(\Omega_\infty) = 1$;
	\item For every $g \in \Omega_\infty$, $\flow_\infty(t,g)$ is defined globally in time. For each $T>0$ and $R>0$, we have $\flow_\infty(t,g)|_{[0,R]} \in C_t^0C^s_r([-T,T] \times [0,R] \to \R)$.
	\end{enumerate}
\end{prop}
\begin{proof}
Let $\Omega_L$ as in Theorem \ref{invariancefinite} and let
	\begin{align*}
	\tilde{\Omega}_\infty 
		&:=  
	\bigcap_{L = 2}^\infty (\rho^\infty_{\lfloor L /2 \rfloor})^{-1} \circ \rho^{L}_{\lfloor L /2 \rfloor} (\Omega_{L})
		\\
		&=
	\{ g \in C^s_{loc}([0,\infty)) \mid \forall L \geq 2, \exists g_L \in \Omega_L \hbox{ s.t. } g_L|_{[0,\lfloor L/2 \rfloor]} \equiv g|_{[0,\lfloor L/2 \rfloor]} \}.
	\end{align*}
Here, $\lfloor L /2 \rfloor$ denotes the largest integer less than or equal to $L/2$. As restriction is a continuous map, therefore $\rho_{\lfloor L /2 \rfloor}^L(\Omega_L)$ is an analytic set. By Proposition~\ref{analyticprops}, it follows that $(\rho^\infty_{\lfloor L /2 \rfloor})^{-1} \circ \rho^{L}_{\lfloor L /2 \rfloor} (\Omega_{L})$ is analytic and, hence, $\tilde{\Omega}_\infty$ is also analytic. By Proposition~\ref{analyticprops}, the set $\tilde{\Omega}_\infty$ is $\nu_\infty$-measurable.

By Theorem \ref{invariancefinite}, we have $\nu_{L} (\Omega_{L}) = 1$. As $\Omega_L \subseteq (\rho^{L}_{\lfloor L /2 \rfloor})^{-1} \circ \rho^{L}_{\lfloor L /2 \rfloor}(\Omega_{L})$, it follows that $\nu_{L}|_{[0,\lfloor L /2 \rfloor]} (\rho^{L}_{\lfloor L /2 \rfloor}(\Omega_{L})) = 1$. By mutual absolute continuity of $\nu_L|_{[0,\lfloor L /2 \rfloor]}$ and $\nu_\infty|_{[0,\lfloor L /2 \rfloor}$ (cf, Theorem \ref{infvolmeas}), we have $\nu_\infty|_{[0,\lfloor L /2 \rfloor]} (\rho^{L}_{\lfloor L /2 \rfloor}(\Omega_{L})) = 1$. In other words,
	\[
	\nu_\infty \left( (\rho^\infty_{\lfloor L /2 \rfloor})^{-1} \circ \rho^{L}_{\lfloor L /2 \rfloor} (\Omega_{L}) \right) = 1.
	\]
Thus, $\nu_\infty(\tilde{\Omega}_\infty) = 1$. In particular, $C^s_{loc}([0,\infty)) \setminus \tilde{\Omega}_\infty$ is measure $0$, and so there is some Borel set $A$ of measure $0$ that contains $C^s_{loc}([0,\infty)) \setminus \tilde{\Omega}_\infty$. We define 
	\[
	\Omega_\infty = C^s_{loc}([0,\infty)) \setminus A.
	\]
Observe that $\Omega_\infty \subseteq \tilde{\Omega}_\infty$ and $\nu_\infty(\Omega_\infty) = 1$.

We now prove the second assertion. Fix $g \in \Omega_\infty$ and fix $T > 0$. For each $L \geq 2$, let $g_L \in \Omega_{L}$ such that $g|_{[0,\lfloor L /2 \rfloor]} \equiv g_L|_{[0,\lfloor L /2 \rfloor]}$. By finite speed of propagation, we have
	\[
	\flow_L(t,g_L)|_{[0,\lfloor L/2\rfloor-t]} = \flow_{L+k}(t,g_{L+k})|_{[0,\lfloor L/2\rfloor-t]}
	\]
for all $L > 2T$, all $t \in [-T,T]$, and all $k \geq 0$. We define $\flow_\infty(t,g)$ to be the unique function such that for each $R > 0$,
	\begin{equation}\label{flowR}
	\flow_\infty(t,g)|_{[0,R]} \equiv \flow_{L}(t,g_L)|_{[0,R]} \hbox{ for all } L > 2(R + T), |t| \leq T
	\end{equation}
and note that it obeys the regularity conditions asserted above.
\end{proof}

We next turn to the invariance assertions of Theorem~\ref{invarianceinf}. To do this, we first show invariance on the fixed time interval $[-1,1]$, and then iterate the flow map to achieve global invariance.

\begin{lem}\label{flownice}
For each $t \in [-1,1]$, the map
	\[
	\flow(t,\cdot) : \Omega_\infty \to C^s_{loc}([0,\infty) \to \C)
	\]
is continuous with respect to \eqref{metric}. Furthermore, for each Borel subset $A \subseteq \Omega_\infty$, the set $\flow_\infty(t,A) = \{\flow_\infty(t,g) \mid g \in A\}$ is a Borel subset of $C^s_{loc}([0,\infty))$.
\end{lem}
\begin{proof}
For $1 \leq k \leq \infty$, let $g_k \in \Omega_\infty$ such that $d(g_k,g_\infty) \to 0$. Observe that convergence in this metric is equivalent to convergence in each semi-norm.

Fix $n \in \N$. For each $1 \leq k \leq \infty$, choose $\tilde{g}_k \in \Omega_{2n+2}$ such that 
	\[
	\tilde{g}_k|_{[0,n+1]} \equiv g_k|_{[0,n+1]}.
	\]
In particular, $\lim_{n \to \infty}\|\tilde{g}_k - \tilde{g}_\infty\|_{C^s([0,n+1])} = 0$. By Corollary~\ref{stability}, we have
	\[
	\lim_{k \to \infty} \|\flow_{2n+2}(t,\tilde{g}_k) - \flow_{2n+2}(t,\tilde{g}_\infty)\|_{C^0_tC^s_r([-1,1]\times[0,n])} = 0.
	\]
By finite speed of propagation,
	\begin{equation}\label{spatialconvergence}
	\lim_{k \to \infty} \|\flow_{\infty}(t,g_k) - \flow_{\infty}(t,g_\infty)\|_{C^0_tC^s_r([-1,1]\times[0,n])} = 0.
	\end{equation}
As \eqref{spatialconvergence} holds for each $n$, it follows by the above observation that
	\[
	\lim_{k \to \infty} d(\flow_{\infty}(t,g_k), \flow_{\infty}(t,g_\infty)) = 0
	\]
for each $t \in [-1,1]$, proving the continuity assertion.

For fixed $t \in [-1,1]$, we extend $\flow_\infty(t,\cdot)$ to $C^s_{loc}([0,\infty) \to \C)$ by
	\[
	\flow_\infty(t,g) := 0
	\]
for $g \in C^s_{loc}([0,\infty) \to \C) \setminus \Omega_\infty$. This defines a Borel measurable map from $C^s_{loc}([0,\infty) \to \C)$ to itself. Furthermore, the restriction of $\flow_\infty(t,\cdot)$ to $\Omega_\infty$ is an injective map, as the flow is reversible. By the Lusin--Souslin Theorem (cf., Theorem~\ref{lusinsouslin}), $\flow_\infty(t,\cdot)$ maps Borel subsets of $\Omega_\infty$ to Borel subsets of $C^s_{loc}([0,\infty) \to \C)$. 
\end{proof}

\begin{lem}\label{closedgood}
Let $K \subseteq C^s_{loc}([0,\infty) \to \C)$ be a closed set such that $K \subseteq \Omega_\infty$. For every $t \in [-1,1]$, we have
	\[
	\nu_\infty(K) \leq \nu_\infty(\flow_\infty(t,K)).
	\]
\end{lem}
\begin{proof}
Let $0 < R \ll L$, and recall that $(\rho_{R+1}^L)^{-1} \circ \rho_{R+1}^\infty(K)$ is an analytic subset of $C^s([0,L] \to \C)$. By Theorem~\ref{invariancefinite}, we have
	\begin{align}
		\label{finitething}
	\nu_L((\rho_{R+1}^L)^{-1} \circ \rho_{R+1}^\infty(K)) 
		&= 
	\nu_L\left(\Omega_L \cap \left((\rho_{R+1}^L)^{-1} \circ \rho_{R+1}^\infty(K) \right)\right)
		\\
		\nonumber
		&=
	\nu_L\left(\flow_L\left[t, \Omega_L \cap (\rho_{R+1}^L)^{-1} \circ \rho_{R+1}^\infty(K) \right] \right)
	\end{align}
for each $t \in \R$.

Next, we claim that, for $t \in [-1,1]$,
	\begin{equation}\label{fsop}
	\rho^L_{R} \left(\flow_L\left[t, \Omega_L \cap \left((\rho_{R+1}^L)^{-1} \circ \rho_{R+1}^\infty(K) \right) \right] \right) \subseteq  \rho^\infty_R \left( \flow_\infty(t,K) \right)
	\end{equation}
Indeed, let $g \in \flow_L\left[t, \Omega_L \cap (\rho_{R+1}^L)^{-1} \circ \rho_{R+1}^\infty(K) \right]$. Then there exists some $f \in (\rho_{R+1}^L)^{-1} \circ \rho_{R+1}^\infty(K) \cap \Omega_L$ such that $\flow_L(t,f) = g$. Furthermore, there exists $\tilde{f} \in K$ such that $\tilde{f}|_{[0,R+1]} \equiv f|_{[0,R+1]}$ and thus, by finite speed of propagation,
	\[
	\flow_\infty(t,\tilde{f})|_{[0,R]} \equiv \flow_L(t,f)|_{[0,R]} \equiv g|_{[0,R]}.
	\]
The claim follows. 

Given \eqref{fsop}, we also have
	\begin{align}
	\label{contains}
	&\flow_L\left[t, \Omega_L \cap \left((\rho_{R+1}^L)^{-1} \circ \rho_{R+1}^\infty(K) \right) \right]
		\\
		\nonumber
		&\hbox{\hskip 18pt}\subseteq
	(\rho_R^L)^{-1} \circ \rho^L_{R} \left(\flow_L\left[t, \Omega_L \cap \left((\rho_{R+1}^L)^{-1} \circ \rho_{R+1}^\infty(K) \right) \right] \right) 
		\\
		\nonumber
		&\hbox{\hskip 18pt}\subseteq 
	(\rho_R^L)^{-1} \circ \rho^\infty_R \left( \flow_\infty(t,K) \right)
	\end{align}
Furthermore, $\flow_\infty(t,K)$ is a Borel subset of $C^s_{loc}([0,\infty) \to \C)$ by Lemma~\ref{flownice}, and thus $(\rho_R^L)^{-1} \circ \rho^\infty_R \left( \flow_\infty(t,K) \right)$ is analytic (and $\nu_L$-measurable) by Proposition~\ref{analyticprops}.

Putting \eqref{contains} into \eqref{finitething}, then
	\begin{align*}
	\nu_L((\rho_{R+1}^L)^{-1} \circ \rho_{R+1}^\infty(K))
		\leq
	\nu_L\left( (\rho_R^L)^{-1} \circ \rho^\infty_R \left( \flow_\infty(t,K) \right) \right),
	\end{align*}
which is to say,
	\[
	\nu_L|_{[0,R+1]} \left( \rho_{R+1}^\infty(K) \right)
		\leq
	\nu_L|_{[0,R]}\left( \rho_R^\infty(\flow_\infty(t,K)) \right).
	\]
Sending $L \to \infty$, Theorem~\ref{infvolmeas} gives
	\[
	\nu_\infty|_{[0,R+1]} \left( \rho_{R+1}^\infty(K) \right)
		\leq
	\nu_\infty|_{[0,R]}\left( \rho_R^\infty(\flow_\infty(t,K)) \right),
	\]
or
	\begin{equation}\label{restrict}
	\nu_\infty \left( (\rho_{R+1}^\infty)^{-1} \circ \rho_{R+1}^\infty(K) \right)
		\leq
	\nu_\infty\left( (\rho_{R}^\infty)^{-1} \circ \rho_R^\infty(\flow_\infty(t,K)) \right).
	\end{equation}
Similar to above, $(\rho_{R}^\infty)^{-1} \circ \rho_R^\infty(\flow_\infty(t,K))$ is an analytic subset of $C^s_{loc}([0,\infty))$ and thus it is indeed $\nu_\infty$-measurable.
	
Note that \eqref{restrict} holds for arbitrary $R > 0$. Furthermore, we have
	\[
	(\rho_{R}^\infty)^{-1} \circ \rho_R^\infty(\flow_\infty(t,K))
		\supseteq
	(\rho_{R+1}^\infty)^{-1} \circ \rho_{R+1}^\infty(\flow_\infty(t,K))
	\]
and similarly for $K$. By dominated convergence,
	\begin{equation}\label{bigcaps}
	\nu_\infty \left( \bigcap_{R=2}^\infty (\rho_{R}^\infty)^{-1} \circ \rho_{R}^\infty(K) \right)
		\leq
	\nu_\infty\left( \bigcap_{R=1}^\infty  (\rho_{R}^\infty)^{-1} \circ \rho_R^\infty(\flow_\infty(t,K)) \right).
	\end{equation}
	
We claim that 
	\begin{equation}\label{closedflow}
	\flow_\infty(t,K) = \bigcap_{R=1}^\infty (\rho_R^\infty)^{-1} \circ \rho_R^\infty (\flow_\infty(t,K)).
	\end{equation}
The containment $\subseteq$ is obvious. Now, let $g \in \bigcap_{R=1}^\infty (\rho_R^\infty)^{-1} \circ \rho_R^\infty (\flow_\infty(t,K))$. For each integer $R \geq 1$, there exists $f_R \in K$ such that 
	\begin{equation}\label{defn}
	\flow_\infty(t,f_R)|_{[0,R]} \equiv g|_{[0,R]}
	\end{equation}
By finite speed of propagation, it follows that for each $n \geq 0$ and for each $R \geq 1$,
	\[
	f_R|_{[0,R-1]} \equiv f_{R+n}|_{[0,R-1]}.
	\]
Thus, the sequence $\{f_R\}_{R=1}^\infty$ converges to some $f \in C^s_{loc}([0,\infty) \to \C)$, with
	\[
	f|_{[0,R-1]} \equiv f_R|_{[0,R-1]}
	\]
for each $R \geq 1$. Indeed, as $K$ is closed, we also have $f \in K$. Given \eqref{defn} and finite speed of propagation, we have
	\[
	\flow_\infty(t,f)|_{[0,R-2]} \equiv g|_{[0,R-2]}
	\]
for each $R \geq 2$, and thus $\flow_\infty(t,f) = g$. This proves the reverse containment.

As $K$ is closed, a similar converging sequence argument as above shows
	\begin{equation}\label{closedset}
	K = \bigcap_{R=2}^\infty (\rho_R^\infty)^{-1}\circ \rho_R^\infty(K).
	\end{equation}
Putting \eqref{closedflow} and \eqref{closedset} into \eqref{bigcaps} finishes the proof.
\end{proof}

The above proof actually shows that $\flow_\infty$ preserves closed subsets $K$ of $C^s_{loc}([0,\infty) \to \C)$ which are contained\footnote{Inner regularity of $\nu_\infty$ guarantees existence of such closed sets. Indeed, we may find compact sets $K \subseteq C^s_{loc}([0,\infty) \to \C)$ contained in $\Omega_\infty$ with measure arbitrarily close to $1$.} in $\Omega_\infty$, though we will not use this fact for later results. Finally, we prove the invariance assertion of Theorem~\ref{invarianceinf}.

\begin{prop}
For each $\nu_\infty$-measurable subset $A \subseteq \Omega_\infty$ and for each $t \in \R$, the set $\flow_\infty(t,A)$ is also $\nu_L$-measurable and $\nu_\infty(\flow_\infty(t,A)) = \nu_\infty(A)$.
\end{prop}
\begin{proof}
We first consider the case when $t \in [-1,1]$ and when $A$ is Borel. In view of Theorem~\ref{regular}, let
	\[
	K_1 \subseteq K_2 \subseteq K_3 \subseteq \cdots \subseteq A
	\]
be compact sets such that $\nu_\infty(A \setminus K_n) \leq \frac{1}{n}$ for each $n$. By Lemma~\ref{closedgood}, we have $\nu_\infty(K_n) \leq \nu_\infty(\flow_\infty(t,K_n))$ and so
	\begin{equation}\label{uponeway}
	\nu_\infty(A) 
		=
	\nu_\infty\Big( \bigcup_{n=1}^\infty K_n \Big) 
		\leq 
	\nu_\infty\Big( \flow_\infty\Big[t,\bigcup_{n=1}^\infty K_n\Big] \Big)
		\leq
	\nu_\infty(\flow_\infty(t,A)).
	\end{equation}
A similar argument also shows
	\begin{equation}\label{upotherway}
	\nu_\infty(\Omega_\infty \setminus A) \leq \nu_\infty(\flow_\infty(t,\Omega_\infty \setminus A)).
	\end{equation}
Since 
	\[
	1 
		= 
	\nu_\infty(A) + \nu_\infty(\Omega_\infty\setminus A)
		\leq
	\nu_\infty(\flow_\infty(t,A)) + \nu_\infty(\flow_\infty(t,\Omega_\infty \setminus A)) 
		\leq
	1,
	\]
all of the inequalities in \eqref{uponeway} and \eqref{upotherway} must actually be equalities.

The case for all $t \in \R$ follows from iterating the flow, and intersecting with $\Omega_\infty$ if necessary. The case for all $\nu_\infty$-measurable $A$ follows from a similar argument as in Section~\ref{prooffinite}.
\end{proof}

\appendix

\section{Some Descriptive Set Theory}\label{LS}

In this section, we recall some defintions and facts from descriptive set theory that will be useful for our paper. In particular, we use these results in Section~\ref{prooffinite} and Section~\ref{invariance}. We refer to \cite{Kech} for proofs. 

\begin{defn}
A topological space is said to be a \textit{Polish space} if it is completely metrizable, and separable with respect to this metric.
\end{defn}

One of the more useful results in this setting is a theorem by Lusin and Souslin, which states that injective Borel maps (in particular, continuous embeddings) between Polish spaces are Borel isomorphisms onto their image.

\begin{thm}[Lusin--Souslin]\label{lusinsouslin}
Let $X,Y$ be Polish spaces, and let $f : X \to Y$ be Borel measurable. If $A \subseteq X$ is Borel and $f|_A$ is injective, then $f(A) \subseteq Y$ is Borel. 
\end{thm}

Indeed, let us deduce the following result as a corollary:

\begin{prop}\label{holdersobolev}
Fix $L \in (0,\infty)$ and $s \in [0,\infty)$. For each of the Banach spaces 
	\begin{align*}
	X &=L^p([0,L]), \hbox{ with } 1 \leq p < \infty, \hbox{ or} \\ 
	X &= \dot{H}^s_0([0,L]), \hbox{ or} \\ 
	X &= C^{0}([0,L]),
	\end{align*}
the following statement holds: 
	\begin{quote}
	Endow $X$ with its usual norm topology, as well as the corresponding Borel $\sigma$-algebra, $\mathcal{B}_X$. Then $C^s_0([0,L]) \in \mathcal{B}_X$. Furthermore, the restriction of $\mathcal{B}_X$ to $C^s_0([0,L])$ coincides with the standard Borel $\sigma$-algebra induced by the $s$-H\"older norm.
	\end{quote}
\end{prop}
\begin{proof}
Recall that all of the Banach spaces mentioned above are separable, and hence are Polish spaces. Furthermore, $C^s_0([0,L])$ embeds continuously into each Banach space $X$ above. By Theorem~\ref{lusinsouslin}, it follows that $C^s_0([0,L]) \in \mathcal{B}_X$ for each $X$ above, and that the $\sigma$-algebras mentioned above must coincide.  
\end{proof}

In this paper, we also consider general continuous images of Borel sets, such as restriction maps. These sets are not necessarily Borel, but still obey nice measure theoretic properties.

\begin{defn}\label{analytic}
Let $X$ be a Polish space. 
	\begin{enumerate}
	\item A set $A \subseteq X$ is called \textit{analytic} if there is a Polish space $Y$, a Borel subset $B \subseteq Y$, and a continuous function $f : Y \to X$ such that $f(B) = A$.
	\item Let $\mu$ be a $\sigma$-finite Borel measure on $X$, and let $\mathcal{F}$ be the completion of the Borel $\sigma$-algebra with respect to $\mu$. A set $A \subseteq X$ is called \textit{$\mu$-measurable} if $A \in \mathcal{F}$.
	\item Finally, a set $A \subseteq X$ is called \textit{universally measurable} if $A$ is $\mu$-measurable for every $\sigma$-finite Borel measure $\mu$ on $X$.
	\end{enumerate}
\end{defn}

\begin{prop}[Properties of analytic sets]\label{analyticprops}
Let $X$ be a Polish space.
	\begin{enumerate}
	\item If $A_n \subseteq X$ is analytic for $n \in \N$, then $\bigcap_{n \in \N} A_n$ and $\bigcup_{n \in \N} A_n$ are also analytic.
	\item Let $Y$ be another Polish space and let $f : X \to Y$ be Borel measurable. If $A\subseteq X$ is analytic, then $f(A) \subseteq Y$ is analytic. If $B \subseteq Y$ is analytic, then $f^{-1}(B) \subseteq X$ is analytic.
	\item (Lusin) Analytic sets are universally measurable.
	\end{enumerate}
\end{prop}

Finally, we recall a generalization of regularity results for Lebesgue measure to the Polish space setting.

\begin{prop}[Regularity of Borel Measures]\label{regular}
Let $X$ be a Polish space and let $\mu$ be a finite Borel measure on $X$. Then $\mu$ is \textit{regular}: for any $\mu$-measurable set $A \subseteq X$,
	\begin{align*}
	\mu(A)
		&=
	\sup\{\mu(K) \mid K \subseteq A, K \hbox{ compact}\}\\
		&=
	\inf\{\mu(U) \mid U \supseteq A, U \hbox{ open}\}.
	\end{align*}
In particular, a set $A \subseteq X$ is $\mu$-measurable if and only if there exists an $F_\sigma$ set $F \subseteq A$ (resp., $G_\delta$ set $G \supseteq A$) such that $\mu(A \setminus F) = 0$ (resp., $\mu(G \setminus A) = 0$). 
\end{prop}

\section{A Multi-time Feynman--Kac formula}

In this section, we make various modifications to the classical Feynman--Kac formula. To this end, we make use of the fact that fundamental solutions of parabolic PDEs (cf., Definition~\ref{fundsolndefn}) also have nice properties in their secondary variables. We list the relevant results below, and refer to \cite{Fri} for proofs. 

\begin{lem}\label{adjoint}
Let $V(r,x) : \R^{\geq 0} \times \R \to \R$ be bounded and continuous, and let $\phi(r,x;s,y)$ be the fundamental solution of
	\[
	\partial_r \phi = \frac{1}{2} \partial_x^2 \phi + V(r,x) \phi
	\] 
Then, as a function of $(s,y)$, we also have $-\partial_s\phi = \frac{1}{2} \partial_y^2 \phi + V(s,y) \phi$. Furthermore, for $g \in C_0(\R)$ and for fixed $r > 0$, the unique solution of 
	\[\left\{ \begin{array}{l}
	-\partial_s \psi = \frac{1}{2} \partial_y^2 \psi + V(s,y) \psi,
		\hbox{\hskip 18pt}
	(s,y) \in [0,r) \times \R
		\\
	\psi(r,\cdot) = g(\cdot)
	\end{array} \right.\]
of sub-exponential growth in $y$ is given by
	\[
	\psi(s,y) = \int_\R g(x) \phi(r,x;s,y) \; dx.
	\]
\end{lem}

Let $W_{s,y}$ denote the Wiener measure for Brownian motion starting from time $s$ and at point $y$. We use $B$ to denote a generic function in the support of $W_{s,y}$. As before, 
	\[
	\phi_0(r,x;s,y) := \frac{1}{\sqrt{2\pi(r-s)}} \exp \left(-\frac{(x-y)^2}{2(r-s)} \right)
	\] 
is the heat kernel.

\begin{thm}[Multi-time Feynman--Kac, Brownian Motion]
Let $V(r,x) : \R^{\geq 0} \times \R \to \R$ be bounded and continuous, and let $\phi(r,x;s,y)$ be the fundamental solution of
	\[
	\partial_r \phi = \frac{1}{2} \partial_x^2 \phi + V(r,x) \phi.
	\] 
Let $L \in (0,\infty)$, $s \in [0,L)$, and $y \in \R$. Let $s < r_1 < \cdots < r_N < L$ and let $f_1, \ldots, f_N,g : \R \to \R$ be bounded, measurable functions. With $(r_0,x_0) := (s,y)$, we then have
	\begin{align}
	\label{expectation}
	&\E^{W_{s,y}} \left[ \exp\left(\int_s^L V(\rho,B(\rho)) \; d\rho\right) g(B(L)) \prod_{j=1}^N f(B(r_j)) \right] 	
		\\
		\nonumber
		&\hbox{\hskip 12pt}
		=
	\int_\R \cdots \int_\R g(x)\phi(L,x;r_N,x_N)\prod_{k=1}^{N} f(x_j)\phi(r_j,x_j;r_{j-1},x_{j-1}) \; dxdx_{N}\cdots dx_1 
	\end{align}
\end{thm}
\begin{proof}
The case $N = 0$ is the usual Feynman--Kac formula, which we shall assume as a base case (see \cite{Var} for a proof). For the inductive step, we shall mimic the technique of this proof. 

Let us write the left hand side of \eqref{expectation} as
	\[
	 \E^{W_{s,y}} \left[ \left( \sum_{n=0}^\infty A_n\right) \exp\left(\int_{r_1}^L V(\rho,B(\rho)) \; d\rho\right) g(B(L)) \prod_{j=1}^N f(B(r_j)) \right],
	\]
where 
	\[
	A_n := \frac{1}{n!} \int_s^{r_1} \cdots \int_s^{r_1} V(\rho_1,B(\rho_1)) \cdots V(\rho_n,B(\rho_n)) \; d\rho_n \cdots d\rho_1
	\] 
for $n \geq 1$ and $A_0 = 1$. Observe that, if we let 
	\[
	\mathcal{R}_n := \{ (\rho_1, \ldots, \rho_n) \in \R^n :  s < \rho_1 < \cdots < \rho_n < r_1 \}
	\] 
and let $d\vec{\rho} := d\rho_n \cdots d\rho_1$, then, for $n \geq 1$,
	\[
	A_n =\int_{\mathcal{R}^n} \prod_{j=1}^n V(\rho_j,B(\rho_j)) \; d\vec{\rho}.
	\]
By hypothesis and by Fubini's theorem, one may interchange expectations, sums, and integrals as desired. In particular,
	\begin{align*}
	LHS\eqref{expectation} = \sum_{n=0}^\infty J_{s,y}(n)
	\end{align*}	
where
	\[
	J_{s,y}(0) := \E^{W_{s,y}} \left[\exp\left(\int_{r_1}^L V(\rho,B(\rho)) \; d\rho\right)g(B(L)) \prod_{j=1}^N f(B(r_j)) \right] 
	\]
and, for $n \geq 1$,
	\begin{align*}
	&J_{s,y}(n):= \int_{\mathcal{R}^n} \E^{W_{s,y}} \Bigg[ \prod_{j=1}^n V(\rho_j,B(\rho_j)) \exp\left(\int_{r_1}^L V(\rho,B(\rho)) \; d\rho\right)  g(B(L)) \times \\
	&\hbox{\hskip 64pt} \prod_{j=1}^N f(B(r_j)) \Bigg] \; d\vec{\rho}.
	\end{align*}

Now, let
	\[
	h(r_1,x_1) := \E^{W_{r_1,x_1}}\left[\exp\left(\int_{r_1}^L V(\rho,B(\rho)) \; d\rho\right)g(B(L)) \prod_{j=2}^N f(B(r_j))\right],
	\]
which, by induction, is equal to
	\begin{equation}\label{inductivestep}
	\int_\R \cdots \int_\R g(x)\phi(L,x;r_N,x_N)\prod_{k=2}^{N} f(x_k)\phi(r_k,x_k;r_{k-1},x_{k-1}) \; dxdx_{N}\cdots dx_2.
	\end{equation}
Applying the Markov property of Brownian motion and \eqref{inductivestep},
	\[
	J_{s,y}(0) = \int_\R h(r_1,x_1) f(x_1) \phi_0(r_1,x_1;s,y)  \; dx_1.
	\]
and, for $n \geq 1$,
	\begin{align*}
	&J_{s,y}(n) = \int_{\mathcal{R}^n} \int_{\R^n} \int_\R h(r_1,x_1) f(x_1) \phi_0(r_1,x_1;\rho_n,w_n)  \\
	&\hbox{\hskip 64pt} \prod_{j=2}^n V(\rho_j,w_j) \phi_0(\rho_j,w_j;\rho_{j-1},w_{j-1})V(\rho_1,w_1) \phi_0(\rho_1,w_1;s,y) dx_1 d\vec{w} d\vec{\rho}
	\end{align*}

For the next step, observe that
	\[
	-\partial_s \phi_0(r,x;s,y) = \frac{1}{2} \partial_y^2 \phi_0(r,x;s,y)
	\]
and so
	\[
	-\partial_s J_{s,y}(0) = \frac{1}{2} \partial_y^2 J_{s,y}(0)
	\]
For $n \geq 1$, differentiating $J_{s,y}(n)$ in $s$ produces two terms: one from the bounds in $\mathcal{R}_n$ and one from differentiating $\phi_0(\rho_1,w_1;s,y)$; namely, for $n \geq 1$, 
	\[
	-\partial_s J_{s,y}(n) = V(s,y) J_{s,y}(n-1) + \frac{1}{2} \partial_y^2 J_{s,y}(n)
	\]
In particular, $\sum_{n=0}^\infty J_{s,y}(n)$ obeys\footnote{Indeed, $\sum_{n} J_{s,y}(n)$ represents the iterated Duhamel expansion for the solution of this PDE.} the backwards Kolmogorov equation
	\[
	-\partial_s \psi = \frac{1}{2} \partial_y^2 \psi + V(s,y) \psi
	\]
with terminal condition $\psi(r_1,\cdot) = h(r_1,\cdot)f(\cdot)$. Given Lemma~\ref{adjoint}, we have
	\[
	LHS\eqref{expectation} = \sum_{n=0}^\infty J_{s,y}(n) = \int_\R h(r_1,x_1) f(x_1) \phi(r_1,x_1;s,y) \; dx_1.
	\]
In view of \eqref{inductivestep}, we are done.
\end{proof}
	
Let $BB_{L,x;s,y}$ denote the measure for the Brownian bridge that starts at time $s$ and at point $y$, with ending at time $L$ and at point $x$. We use $BB$ to denote a generic element in the support of $BB_{L,x;s,y}$. The proof of the above theorem also applies, mutatis mutandis, to the following setting.
	
\begin{thm}[Multi-time Feynman--Kac, Brownian Bridge]\label{multitimebridge}
Let $V(r,x) : \R^{\geq 0} \times \R \to \R$ be bounded and continuous, and let $\phi(r,x;s,y)$ be the fundamental solution of 
	\[
	\partial_r \phi = \frac{1}{2} \partial_x^2 \phi + V(r,x) \phi.
	\] 
Let $L \in (0,\infty)$, $s \in [0,L)$, and $y \in \R$. Let $s < r_1 < \cdots < r_N < L$ and let $f_1, \ldots, f_N : \R \to \R$ be bounded, measurable functions. With $(r_0,x_0) := (s,y)$, we then have
	\begin{align*}
	&\E^{BB_{L,x;s,y}} \left[ \exp\left(\int_s^L V(\rho,BB(\rho)) \; d\rho\right) f_1(BB(r_1))\cdots f_N(BB(r_N)) \right] 	
		\\
		&\hbox{\hskip 12pt}
		=
	\int_\R \cdots \int_\R \frac{\phi(L,x;r_N,x_N)}{\phi_0(L,x;s,y)}\prod_{j=1}^{N} f(x_j)\phi(r_j,x_j;r_{j-1},x_{j-1}) \; dx_{N}\cdots dx_1.
	\end{align*}
\end{thm}

\end{document}